\documentclass[12pt]{article} 
\usepackage{amsmath}
\usepackage{amsfonts,amssymb}
\usepackage{fullpage}
\usepackage{graphicx}
\usepackage[usenames,dvipsnames]{color}
\usepackage{epsfig}
\usepackage{subfig}
\usepackage{amsmath}
\usepackage{amsfonts,amssymb}
\usepackage{amssymb}
\usepackage{latexsym}
\usepackage{url}

\def\dmax{d_{\rm max}}
\def\an{a^{(n)}}
\def\An{A^{(n)}}

\def\Cn{C^{(n)}}
\def\dn{d^{(n)}}
\def\Dn{D^{(n)}}
\def\Fn{F^{(n)}}
\def\Gn{G^{(n)}}
\def\Hn{H^{(n)}}
\def\kn{k^{(n)}}
\def\Mn{M^{(n)}}
\def\pn{p^{(n)}}
\def\qn{q^{(n)}}
\def\Rn{R^{(n)}}
\def\Sn{S^{(n)}}
\def\Tn{T^{(n)}}
\def\Un{U^{(n)}}
\def\vn{v^{(n)}}
\def\Vn{V^{(n)}}
\def\Xn{X^{(n)}}
\def\Yn{Y^{(n)}}
\def\Xne{X^{(n)}_E}
\def\Yne{Y^{(n)}_E}
\def\Zne{Z^{(n)}_E}

\def\qtn{\tilde{q}^{(n)}}
\def\qhatn{\hat{q}^{(n)}}

\def\Hnt{\tilde{H}^{(n)}}
\def\Hnhat{\hat{H}^{(n)}}
\def\Tnt{\tilde{T}^{(n)}}
\def\Tncheck{\check{T}^{(n)}}
\def\Vnbar{\bar{V}^{(n)}}
\def\Xnt{\tilde{X}^{(n)}}
\def\Xncheck{\check{T}^{(n)}}

\def\Ynet{\tilde{Y}^{(n)}_E}
\def\Znet{\tilde{Z}^{(n)}_E}

\def\betat{\tilde{\beta}}
\def\betabar{\bar{\beta}}

\def\cnt{\tilde{c}^{(n)}}
\def\dnt{\tilde{d}^{(n)}}\newcommand{\hfigwidth}{0.45\textwidth}

\newenvironment{proof}[1][Proof]
              {\par \normalfont
              \trivlist
             \item[
               \hspace{12pt}               \itshape #1{.}]\ignorespaces
        }{\hfill$\Box$ \endtrivlist}

\def\pnt{\tilde{p}^{(n)}}

\def\bWn{\boldsymbol{W}^{(n)}}
\def\bXn{\boldsymbol{X}^{(n)}}
\def\bWnt{\tilde{\boldsymbol{W}}^{(n)}}
\def\bWnbar{\bar{\boldsymbol{W}}^{(n)}}
\def\bXnt{\tilde{\boldsymbol{X}}^{(n)}}
\def\bXncheck{\check{\boldsymbol{X}}^{(n)}}
\def\bWnhat{\hat{\boldsymbol{W}}^{(n)}}

\def\epsilonn{\epsilon^{(n)}}
\def\pin{\pi^{(n)}}
\def\taun{\tau^{(n)}}
\def\tautn{\tilde{\tau}^{(n)}}
\def\tauhatn{\hat{\tau}^{(n)}}
\def\taubarn{\bar{\tau}^{(n)}}
\def\mudn{\mu_D^{(n)}}

\def\CnMRbond{C^{(n)}_{\rm MR, B}}
\def\CnMRsite{C^{(n)}_{\rm MR, S}}
\def\CnNSWbond{C^{(n)}_{\rm NSW, B}}
\def\CnNSWsite{C^{(n)}_{\rm NSW, S}}

\def\TnMR{T^{(n)}_{\rm MR}}

\def\TnNSW{T^{(n)}_{\rm NSW}}
\def\sigmaMR{\sigma^2_{\rm MR}}
\def\sigmaNSW{\sigma^2_{\rm NSW}}
\def\sigmatMR{\tilde{\sigma}^2_{\rm MR}}
\def\sigmatNSW{\tilde{\sigma}^2_{\rm NSW}}
\def\SigmaMR{\Sigma^{\rm MR}}
\def\SigmaNSW{\Sigma^{\rm NSW}}
\def\SigmabarMR{\bar{\Sigma}^{\rm MR}}
\def\SigmabarNSW{\bar{\Sigma}^{\rm NSW}}

\def\sigmaMRbond{\sigma^2_{\rm MR, B}}
\def\sigmaMRsite{\sigma^2_{\rm MR, S}}
\def\sigmaNSWbond{\sigma^2_{\rm NSW, B}}
\def\sigmaNSWsite{\sigma^2_{\rm NSW, S}}

\def\Ynehat{\hat{Y}^{(n)}_E}
\def\Znehat{\hat{Z}^{(n)}_E}

\def\Ynecheck{\check{Y}^{(n)}_E}
\def\Znecheck{\check{Z}^{(n)}_E}

\def\EEn{\mathcal{E}^{(n)}}
\def\EEtn{\tilde{\mathcal{E}}^{(n)}}
\def\EEcheckn{\check{\mathcal{E}}^{(n)}}
\def\GGn{\mathcal{G}^{(n)}}
\def\GGnbond{\mathcal{G}^{(n)}_{\rm bond}}
\def\GGnsite{\mathcal{G}^{(n)}_{\rm site}}
\def\GGtn{\tilde{\mathcal{G}}^{(n)}}
\def\IIn{\mathcal{I}^{(n)}}
\def\NNn{\mathcal{N}^{(n)}}
\def\TTn{\mathcal{T}^{(n)}}

\def\BBtn{\tilde{\mathcal{B}}^{(n)}}
\def\BBcheckn{\check{\mathcal{B}}^{(n)}}

\def\BBt{\tilde{\mathcal{B}}}

\newcommand\E{\rm E}
\def\P{\rm P}

\newcommand\var{\rm var}

\newcommand\re{{\rm e}}

\def\mudt{\mu_{\tilde{D}-1}}
\def\mundt{\mu^{(n)}_{\tilde{D}-1}}

\def\B{\mathcal{B}}

\newcommand\bc{\boldsymbol{c}}
\newcommand\bcbar{\bar{\boldsymbol{c}}}

\newcommand\bD{\boldsymbol{D}}
\newcommand\bn{\boldsymbol{n}}
\newcommand\bnhat{\hat{\boldsymbol{n}}}

\newcommand\ber{\boldsymbol{e}^{\rm R}}
\newcommand\bes{\boldsymbol{e}^{\rm S}}
\newcommand\bei{\boldsymbol{e}^{\rm I}}
\newcommand\bev{\boldsymbol{e}^{\rm V}}
\newcommand\bi{\boldsymbol{i}}
\newcommand\bl{\boldsymbol{l}}
\newcommand\blhat{\hat{\boldsymbol{l}}}
\newcommand\bp{\boldsymbol{p}}
\newcommand\bV{\boldsymbol{V}}
\newcommand\bw{\boldsymbol{w}}
\newcommand\bx{\boldsymbol{x}}
\newcommand\by{\boldsymbol{y}}

\newcommand\bzero{\boldsymbol{0}}
\newcommand\bone{\boldsymbol{1}}

\newcommand\bwt{\tilde{\boldsymbol{w}}}
\newcommand\bwhat{\hat{\boldsymbol{w}}}
\newcommand\bwbar{\bar{\boldsymbol{w}}}
\newcommand\etat{\tilde{\eta}_E}
\newcommand\taut{\tilde{\tau}}
\newcommand\Bbar{\bar{B}}
\newcommand\ct{\tilde{c}}
\newcommand\dt{\tilde{d}}
\newcommand\gt{\tilde{g}}
\newcommand\pt{\tilde{p}}
\newcommand\xt{\tilde{x}}
\newcommand\xbar{\bar{x}}
\newcommand\vt{\tilde{v}}
\newcommand\vbar{\bar{v}}
\newcommand\xet{\tilde{x}_E}

\newcommand\yet{\tilde{y}_E}
\newcommand\yebar{\bar{y}_E}
\newcommand\zet{\tilde{z}_E}
\newcommand\zebar{\bar{z}_E}
\newcommand\Phit{\tilde{\Phi}}
\newcommand\Phibar{\bar{\Phi}}
\newcommand\Sigmat{\tilde{\Sigma}}
\newcommand\Sigmabar{\bar{\Sigma}}

\newcommand\mud{\mu_D}

\newcommand\fde{f_{D_{\epsilon}}}
\newcommand{\convas}{\stackrel{a.s.}{\longrightarrow}}
\newcommand{\convD}{\stackrel{D}{\longrightarrow}}
\newcommand{\convp}{\stackrel{p}{\longrightarrow}}

\newcommand{\eqD}{\stackrel{D}{=}}

\newtheorem{theorem}{Theorem}[section]
\newtheorem{cor}[theorem]{Corollary}

\newtheorem{remark}[theorem]{Remark}

\numberwithin{equation}{section}

\begin{document}


\title{Central limit theorems for SIR epidemics and percolation on configuration model random graphs}


\author{Frank Ball$^{1}$}
\footnotetext[1]{School of Mathematical Sciences, The University of Nottingham, University Park, Nottingham NG7 2RD, UK}

\date{\today}
\maketitle

\begin{abstract}
We consider a stochastic SIR (susceptible $\to$ infective $\to$ recovered) epidemic defined on
a configuration model random graph, in which infective individuals can infect only their neighbours in the graph during an infectious period which has
an arbitrary but specified distribution.  Central limit theorems for the final size (number of initial susceptibles that
become infected) of such an epidemic as the population size $n$ tends to infinity, with explicit, easy to compute expressions for the asymptotic variance, are proved
assuming that the degrees are bounded.  The results are obtained for both the Molloy-Reed random graph, in which the degrees of individuals are deterministic,
and the Newman-Strogatz-Watts random graph, in which the degrees are independent and identically distributed.  The central limit theorems cover the cases
when the number of initial infectives either (a) tends to infinity or (b) is held fixed as $n \to \infty$.  In (a) it is assumed that the 
fraction of the population that is initially infected converges to a limit (which may be $0$) as $n \to \infty$, while in (b) the central limit theorems
are conditional upon the occurrence of a large outbreak (more precisely one of size at least $\log n$).  Central limit theorems for the size of the
largest cluster in bond percolation on Molloy-Reed and Newman-Strogatz-Watts random graphs follow immediately from our results, as do
central limit theorems for the size of the giant component of those graphs.  Corresponding central limit theorems for site percolation on those graphs
are also proved.
\end{abstract}




\section{Introduction} 
\label{sec:intro}
There has been considerable work in the past two decades on models for the spread of epidemics on random networks; see, for example, 
the recent book~\cite{KMS:2017}.  The usual paradigm is that individuals in a population are represented by nodes in a random graph
and infected individuals are able to transmit infection only to their neighbours in the graph.  The graph is often constructed using 
the configuration model (see, for example, \cite{vdHofstad:2016}, Chapter 7), which allows for an arbitrary but specified degree distribution.  The most-studied type of
epidemic model is the SIR (susceptible $\to$ infective $\to$ recovered) model.  In this model individuals are classified into three
types: susceptibles, infectives and recovered.  If a susceptible individual is contacted by an infective then it too becomes an
infective and remains so for a time, called its infectious period, that is distributed according to a non-negative random variable $I$
having an arbitrary but specified distribution.  An infective individual recovers at the end of its infectious period and is then immune to
further infection.  During its infectious period, an infective contacts its susceptible neighbours in the graph independently at the points of
Poisson processes having rate $\lambda$.  The graph is assumed to be static and the population closed (i.e.~there are no births or deaths),
so eventually the epidemic process terminates.  The final size of the epidemic is the number of initial susceptibles that are infected
during its course.  The final size is a key epidemic statistic, not only as a measure of the impact of an epidemic but also in an inferential
setting, since often it can be observed more reliably than the precise temporal spread.  The main aim of this paper is to develop
central limit theorems for the final size of an SIR epidemic on configuration model graphs as the population size $n \to \infty$.

The configuration model, which was introduced by~\cite{Bollobas:1980}, is a model
for random graphs with a given degree sequence.  There are two distinct approaches for constructing configuration model graphs
with a given degree distribution as $n \to \infty$.  In both approaches, individuals are assigned  a number of half-edges, corresponding to their
degree, and then these half-edges are paired uniformly at random.  In~\cite{MR:1995}, the degrees of indiviudals are prescribed deterministically 
whilst in~\cite{NSW:2001} they are i.i.d.~(independent and identically distributed) copies of a random variable $D$, that describes the 
limiting degree distribution.  We refer to the former as the MR random graph and to the latter as the NSW random graph. Subject to suitable
conditions on the degree sequences in the MR model and $D$ in the NSW model, law of large number limits for SIR epidemics on the two
graphs are the same.  That is not the case for central limit theorems, as for finite $n$, there is greater variability in the degrees of
individuals in the NSW model than in the MR model; indeed, in the NSW model, such variability is of the same order of magnitude as the 
variability in the epidemic.  Thus, though the asymptotic means are the same, the asymptotic variances in the central limit theorems 
for final size are greater for the epidemic on the NSW random graph.

There have been numerous studies, some fully rigorous and some heuristic, of SIR epidemics on configuration model networks making various
assumptions concerning the infectious period random variable $I$.  For example, assuming $I$ is constant,~\cite{Andersson:1998} derives a law of
large numbers for the final size of an epidemic on an NSW random graph when a strictly positive fraction of the population is initially infected in the limit
as $n \to \infty$ and~\cite{BJML:2007} obtain a similar result for epidemics on an MR random graph initiated by a single infective.  In the latter case,
a large outbreak is possible only if the basic reproduction number $R_0>1$; see~\eqref{equ:Rzero} in Section~\ref{sec:mainresults}.  In a highly influential paper,~\cite{Newman:2002}
uses heuristic percolation arguments to obtain a number of results, including the fraction of the population infected by a large outbreak, for SIR epidemics on NSW random graphs with $I$ having an arbitrary but specified distribution.  Several authors have studied the case when $I$ has an exponential distribution,
so the model becomes Markovian.  \cite{DDMT:2012} obtain a law of large numbers type result for the epidemic process on an NSW random graph
with a strictly positive fraction initially infected, which yields a rigorous proof of the deterministic approximation of~\cite{Volz:2008}
(see also~\cite{Miller:2011} and~\cite{MSV:2012}).  \cite{BP:2012} obtain law of large numbers results for both the process and final size of
an epidemic with one intial infective on an MR random graph with bounded degrees.  \cite{JLW:2014} obtain similar results under weaker conditions on the 
degree sequences considering the cases when the fraction initially infected, in the limit as $n\to \infty$, is either strictly positive or zero (assuming 
of course there is at least one initial infective).  In the latter case, the limiting ``deterministic" process involves a random time translation
reflecting the time taken for the number of infectives to reach order $n$; a similar result is obtained by~\cite{BR:2013} assuming a bounded degree
sequence and an arbitrary but specified distribution for $I$.

There has been very little work to date on central limit theorems for SIR epidemics on configuration model networks.  A functional central limit 
theorem for the SI epidemic (in which ${\rm P}(I=\infty)=1$, so infectives remain infectious forever) on an MR random graph with unbounded degrees
is obtained by~\cite{KWRK:2017}, who note that their method is not straightforward to extend to an SIR model.  Assuming that $I$ follows an exponential
distribution and bounded degrees,~\cite{BBLS:2018} use an effective degree approach (\cite{BN:2008}) and density dependent population processes (\cite{EK86}, Chapter 11) to obtain functional central limit theorems for SIR epidemics on MR and NSW random graphs, in which susceptible individuals can also drop their edges to infective neighbours.  They also conjecture central limit theorems for the final size of such epidemics
(and hence as a special case for the final size of standard SIR epidemics, without dropping of edges), assuming
either a strictly positive fraction or a constant number of initial infectives (in which case the central limit theorem is conditional on the occurrence of a large outbreak). However, the arguments are not fully rigorous and the result for a constant
number of initial infectives is based purely on the existence of equivalent results for other (non-network) SIR epidemic models.  Another limitation of~\cite{BBLS:2018}
is the assumption that $I$ is exponentially distributed, which is unrealistic for most real-life diseases.

In the present paper, we address these shortcomings and derive fully rigorous
central limit theorems for the final size of SIR epidemics on MR and NSW random graphs having bounded degrees, when the infectious period $I$ follows an arbitrary but specified distribution.  We consider the cases when the limiting fraction of the population is (i) strictly positive and (ii) zero. For the latter
we treat the situtations where the number of initial infectives either (i) is held fixed independent of $n$ or (ii) tends to $\infty$ as $n \to \infty$. 
The mean parameter $\rho$ in the central limit theorems, which coincides with the corresponding law of large numbers limit, depends on the solution
$z$ of a non-linear equation (see~\eqref{equ:z} and~\eqref{equ:zthm2} in Section~\ref{sec:mainresults}).  Given $z$, the variance parameter in
the central limit theorems is fully explicit and hence easy to compute.  

If $I$ is constant, say ${\rm P}(I=1)=1$ then the above SIR model is essentially bond percolation with probability $\pi=1-\re^{-\lambda}$
and if ${\rm P}(I=\infty)=\pi=1-{\rm P}(I=\infty)=0$ then it is closely related to site percolation with probability $\pi$; see, for example,~\cite{Durrett:2007}, page
15, and~\cite{Janson:2009a}.   Central limit theorems for the size of the giant component (largest cluster) of bond percolation on
the MR and NSW random graphs follow immediately from our results. (Corresponding theorems for site percolation are also obtained using our
methodology.)  Further,
setting $\pi=1$ yields central limit theorems for the giant component of those graphs (Remark~\ref{rmk:giantclt}); cf.~\cite{BR:2017} who obtain a central limit
theorem for the giant component of the MR random graph and~\cite{BN:2017} who derive the asymptotic variance of the giant component of
MR and NSW random graphs, all allowing for unbounded degrees.

The proofs involve constructing the random graph and epidemic on it simultanaeously,
modifying the infection mechanism so that when a susceptible is infected it decides which of its half-edges it will try to infect along with
(its remaining half-edges becoming recovered half-edges), with the times of those infection attempts (relative to the time of infection of the susceptible)
being realisations of i.i.d.~exponential random variables.  The distribution of the final outcome of the epidemic, and hence also its final size, is
invariant to this modification.  The process describing the evolution of the numbers of susceptibles of different degrees, infective half-edges and
recovered half-edges is an asymptotically density dependent population process (\cite{EK86}, Chapter 11, and~\cite{Pollett90}).  The asymptotic distribution
of the final outcome of the epidemic is studied by considering a boundary crossing problem for a random time-scale transformation
of that process.  The proofs extend,  at least in principle, to SIR epidemics and percolation on extensions of the configuration
model that include fully-connected cliques (\cite{Trapman:2007}, \cite{Gleeson:2009}, \cite{BST:2010} and~\cite{CL:2014}), though explicit calculation of the 
asymptotic variances may be difficult.

The remainder of the paper is organised as follows.  The MR and NSW random graphs are defined in Section~\ref{sec:network} and the
SIR epidemic model is described in Section~\ref{sec:epidemic}.  The main central limit theorems (Theorems~\ref{thm:posclt}-\ref{thm:zeroclt} for SIR epidemics and Theorem~\ref{thm:percclt} for percolation) are stated in Section~\ref{sec:mainresults}, together with some remarks giving comparisons of their variance parameters and applications to giant components 
indicated above.  Some numerical illustrations, which show that the central limit
theorems can yield good approximations even for relatively small graphs, are given in Section~\ref{sec:illustrations}.   The proofs are given in Section~\ref{sec:proofs}.
They make extensive use of asymptotically density dependent population processes and in particular require a version of the functional central limit 
theorem for such processes to include asymptotically random initial conditions.  For ease of reference, the required results for such processes are collected together in
Section~\ref{sec:DPPP}.  Some brief concluding comments are given in Section~\ref{sec:conc}. Calculation of the asymptotic variances for the central limit theorems is lengthy, though straightforward, so this and a few other details are deferred to an appendix.

\subsection{Notation}
\label{sec:notation}
All vectors are row vectors and $^{\top}$ denotes transpose.  With the dimension being obvious from the context, $I$ denotes an identity matrix and 
$\bzero$ and $\bone$ denotes vectors all of whose elements are $0$ and $1$, respectively.  For $x \in \mathbb{R}$, the usual floor and
ceiling functions are denoted by $\left \lfloor{x}\right \rfloor$ and $\left \lceil{x}\right \rceil$, respectively.  Thus $\left \lfloor{x}\right \rfloor$
is the greatest integer $\le x$ and $\left \lceil{x}\right \rceil$ is the smallest integer $\ge x$. For a  positive integer $k$, the 
$k$th derivative of a real-valued function $f$ is denoted by $f^{(k)}$.  The cardinality of a set $A$ is denoted by $|A|$.  Sums are zero if vaccuous.  We use $\convp$, $\convas$ and $\convD$ to denote convergence in probability, convergence almost sure and convergence in distribution, respectively.

Further, ${\rm U}(0,1)$ denotes a uniform random variable on $(0,1)$; ${\rm Exp}(1)$ denotes an exponential random variable with mean $1$;
${\rm N}(0,\sigma^2)$ denotes a univariate normal random variable with mean $0$ and variance $\sigma^2$; and ${\rm N}(\bzero,\Sigma)$ denotes
a multivariate normal random variable with mean vector $\bzero$ and variance matrix $\Sigma$, whose dimension again is obvious from the context.
For a positive integer $n$ and $p \in [0,1]$,  ${\rm Bin}(n,p)$ denotes a binomial random variable with $n$ trials and success probability $p$.
Also, if $I$ is a non-negative random variable and $\lambda \in (0,\infty)$ then ${\rm Bin}(n,1-\re^{-\lambda I})$ denotes a mixed-Binomial
random variable obtained by first sampling $I_1$ from the distribution of $I$ and then, given $I_1$, sampling independently from ${\rm Bin}(n,1-\re^{-\lambda I_1})$.
Similarly,  if $I$ is a non-negative random variable and $D$ is a non-negative integer-valued random variable, then ${\rm Bin}(D,1-\re^{-\lambda I})$
denotes a mixed-Binomial random variable, where the realisations of $D$ and $I$ are independent.  Thus, if $X \sim  {\rm Bin}(D,1-\re^{-\lambda I})$,
then 
\[
{\rm P}(X=k)=\sum_{d=k}^{\infty}{\rm P}(D=d){\rm E}\left[\binom{d}{k}(1-\re^{-\lambda I})^k \re^{-(d-k)\lambda I}\right] \qquad (k=0,1,\dots).
\]
Note that we allow the possibility $D=0$.

\section{Model and main results}
\label{sec:model}
\subsection{Random graph}
\label{sec:network}
Consider a population of $n$ indivdiuals labelled $1,2,\dots,n$.  For $i=1,2,\dots,n$, let $\Dn_i$ denote the 
degree of individual $i$.  We assume that $0 \le \Dn_i \le \dmax$ for all $i$, i.e.~that there is a maximum degree $\dmax$.  In the MR random graph the degrees are prescribed, while in the NSW random graph 
$\Dn_1,\Dn_2,\dots,\Dn_n$ are i.i.d. copies of a random variable $D$ having
probability mass function given by ${\rm P}(D=i)=p_i$ $(i=0,1,\dots,\dmax)$.  Under both models, the network (random graph) is formed by
attaching $\Dn_i$ half-edges to individual $i$, for $i=1,2,\dots,n$, and then pairing up the 
$\Dn_1+\Dn_2+\dots+\Dn_n$ half-edges uniformly at random to give the edges in the random graph, which we denote by $\GGn$.  In the NSW model, 
if $\Dn_1+\Dn_2+\dots+\Dn_n$ is odd there is a left-over stub, which is ignored.  (Of course in the MR model the prescribed degrees
can be chosen so that $\Dn_1+\Dn_2+\dots+\Dn_n$ is even.)  

We are interested in asymptotic results as 
the number of individuals $n \to \infty$.  In the MR random graph, for $i=0,1,\dots,\dmax$, let $\vn_i=\sum_{k=1}^n 1_{\{\Dn_k=i\}}$
be the number of individuals having degree $i$.  We assume that
\begin{equation}
\label{equ:vni}
\lim_{n \to \infty} \sqrt{n}\left(n^{-1} \vn_i-p_i\right)=0 \qquad (i=0,1,\dots, \dmax).
\end{equation}
Note that~\eqref{equ:vni} implies $\lim_{n \to \infty} n^{-1}\vn_i=p_i$ $(i=0,1,\dots, \dmax)$.

In both models the random graph may have some imperfections, specifically self-loops and multiple edges, but they are sparse in the 
network as $n \to \infty$; more precisely, the number of such imperfections converges in distrubution to a Poisson random variable
as $n \to \infty$ (\cite{Durrett:2007}, Theorem 3.1.2).  Moreover, \cite{Janson:2009b} implies that the probability that the random graph
is simple is bounded away from $0$, as $n \to \infty$, and law of large numbers results continue to hold if the graph is conditioned on being simple.  However,
that is not necessarily the case for convergence in distribution; see~\cite{Janson:2010}, Remark 1.4, and~\cite{BR:2017}, Remark 2.5.
Thus whether or not our central limit theorems continue to hold when the random graph is conditioned on being simple is an open question.

\subsection{SIR epidemic}
\label{sec:epidemic}
An SIR epidemic, denoted by $\EEn$, is constructed on the above network as follows.  Initially, at time $t=0$, a number of individuals are infective
and the remaining individuals are susceptible.  (Precise statements concerning the initial infectives are made later.)  Distinct infectives 
behave independently of each other and of the construction of the network.  Each infective remains infectious for a period of time 
that is distributed according to a random variable $I$, having an arbitrary but specified distribution,
after which it becomes recovered.
During its infectious period,  an infective contacts its neighbours in the network independently at the points of Poisson processes each
having rate $\lambda$, so the probability that a given neighbour is contacted is $p_I=1-\phi(\lambda)$, where $\phi(\theta)=\E[\exp(-\theta I)]$
$(\theta \ge 0)$ is the Laplace transform of $I$.  If a contacted individual is susceptible then it becomes an infective, otherwise nothing happens. 
The epidemic ends when there is no infective individual in the population.  

For $t \ge 0$, let $\Xn_i(t)$ be the number of degree-$i$ susceptible individuals at time $t$ ($i=0,1,\dots,\dmax$) and let
$\Yn(t)$ be the total number of infectives at time $t$.  Let $\taun=\inf\{t \ge 0:\Yn(t)=0\}$ be the time of the end of the epidemic.
Then $\Tn_i=\Xn_i(0)-\Xn_i(\taun)$ is the total number of degree-$i$ susceptibles that are infected by the epidemic. Let
$\Tn=\sum_{i=0}^{\dmax} \Tn_i$ be the total number of susceptibles infected by the epidemic, i.e.~the final size of the epdiemic.
We are primarily interested in the asymptotic distribution of $\Tn$ as $n \to \infty$.

The proofs allow for the possibility that $p_I=1$, i.e.~${\rm P}(I=\infty)=1$.  In that case the set of individuals that are 
infected during the epidemic $\EEn$ comprises all individuals in the components of $\GGn$ that contain at least one initial infective. 
Thus central limit theorems for the size of the giant component in MR and NSW random graphs follow immediately from our results; 
see Remark~\ref{rmk:giantclt} below.

\subsection{Bond and site percolation}
\label{sec:percolation}
In bond percolation on $\GGn$, each edge in $\GGn$ is deleted indpendently with probability $1-\pi$, while in site percolation $\GGn$,
each vertex (together with all incident edges) is deleted independently with probability $1-\pi$ (\cite{Janson:2009a}).  Interest is
often focused on the size, $\Cn$ say, of the largest connected component in the resulting graph.   

\subsection{Main results}
\label{sec:mainresults}
For $i=0,1,\dots,\dmax$, let $\an_i$ be the number of degree-$i$ initial infectives in the epidemic $\EEn$ and
let $\an=\sum_{i=0}^{\dmax} \an_i$ denote the total number of initial infectives.  In the epidemic on the MR random graph,
we assume that $\an_0, \an_1,\dots,\an_{\dmax}$ are prescribed.  In the epidemic on the NSW random graph, we assume that
$\an$ is prescribed and that the $\an$ initial infectives are chosen by sampling uniformly at random without replacement
from the $n$ individuals in the population.

Let $\epsilonn=n^{-1}\an$ and $\epsilonn_i=n^{-1}\an_i$ $(i=0,1,\dots,\dmax)$.  Suppose that $\epsilonn \to \epsilon$ as
$n \to \infty$ and that $\lim_{n \to \infty} \sqrt{n}(\epsilonn-\epsilon)=0$.  For the epidemic on the MR random graph,
suppose further that, for $i=0,1,\dots,\dmax$, there exists $\epsilon_i$ such that 
$\lim_{n \to \infty} \sqrt{n}(\epsilonn_i-\epsilon_i)=0$. For the epidemic on the NSW random graph, let
$\epsilon_i=\epsilon p_i$ ($i=0,1,\dots,\dmax$).  Let $\mud={\rm E}[D]$, $\sigma_D^2=\var(D)$ and, for $s \in [0,1]$,
$f_D(s)=\sum_{i=0}^{\dmax} p_i s^i$ and $\fde(s)=\sum_{i=0}^{\dmax} (p_i-\epsilon_i) s^i$.  Let $q_I=1-p_I=\phi(\lambda)$ be the probability 
that an infective fails to contact a given neighbour and $q_I^{(2)}=\phi(2\lambda)$ be the probability
that a given infective fails to contact two given neighbours.

The first theorem concerns the case $\epsilon>0$, so in the limit as $n \to \infty$ a strictly positive fraction of the population is initially infective. 
Let $\TnMR$ and $\TnNSW$ denote the final size of the epidemic
$\EEn$ on the MR and NSW random graphs, respectively.  Let $z$ be the unique solution in $[0,1)$ of
\begin{equation}
\label{equ:z}
z-q_I=\mud^{-1}p_I \fde^{(1)}(z)
\end{equation}
and 
\begin{equation}
\label{equ:rho}
\rho=1-\epsilon-\fde(z).
\end{equation}

\begin{theorem}
\label{thm:posclt}
Suppose that $p_i\epsilon_i>0$ for at least one $i>0$ and, if $p_I=1$ then $p_1-\epsilon_1>0$.  Then, as $n \to \infty$,
\begin{equation*}
\sqrt{n}\left(n^{-1}\TnMR-\rho\right) \convD {\rm N}(0,\sigmaMR)
\end{equation*}
and
\begin{equation*}
\sqrt{n}\left(n^{-1}\TnNSW-\rho\right) \convD {\rm N}(0,\sigmaNSW),
\end{equation*}
where
\begin{align}
\label{equ:sigmamr}
\sigmaMR&=h(z)^2\left\{\left[(p_Iq_I+2(z-q_I)^2\right]\mud-p_I^2\left[\fde^{(1)}(z^2)+z^2 \fde^{(2)}(z^2)\right]\right\}\\
&\,+h(z)\left[2p_Iz\fde^{(1)}(z^2)-(z-q_I)\mud\right]+1-\epsilon-\rho-\fde(z^2)\nonumber\\
&\,+\left(q_I^{(2)}-q_I^2\right)h(z)^2 \left[f_D^{(2)}(1)-\fde^{(2)}(z)\right],\nonumber
\end{align}
with 
\begin{equation}
\label{equ:hmr}
h(z)=\frac{p_I^{-1}(q_I-z)}{1-p_I\mud^{-1}\fde^{(2)}(z)},
\end{equation}
and
\begin{align}
\label{equ:sigmansw}
\sigmaNSW&=\frac{\rho(1-\epsilon-\rho)}{1-\epsilon}-h(z)(z-q_I)\left(1-2\epsilon+\frac{2\epsilon\rho}{1-\epsilon}\right)\mud\\
&\,+h(z)^2\left\{\left[p_Iq_I-2z(z-q_I)\right]\mud+(z-q_I)^2\left(\sigma_D^2+\frac{1-2\epsilon}{1-\epsilon}\mud^2\right)\right\}\nonumber\\
&\,+\left(q_I^{(2)}-q_I^2\right)h(z)^2 \left[f_D^{(2)}(1)-(1-\epsilon)f_D^{(2)}(z)\right],\nonumber
\end{align}
with 
\begin{equation}
\label{equ:hnsw}
h(z)=\frac{p_I^{-1}(q_I-z)}{1-p_I\mud^{-1}(1-\epsilon)f_D^{(2)}(z)}.
\end{equation}
\end{theorem}

The second theorem concerns the case when the number of initial infectives is held fixed as $n \to \infty$,
so $\epsilon=0$.  More specifically, in the epidemic on the MR random graph, we assume that $\an_i=a_i$ $(i=0,1,\dots,\dmax)$
for all $n \ge a=\sum_{i=1}^{\dmax}a_i$ and in the epidemic on the NSW random graph, we assume that 
$\an=a$ for all  $n \ge a$.  It is well known that, for large $n$, the process of infectives in the early stages 
of such an epidemic can be approximated by a Galton-Watson branching process, $\B$ say, in which, except
for the initial generation, the offspring distribution is ${\rm Bin}(\tilde{D}-1,1-\re^{-\lambda I})$, where $\tilde{D}$ and $I$ are independent and $\tilde{D}$ has 
the size-biased degree distribution ${\rm P}( \tilde{D}=k)=\mud^{-1}k p_k$ $(k=1,2,\dots,\dmax)$; see, for example,
\cite{BS:2013}.  This offspring distribution has mean
\begin{equation}
\label{equ:Rzero}
R_0=\mudt p_I=\left(\mud+\mud^{-1}\sigma_D^2-1\right)p_I,
\end{equation}
where $\mudt={\rm E}[\tilde{D}-1]$.
The quantity $R_0$ is called the basic reproduction number of the epidemic. 

For $\EEn$, we say that a major outbreak occurs if and only if the event $\Gn=\{\Tn \ge \log n \}$ occurs.
Now $\lim_{n \to \infty}{\rm P}(\Gn)={\rm P}(B=\infty)$, where $B$ is the total progeny (not including the initial generation)
of the branching process $\B$ (cf.~Theorem~\ref{thm:BPapprox}).  Thus, in the limit as $n \to \infty$, a major outbreak occurs with non-zero probability
if and only if $R_0>1$.  

Suppose that that $R_0>1$.  Now let $z$ be the unique solution in $[0,1)$ of
\begin{equation}
\label{equ:zthm2}
z-q_I=\mud^{-1}p_I f_D^{(1)}(z)
\end{equation}
and $\rho=1-f_D(z)$.  Note that if $p_I \in (0,1)$, or $p_I=1$ and $p_1>0$, then $z>0$.

\begin{theorem}
\label{thm:majclt}
Suppose that $R_0>1$ and, if $p_I=1$ then $p_1>0$.  Then, as $n \to \infty$,
\begin{equation*}
\sqrt{n}\left(n^{-1}\TnMR-\rho\right)|\Gn \convD {\rm N}(0,\sigmatMR)
\end{equation*}
and
\begin{equation*}
\sqrt{n}\left(n^{-1}\TnNSW-\rho\right)|\Gn \convD {\rm N}(0,\sigmatNSW),
\end{equation*}
where $\sigmatMR$ is given by~\eqref{equ:sigmamr} and~\eqref{equ:hmr}, with $\fde$ replaced by $f_D$, and
$\sigmatNSW$ is obtained by setting $\epsilon=0$ in~\eqref{equ:sigmansw} and~\eqref{equ:hnsw}.
\end{theorem}

The next theorem concerns the case when $\epsilon=0$ but the number of initial infectives $\an \to \infty$
as $n \to \infty$.  

\begin{theorem}
\label{thm:zeroclt}
Suppose that $R_0>1$, $\epsilon=0$, $\sum_{i=1}^{\dmax} \an_i \to \infty$ as $n \to \infty$ and, if  $p_I=1$ then $p_1>0$.  Then, as $n \to \infty$,
\begin{equation*}
\sqrt{n}\left(n^{-1}\TnMR-\rho\right) \convD {\rm N}(0,\sigmatMR)
\end{equation*}
and
\begin{equation*}
\sqrt{n}\left(n^{-1}\TnNSW-\rho\right) \convD {\rm N}(0,\sigmatNSW),
\end{equation*}
where $\sigmatMR$ and $\sigmatNSW$ are as in Theorem~\ref{thm:majclt}.
\end{theorem}

\begin{remark}
Note that $q_I^{(2)}=q_I^2$ if $I$ is almost surely constant, otherwise $q_I^{(2)}>q_I^2$ by
Jensen's inequality.  Also $f_D^{(2)}(1)-\fde^{(2)}(z)>0$,
so as one would expect on intuitive grounds, if $p_I$ is held fixed, the asymptotic variance
$\sigmaMR$ is smallest when the infectious period is constant.  A similar comment holds for 
$\sigmaNSW, \sigmatMR$ and $\sigmatNSW$.
\end{remark}

\begin{remark}
\label{rmk:varcomp}
Although it is not transparent from~\eqref{equ:sigmamr} and~\eqref{equ:sigmansw}, it is
seen easily from the proof that, again as one would expect on intuitive grounds, $\sigmaNSW \ge\sigmaMR$ and $\sigmatNSW \ge \sigmatMR$,
with strict inequalities unless the support of the degree random variable $D$ is concentrated on a single point (when the two models are identical); see  Appendix~\ref{app:proofrmkvarcomp}.
\end{remark}

\begin{remark}
\label{rmk:giantclt}
In the setting of Theorem~\ref{thm:majclt}, if $\an=1$ and $p_I=1$ then with probability tending to $1$ as $n \to \infty$, the event $\Gn$
occurs if and only if the initial infective belongs to the giant component of $\GGn$.  Thus setting $p_I=1$ in Theorem~\ref{thm:majclt}
yields central limit theorems for giant components of MR and NSW random graphs.
\end{remark}

The final theorem is concerned with percolation.  Let $R_0$ be given by~\eqref{equ:Rzero} with $p_I=\pi$. Let $\CnMRbond$ and $\CnMRsite$ denote the
size of the largest connected component after bond and site percolation, respectively, on an MR graph; define $\CnNSWbond$ and  
$\CnNSWsite$ analogously for the NSW graph.  Then, if $R_0>1$, for each of these four choices for $\Cn$, there exists $\epsilon>0$
such that $\lim_{n \to \infty}{\rm P}(\Cn \ge \epsilon n)=1$; see \cite{Janson:2009a}, Theorems 3.5 and 3.9, which also give law of
large number limits for the $\Cn$ under weaker conditions than here.  The following theorem gives associated central limit theorems.

\begin{theorem}
\label{thm:percclt}
Suppose that $\pi \in (0,1)$ and $R_0>1$.  Let $z$ be the unique solutiuon in $(0,1)$ of
\begin{equation}
\label{equ:zperc}
(z-1+\pi)=\mud^{-1}\pi f_D^{(1)}(z)
\end{equation}
and $\rho=1-f_D(z)$.
Then, as $n \to \infty$,
\begin{equation}
\label{equ:MRbondclt}
\sqrt{n}\left(n^{-1}\CnMRbond-\rho\right) \convD {\rm N}(0,\sigmaMRbond),
\end{equation}
\begin{equation}
\label{equ:NSWbondclt}
\sqrt{n}\left(n^{-1}\CnNSWbond-\rho\right) \convD {\rm N}(0,\sigmaNSWbond),
\end{equation}
\begin{equation}
\label{equ:MRsiteclt}
\sqrt{n}\left(n^{-1}\CnMRsite-\pi\rho\right) \convD {\rm N}(0,\sigmaMRsite),
\end{equation}
and
\begin{equation}
\label{equ:NSWsiteclt}
\sqrt{n}\left(n^{-1}\CnNSWsite-\pi\rho\right) \convD {\rm N}(0,\sigmaNSWsite),
\end{equation}
where
\begin{align}
\label{equ:sigmamrbond}
\sigmaMRbond&=h(z)^2\left\{\left[\pi(1-\pi)+2(z-1+\pi)^2\right]\mud-\pi^2\left[f_D^{(1)}(z^2)+z^2 f_D^{(2)}(z^2)\right]\right\}\\
&\,+h(z)\left[2\pi z f_D^{(1)}(z^2)-(z-1+\pi)\mud\right]+1-\rho-f_D(z^2)\nonumber
\end{align}
and
\begin{align}
\sigmaNSWbond\label{equ:sigmanswbond}
&=\rho(1-\rho)-h(z)(z-1+\pi)\mud\\
&\,+h(z)^2\left\{\left[\pi(1-\pi)-2z(z-1+\pi)\right]\mud+(z-1+\pi)^2\left(\sigma_D^2+\mud^2\right)\right\},\nonumber
\end{align}

with
\begin{equation*}
h(z)=\frac{\pi^{-1}(1-\pi-z)}{1-\pi\mud^{-1}f_D^{(2)}(z)};
\end{equation*}
\begin{align}
\label{equ:sigmamrsite}
\sigmaMRsite&=\pi^2\left\{\sigmaMRbond+\pi(1-\pi)h(z)^2\left[f_D^{(2)}(1)-f_D^{(2)}(z)\right]\right\}\\
&\qquad+\pi(1-\pi)[\rho-2h(z)(1-z)\mud]
\end{align}
and $\sigmaNSWsite$ is given by~\eqref{equ:sigmamrsite}, with $\sigmaMRbond$ replaced by $\sigmaNSWbond$.
\end{theorem}

\section{Illustrations}
\label{sec:illustrations}
In this section, we illustrate the central limit theorems in Section~\ref{sec:mainresults} by using simulations to explore their
applicability to graphs with finite $n$.  We also investigate briefly the dependence of the limiting variances in the central limit theorems on
the degree distribution, graph type and infectious period distribution.  We consider four degree distributions: 
\begin{enumerate}
\item[(i)]
$D \equiv d$, i.e.~$D$ is constant with $p_d=1$;
\item[(ii)]
$D \sim {\rm Po}(\mud)$, i.e.~$D$ is Poisson with mean $\mud$;
\item[(iii)]
$D \sim {\rm Geom}(p)$, i.e.~$p_k=(1-p)^k p$ $(k=0,1,\dots)$;
\item[(iv)]
$D \sim {\rm Power}(\alpha,\kappa)$, i.e.~$p_k=ck^{-\alpha}\re^{-\frac{k}{\kappa}}$ $(k=1,2,\dots)$, where $\alpha, \kappa \in (0,\infty)$
and the normalising constant $c={\rm Li}_{\alpha}(\re^{-\frac{1}{\kappa}})$, with ${\rm Li}_{\alpha}$ being the polylogarithm function.
\end{enumerate}
The fourth distribution is a power law with exponential cut-off (see, for example,~\cite{Newman:2002}) that has been used extensively in the physics literature. Note that, with $\theta=\re^{-\frac{1}{\kappa}}$ and $\beta={\rm Li}_{\alpha}(\theta)$, $f_D(s)=\beta^{-1}{\rm Li}_{\alpha}(\theta s)$, $f_D^{(1)}(s)=\frac{1}{\beta s}{\rm Li}_{\alpha-1}(\theta s)$ and $f_D^{(2)}(s)=\frac{1}{\beta s^2}[{\rm Li}_{\alpha-2}(\theta s)-
{\rm Li}_{\alpha-1}(\theta s)]$, which enables $R_0$ and the asymptotic means and variances in the central limit theorems to be calculated.
Distributions (ii)-(iv) have unbounded support, so do not satisfy the conditions of our theorems.  The asymptotic ditributions in this section are calculated under the assumption that the theorems still apply.

Each simulation consists of first simulating one graph, by simulating $\Dn_1, \Dn_2,\dots,\Dn_n$ and pairing up the half-edges, and then
simulating one epidemic, or percolation process, on it.  For an MR graph with (limiting) degree random variable $D$ on $n$ nodes,
the degrees are given by $\Dn_i=\inf\{d \in \mathbb{Z}_+:F_D(d)>i/(n+1)\}$ $(i=1,2,\dots,n)$, where $F_D$ is the distribution function
of $D$.  For heavy-tailed $D$, in particular, this choice of MR degrees converges faster with $n$ to the intended $D$ than one based on
rounding $n p_i$ $(i=0,1,\dots)$ to nearest integers.  Two choices of infectious period distribution are used in the
simulations: (i) $I \equiv 1$ (i.e.~$\P(I=1)=1$) and (ii) $I$ is $0$ or $\infty$, matched to have a common $p_I$.  Note that these yield the minimum and 
maximum asymptotic variances for fixed $p_I$.
The parameters of the degree distributions are chosen so that $\mud=5$.

We first consider epidemics on an NSW network in which a fraction $\epsilon=0.05$ of the population
is initially infective. Table~\ref{table:PF} shows estimates of $\rho_n=n^{-1}\E[\TnNSW]$ and 
$\sigma_n^2=n^{-1}\var[\TnNSW]$ for epidemics with $p_I=0.3$ and various population size $n$, based on 100,000 simulations for
each set of parameters, together with the corresponding asymptotic ($n=\infty$) values given by Theorem~\ref{thm:posclt}.
The table indicates that the asymptotic approximations are useful for even moderate $n$.  The approximations 
are better for the model with $I \equiv 1$ than that with $I= 0$ or $\infty$, as one might expect since there is less
variability in the process with $I \equiv 1$, and improve with increasing $\var(D)$.  The approximations are noticeably
worse when $n=200$ than for the other values of $n$.  Histograms of the final size of 100,000 simulated epidemics on
an NSW network with $n=200$ and $D \sim {\rm Geom}(1/6)$, together with the corresponding density of ${\rm N}(n\rho,n\sigmaNSW)$
with $\rho$ and $\sigmaNSW$ given by Theorem~\ref{thm:posclt}, are shown in Figure~\ref{fig:PFgeom}.  For the epidemic with
$I \equiv 1$, the asymptotic normal distrbution gives an excellent approximation, even though $n$ is rather small.  The
approximation is markedly worse for model with  $I= 0$ or $\infty$, owing to the increased likelihood of small outbreaks.
If $\epsilon$ is held fixed, the probability of a small outbreak decreases approximately exponentially with $n$, as the number of initial infectives is proportional to $n$, and the approximation improves significantly, particularly in the $I= 0$ or $\infty$ model.

\begin{table}
\begin{center}
\[ \begin{array}{|c|c|c|c|c|c|} 
\hline 
& & \multicolumn{2}{|c|}{I \equiv 1} & \multicolumn{2}{|c|}{I=0\mbox{ or }\infty}
\\ \cline{3-6} \mbox{D} &n & \rho_n & \sigma^2_n & \rho_n & \sigma^2_n  \\
\hline
& 200 & 0.5097 & 2.6355& 0.4259  & 9.9771  \\
 & 500 &  0.5284  & 2.3239& 0.4957  & 9.6846  \\
 & 1000 & 0.5334  &2.2180& 0.5200  & 8.0686  \\
D \equiv 5  & 2000 &  0.5360 &  2.1794& 0.5295 & 7.2082  \\
 & 5000 &  0.5374 & 2.1377& 0.5351  & 6.7611  \\ 
 & 10000 &  0.5379 & 2.1177& 0.5366  & 6.6402  \\
 & \infty &  0.5384  & 2.1187 & 0.5384  & 6.5200  \\
\hline 
 & 200 & 0.5659  & 1.3149 & 0.4624 & 10.2610  \\
 & 500 &  0.5761 & 1.0808 &0.5504 & 6.6926  \\
 & 1000 & 0.5789 & 1.0416 & 0.5708  & 4.0793  \\
\mbox{Po}(5)  & 2000 &  0.5803 & 1.0204 & 0.5766 & 3.5567 \\
  & 5000 &  0.5811 & 1.0095& 0.5798   & 3.3708  \\ 
  & 10000 &  0.5814 & 1.0058& 0.5807 & 3.3139  \\
   & \infty &  0.5817  & 1.0044 & 0.5817 & 3.2505  \\
\hline
 & 200 & 0.5160 & 0.4029 & 0.4132 & 8.5412  \\
 & 500 &  0.5184  & 0.3698& 0.5055 & 3.0416  \\
 & 1000 & 0.5189  & 0.3690 &0.5163 & 1.1780  \\
\mbox{Geom}(1/6)  & 2000 &  0.5194  & 0.3635 & 0.5180 & 1.0824  \\
  & 5000 &  0.5196   & 0.3635& 0.5190  & 1.0607  \\ 
  & 10000 &  0.5196  & 0.3653 & 0.5194 & 1.0438  \\
  & \infty &  0.5197 & 0.3650& 0.5197 & 1.0381  \\
\hline
& 200 & 0.4957 & 0.4826 & 0.3900  & 8.8942  \\
 & 500 &  0.4985 & 0.4260& 0.4828  & 4.1147  \\
 & 1000 & 0.4992 &  0.4237 & 0.4960  & 1.8536  \\
\mbox{Power}(1,13.796)  & 2000 &  0.4996 & 0.4177& 0.4981 & 1.6862  \\
  & 5000 &  0.4999 & 0.4164 & 0.4992& 1.6534  \\ 
  & 10000 &  0.4999 &  0.4182& 0.4996  & 1.6483  \\
  & \infty &  0.5000 & 0.4180 & 0.5000  & 1.6372  \\
\hline
\end{array} \]
\end{center}
\caption{ Simulation results against
theoretical (asymptotic) calculations for final size of epidemics with $\epsilon=0.05$ and $p_I=0.3$ on
NSW networks.}
\label{table:PF} 
\end{table}

Turning to Theorem~\ref{thm:majclt}, Figure~\ref{fig:majclt} shows simulations of the final size of epidemics on an NSW network with one initial infective and $I$ constant.  Note that with a single initial infective there is always a non-negligible probability of a minor outbreak, even if $n$ is large.  As $n \to \infty$, the probability of a major epidemic converges to the survival probability,
$p_{\rm maj}$ say, of a branching process; see, for example,~\cite{BS:2013}, Section 2.1.1. Theorem~\ref{thm:BPapprox}. Moreover,
as $I$ is constant, $p_{\rm maj}=\rho$; see, for example, \cite{Kenah07}.
Superimposed on each histogram in  Figure~\ref{fig:majclt} is the density of ${\rm N}(n\rho,n\sigmatNSW)$ multiplied by
$p_{\rm maj}$, which approximates the component of the distribution of $\TnNSW$ corresponding to a major outbreak.
Even with the very small $n$, the approximation is very good when $D \sim {\rm Geom}(1/6)$ or $D \sim {\rm Power}(1,13.796)$.
For both of these degree distributions there is a clear distiction between major and minor outbreaks.  That is not the case for the other two degree distributions, though the approximation is still useful.

\begin{figure}
\begin{center}
\resizebox{0.49\textwidth}{!}{\includegraphics{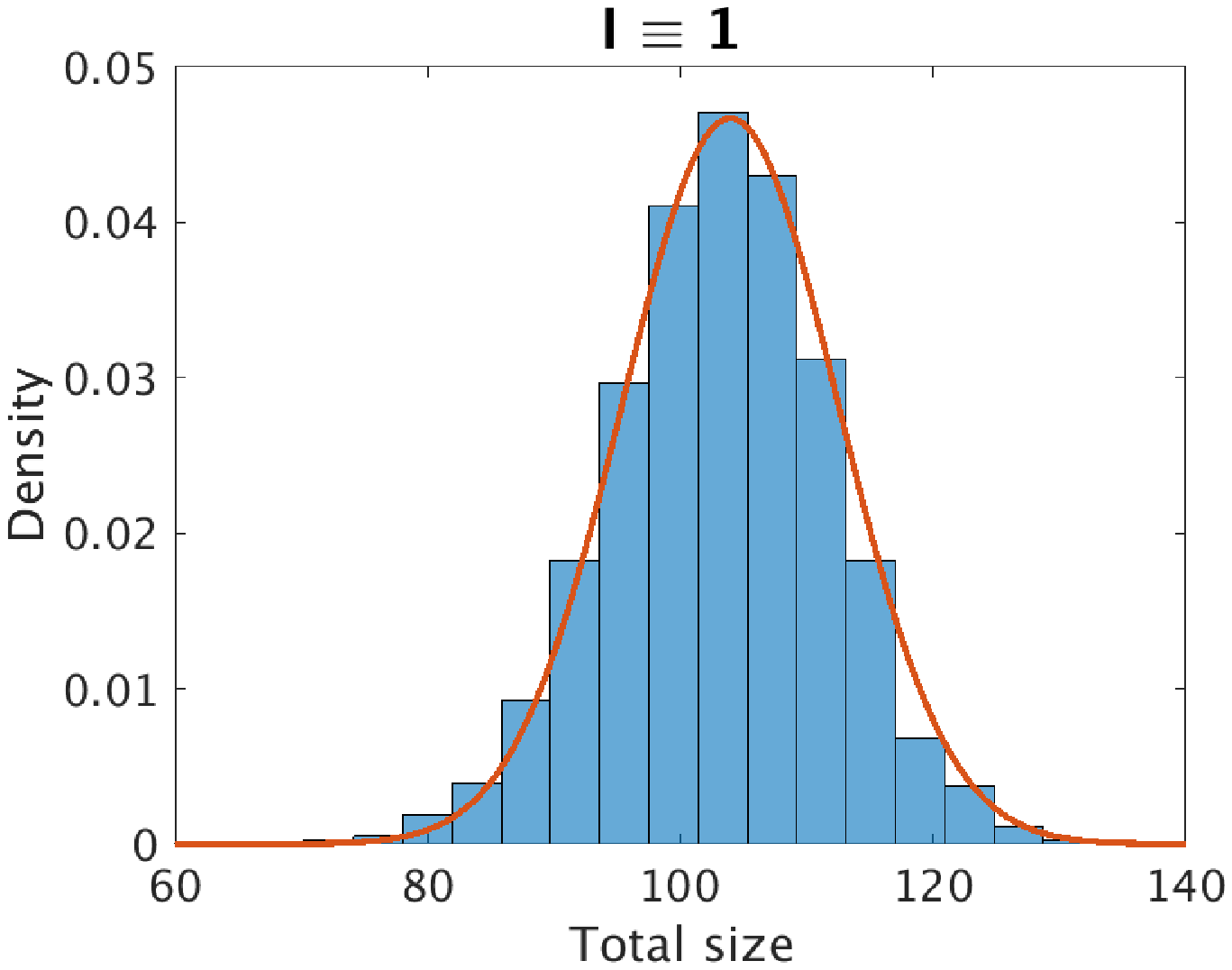}}
\resizebox{0.49\textwidth}{!}{\includegraphics{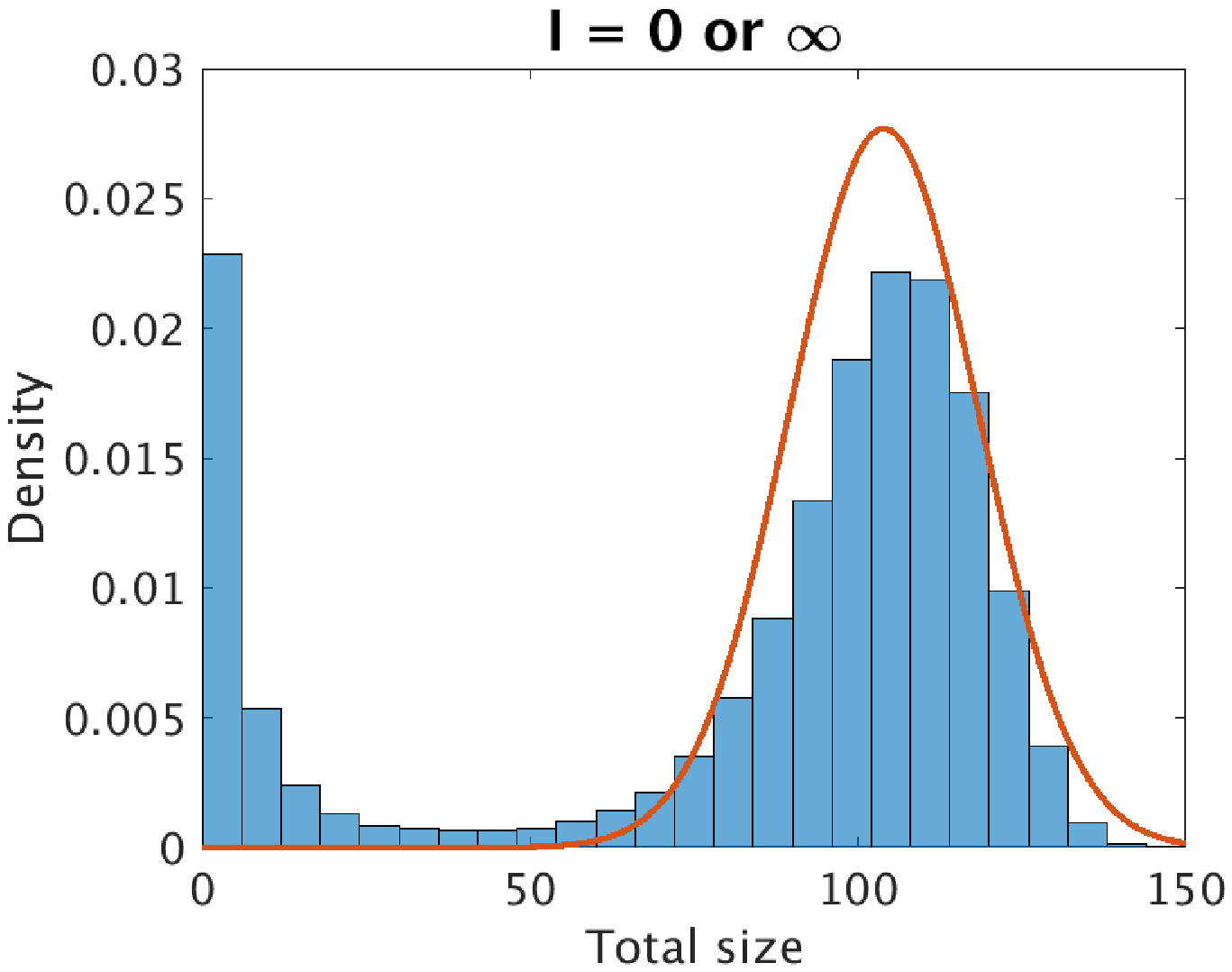}}
\end{center}
\caption{Histograms of 100,000 simulations of final size for epidemics with $n=200, \epsilon=0.05$ and $p_I=0.3$
on NSW networks with $D \sim {\rm Geom}(1/6)$, with asymptotic normal approximation superimposed.}
\label{fig:PFgeom}
\end{figure}

\begin{figure}
\begin{center}
\resizebox{\hfigwidth}{!}{\includegraphics{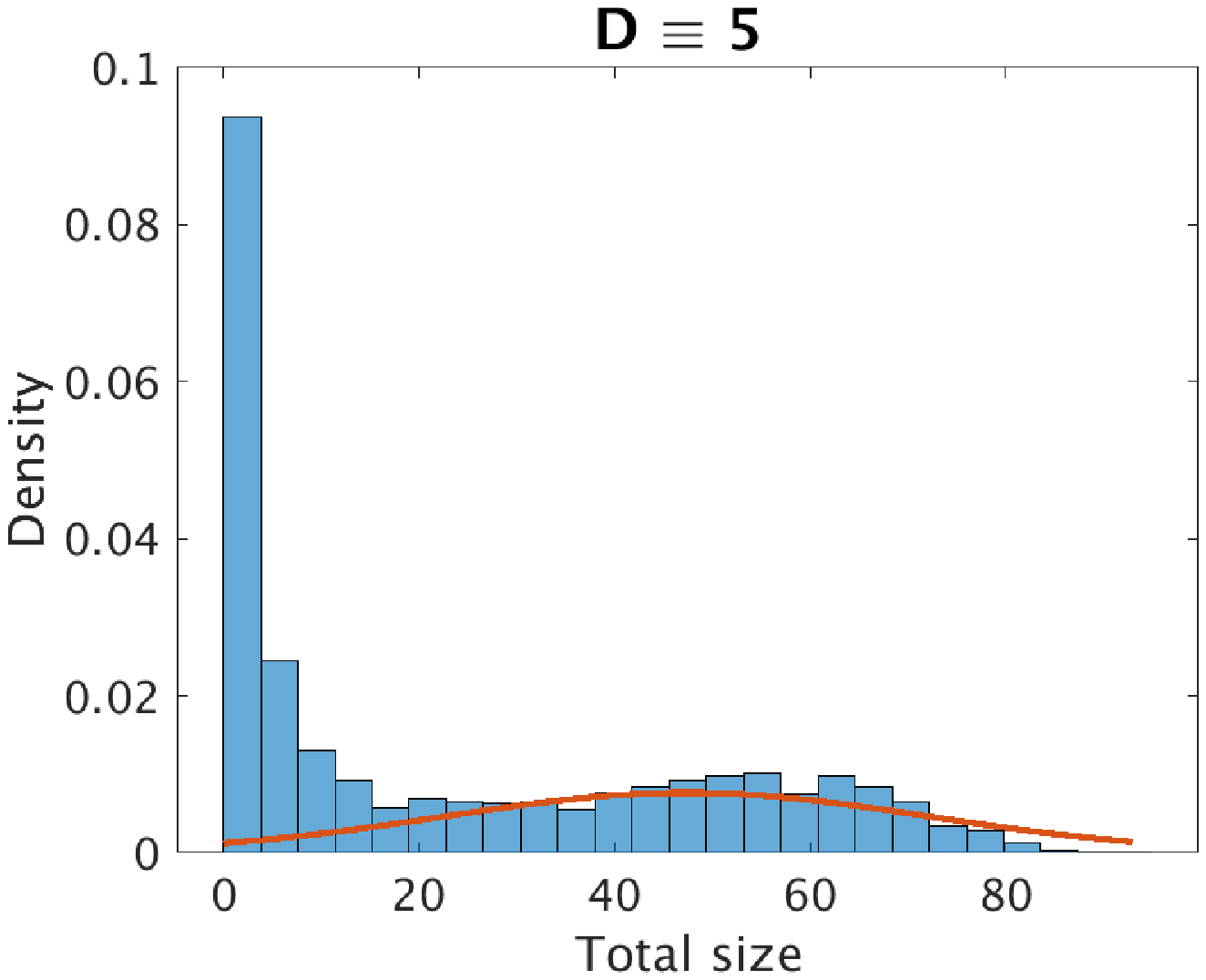}}
\resizebox{\hfigwidth}{!}{\includegraphics{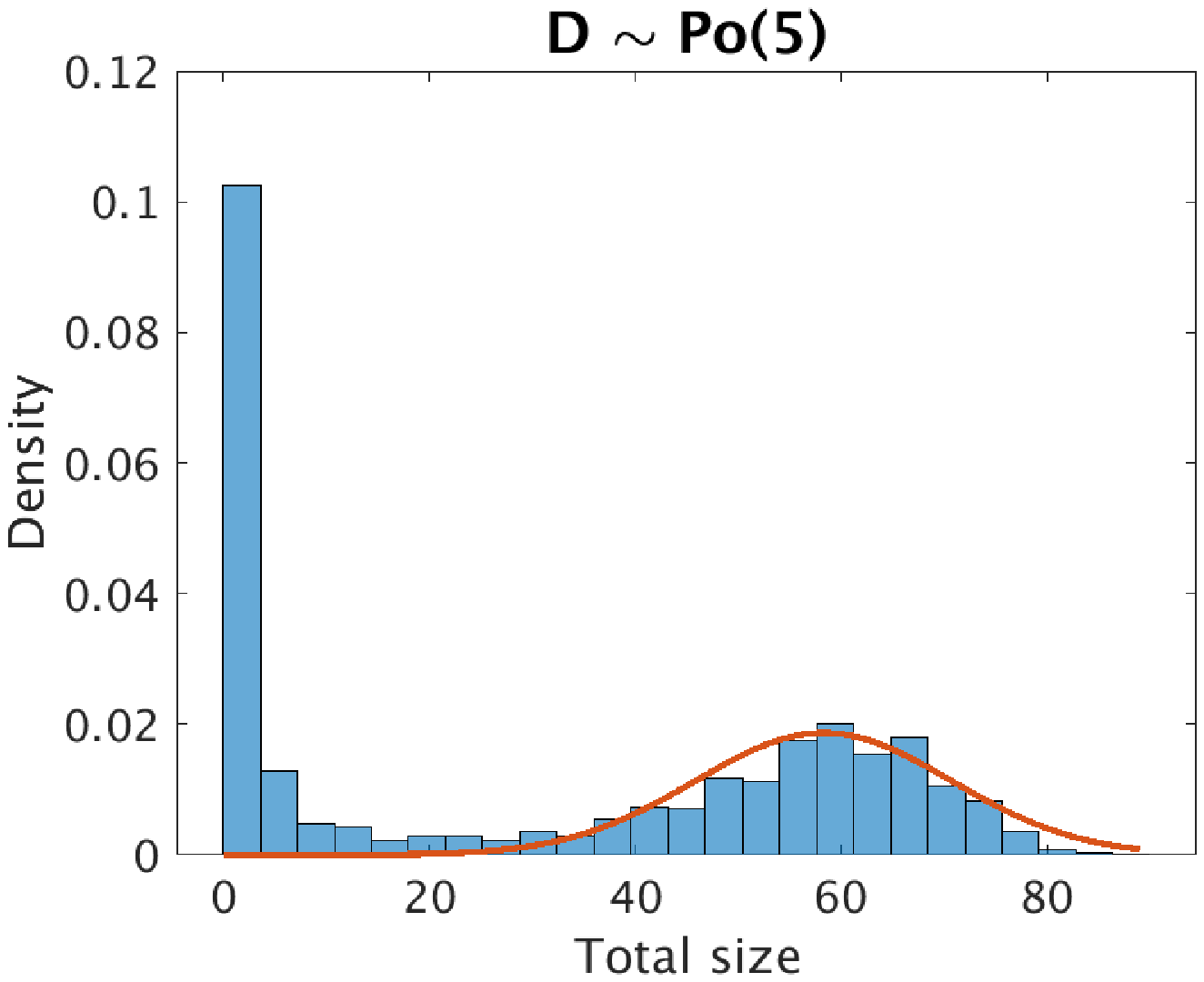}}\\
\resizebox{\hfigwidth}{!}{\includegraphics{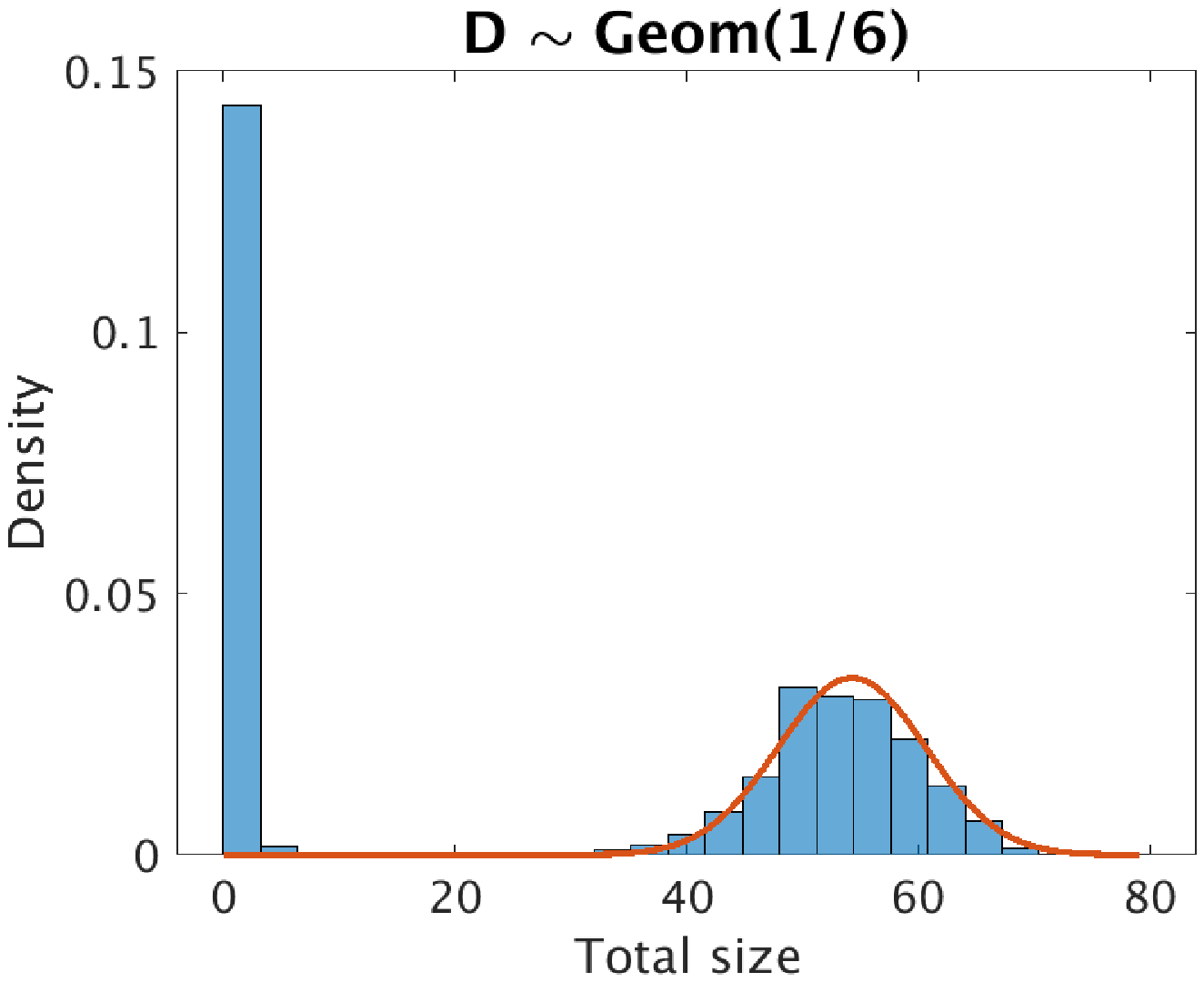}}
\resizebox{\hfigwidth}{!}{\includegraphics{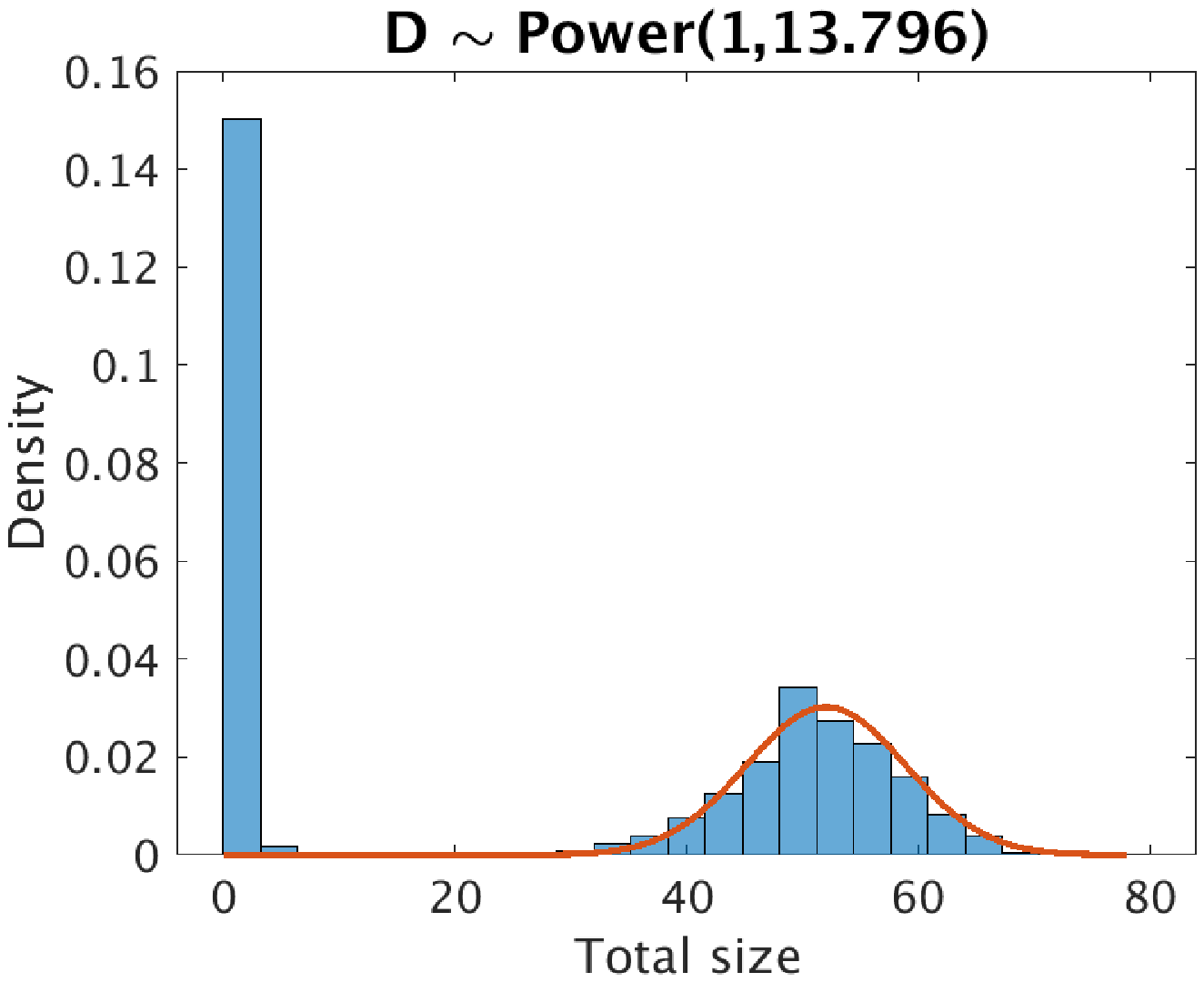}}
\end{center}
\caption{Histograms of 100,000 simulations of final size for epidemics on NSW networks with $I \equiv 1, p_I=1, n=100$ and one initial
infective, with a normal approximation superimposed; see text for details.}
\label{fig:majclt}
\end{figure}

The upper panels of Figure~\ref{fig:limitingvariance} shows the dependence of $\rho$ (left panel) and 
$\sigmatMR$ and $\sigmatNSW$ (right panel) on $p_I$.  The latter two are for the model with $I$ constant.  (Recall that, given $p_I$, 
the scaled asymptotic mean $\rho$ is independent of the distribution of $I$.)  The asymptotic scaled variances all decrease with $p_I$ and converge to
the asymptotic scaled variance of the giant component on the relevant graph as $p_I \uparrow 1$ (cf.~Remark~\ref{rmk:giantclt}).  The asymptotic scaled variances
tend to $\infty$ as $p_I \downarrow p_C$, where $p_C=\mudt^{-1}$ is the critical value of $p_I$ so that $R_0=1$. 
Note that $\sigmatNSW \ge \sigmatMR$ for all choices of $D$, cf.~Remark~\ref{rmk:varcomp}.  The lower panels of
Figure~\ref{fig:limitingvariance} show plots of $(\sigmatNSW - \sigmatMR)/\sigmatMR$ against $p_I$.  Note the
plots for the Poisson and geometric degree distributions are both increasing with $p_I$, while that for the 
${\rm Power}(1,13.796)$ distribution is unimodal.

\begin{figure}
\begin{center}
\resizebox{\hfigwidth}{!}{\includegraphics{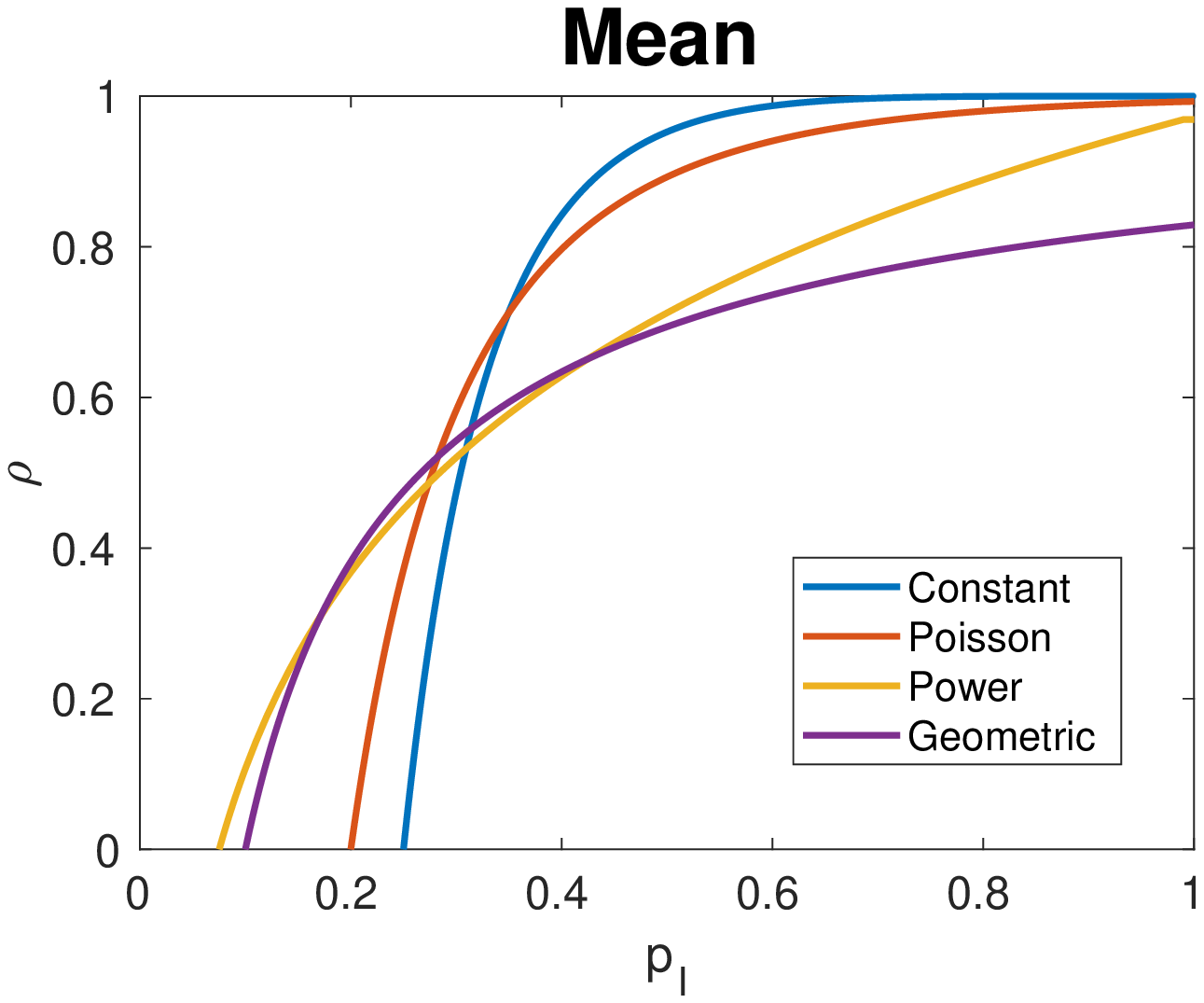}}
\resizebox{\hfigwidth}{!}{\includegraphics{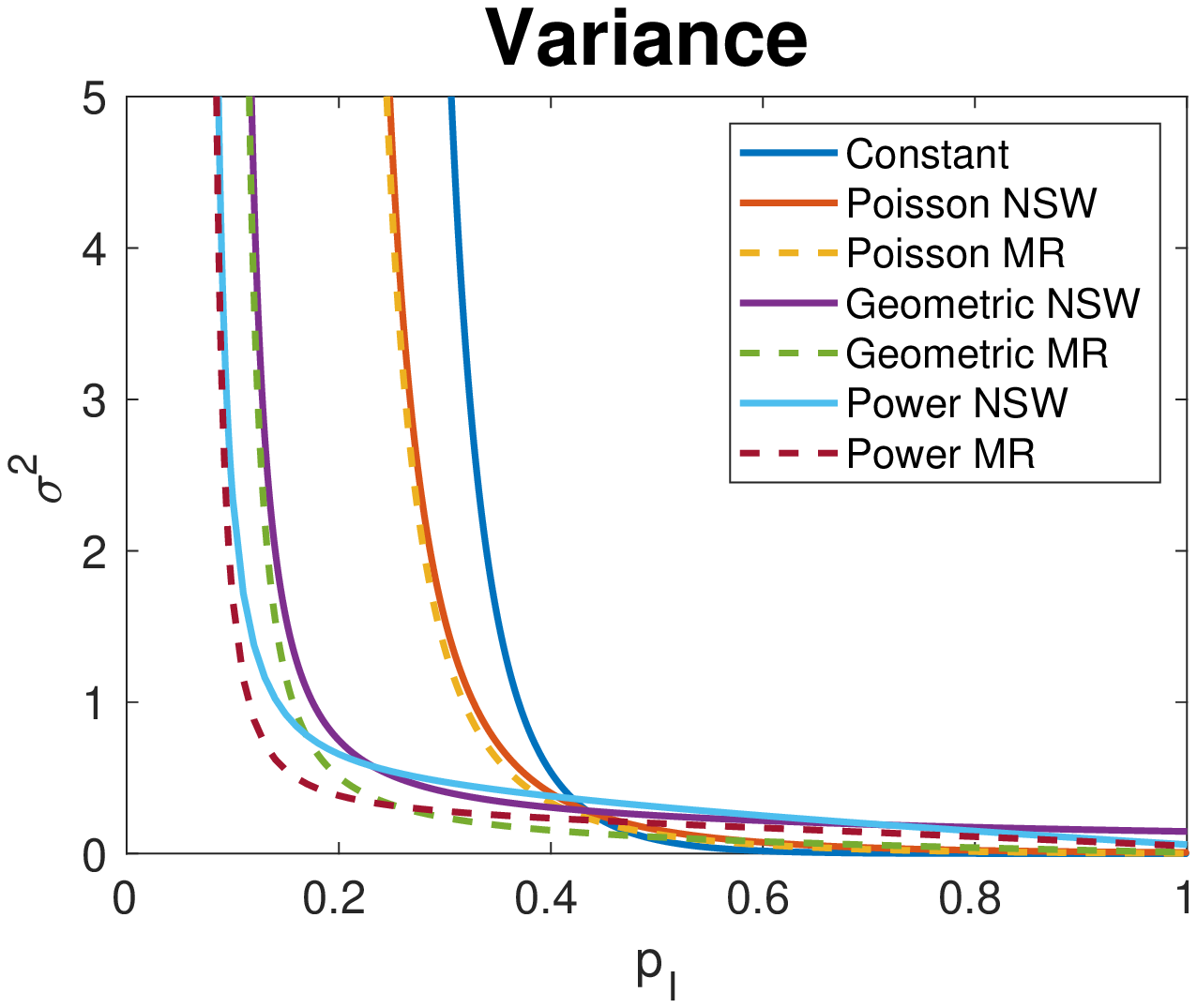}}\\
\resizebox{\hfigwidth}{!}{\includegraphics{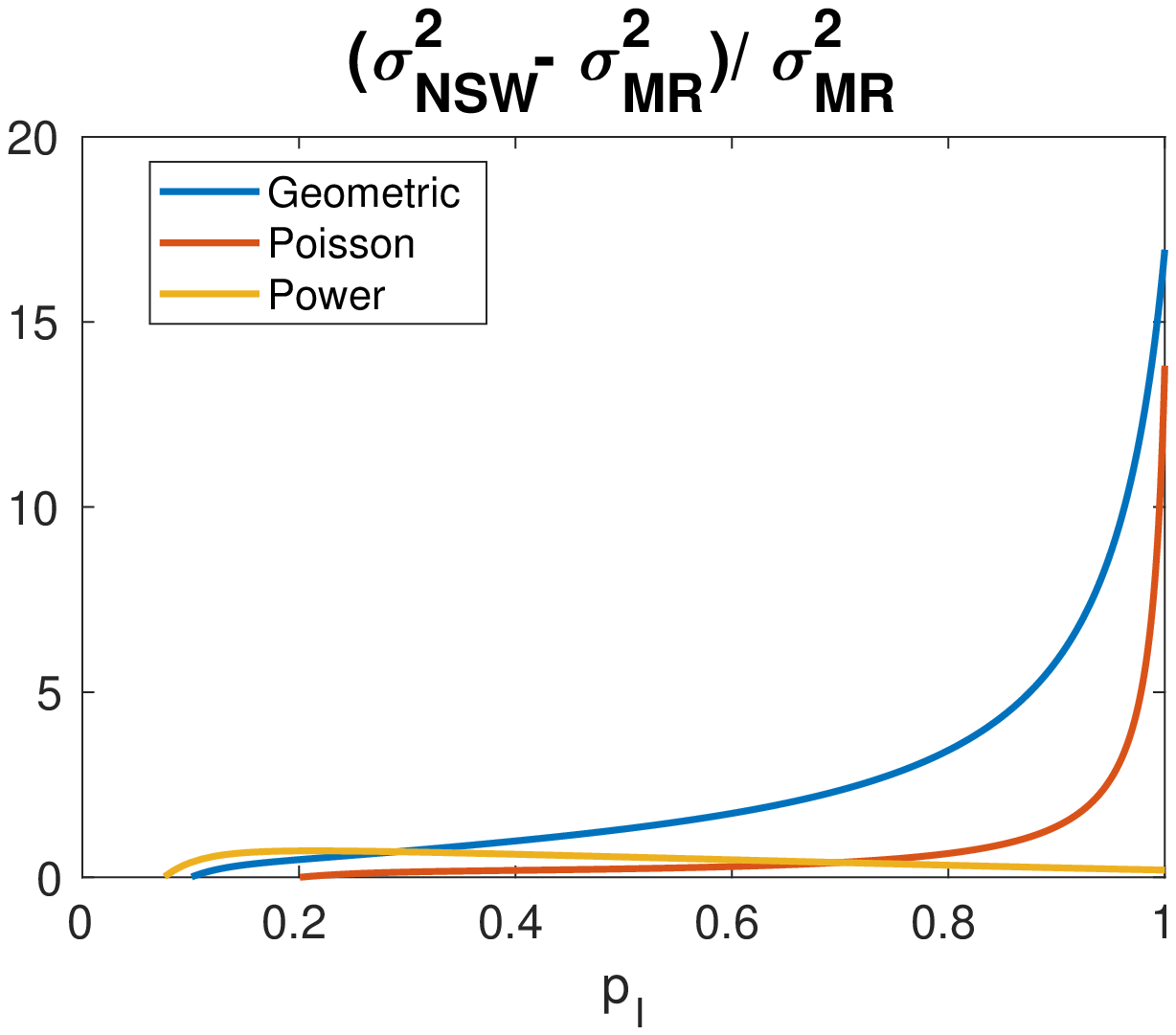}}
\resizebox{\hfigwidth}{!}{\includegraphics{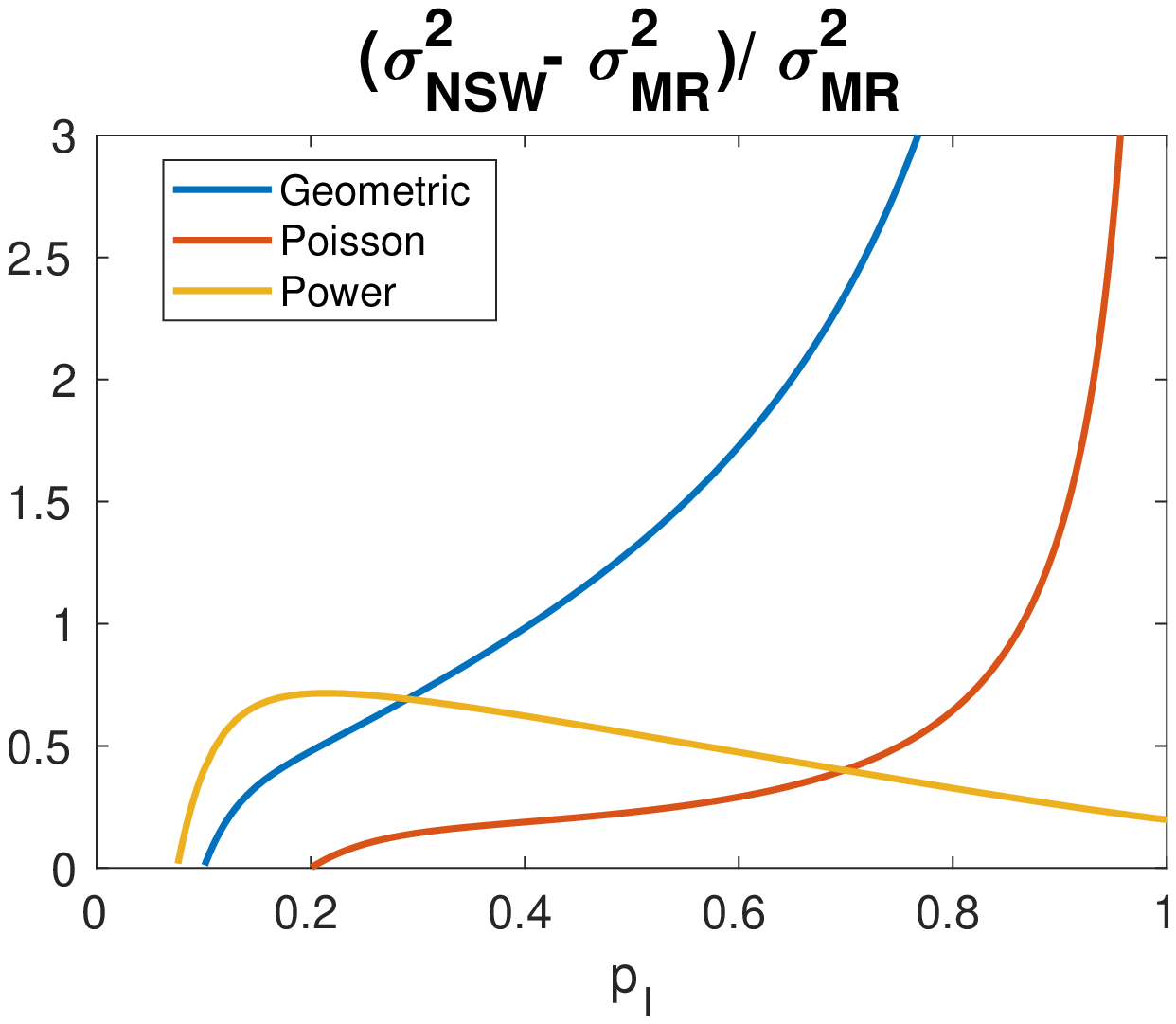}}
\end{center}
\caption{Dependence of asymptotic means and variances on $p_I$; see text for details.}
\label{fig:limitingvariance}
\end{figure}

The final illustrations are concerned with percolation.  Figure~\ref{fig:perc1} shows plots of estimates $\hat{\sigma}_n^2$ of the scaled variance $n^{-1}\var(\Cn)$  of the size of the largest component, based on
$n_{\rm sim}=10,000$ simulations for each choice of parameters, together with $95$\% equal-tailed confidence intervals given by
$\left[(n_{\rm sim}-1)/q_2,(n_{\rm sim}-1)/q_1\right]$, where $q_1$ and $q_2$ are respectively the 2.5\% and 97.5\% quantiles of
the $\chi^2_{n_{{\rm sim}-1}}$ distribution.  In all cases $\pi=0.3$.  The filled and unfilled markers correspond to percolation on
NSW and MR networks, respectively.  The $n \to \infty$ scaled variances, given by Theorem~\ref{thm:percclt},
are shown by horizontal dashed and solid lines for NSW and MR networks, respectively.  The figure suggests that the asymptotic approximations are again 
generally good, even for moderately-sized networks.  For fixed $n$, the approximation is better for  ${\rm Power}(1,13.796)$ distribution than for the ${\rm Po}(5)$ distribution.  The plot for site percolation when $D \sim {\rm Po}(5)$ is a bit odd, as
$\hat{\sigma}_n^2$ is not monotonic in $n$.  This is explored further in Figure~\ref{fig:perc2}, which is for percolation on
NSW random graphs with $D \sim{ \rm Po}(5)$.  Note that the distribution of $\Cn$ is clearly bimodal for site percolation with
$n=200$ and $\pi =0.3$. The lower plots in Figure~\ref{fig:perc1} demonstrate that increasing $n$ or $\pi$ alleviates the issue of small largest components.

\begin{figure}
\begin{center}
\resizebox{\hfigwidth}{!}{\includegraphics{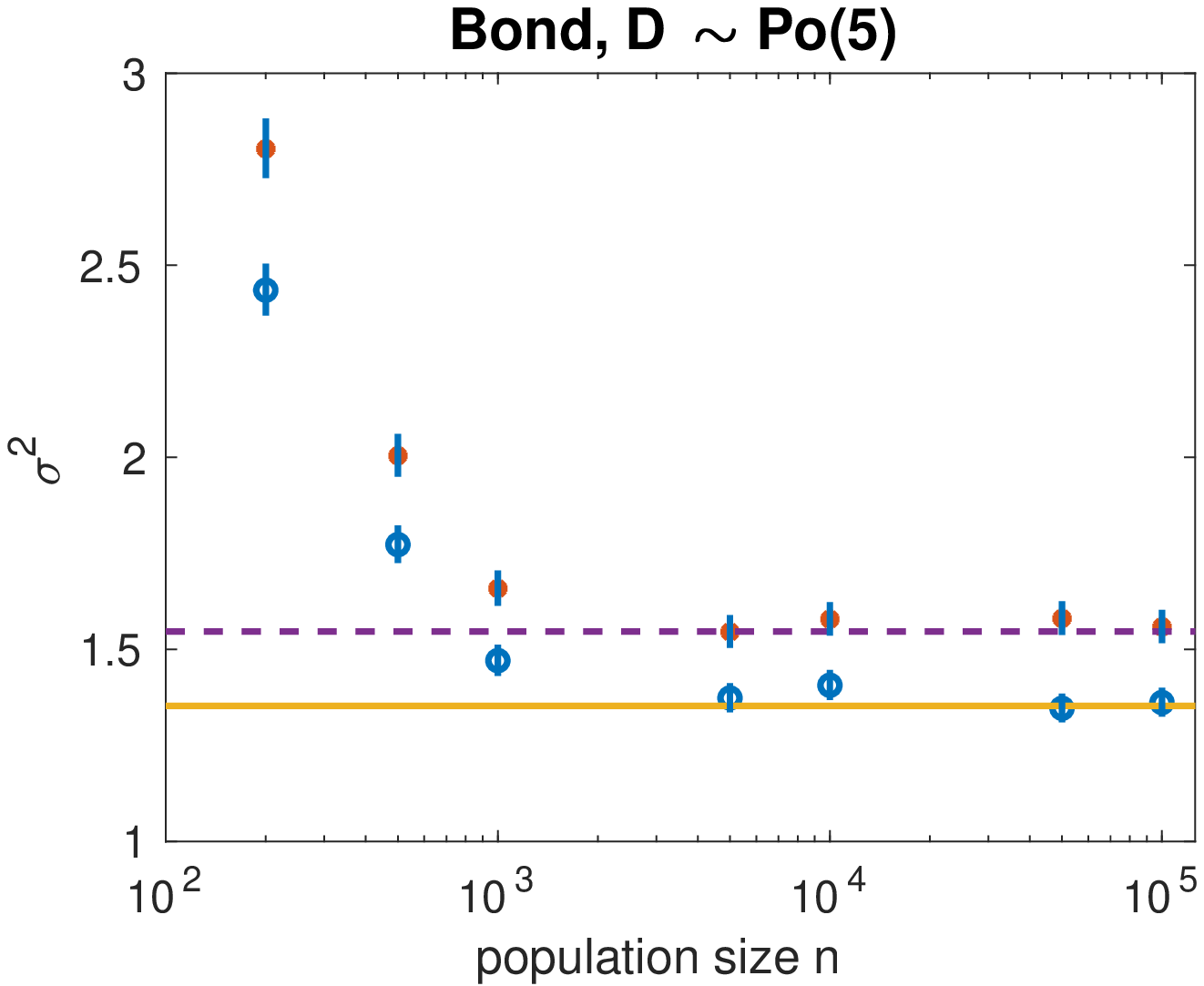}}
\resizebox{\hfigwidth}{!}{\includegraphics{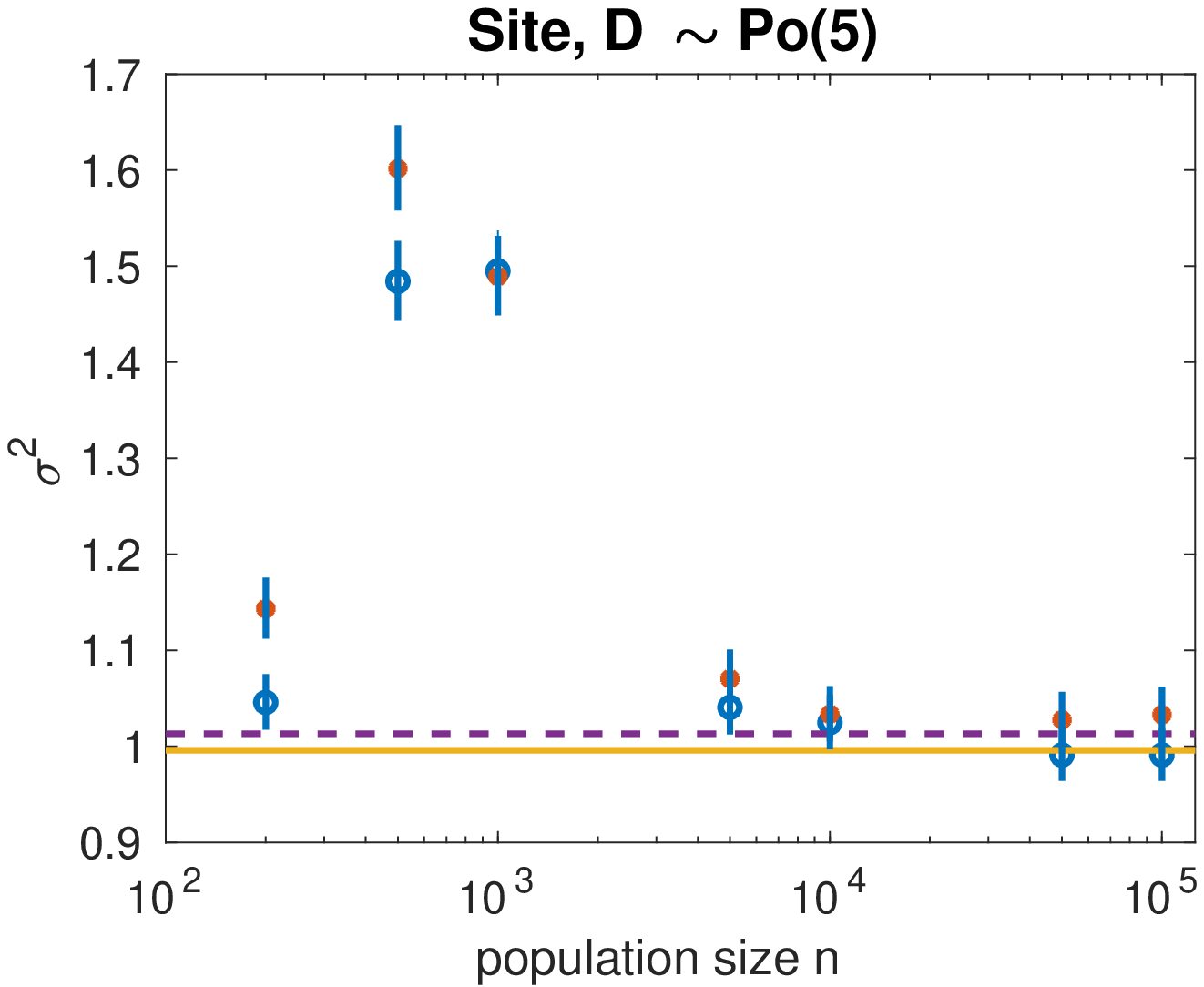}}\\
\resizebox{\hfigwidth}{!}{\includegraphics{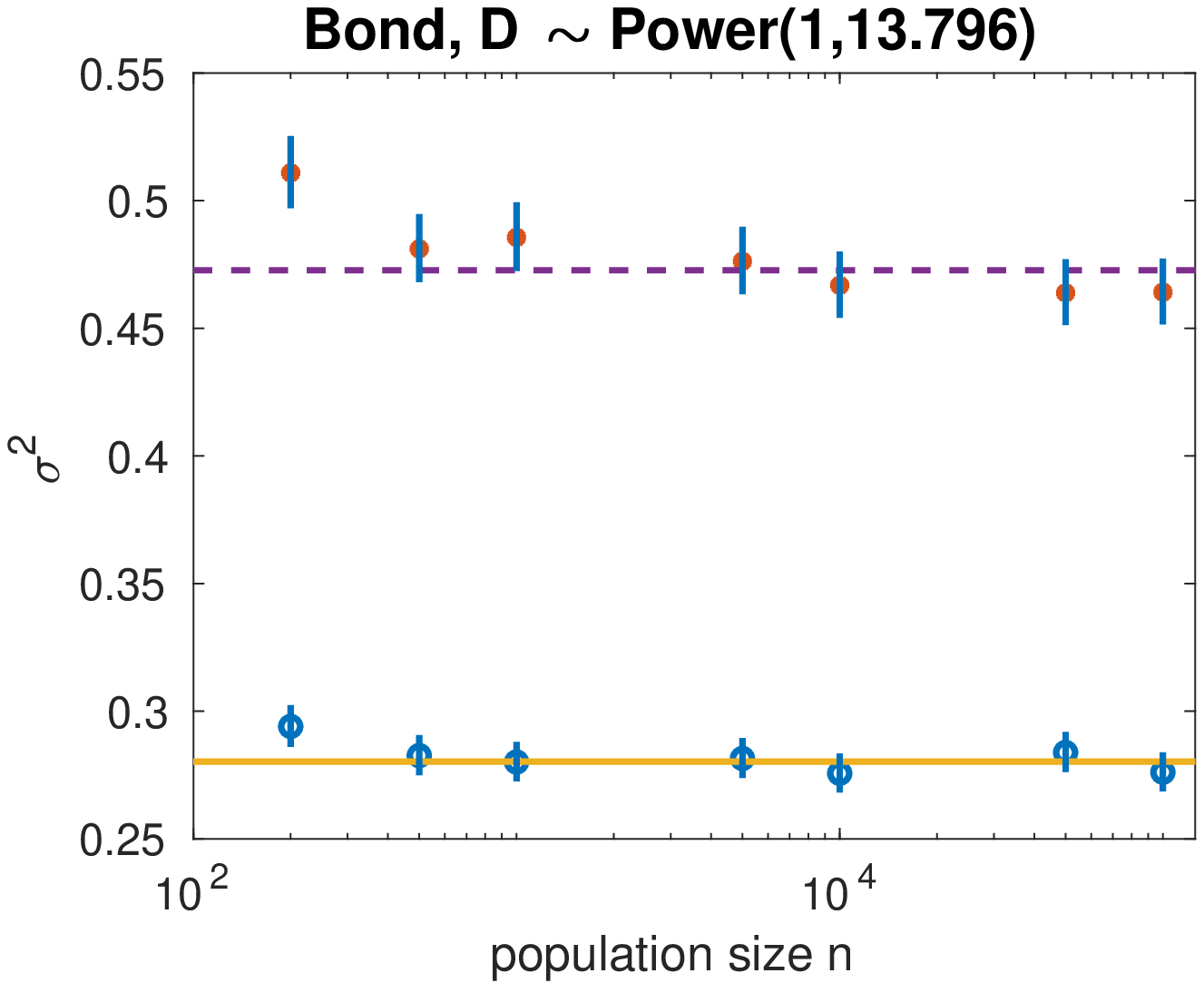}}
\resizebox{\hfigwidth}{!}{\includegraphics{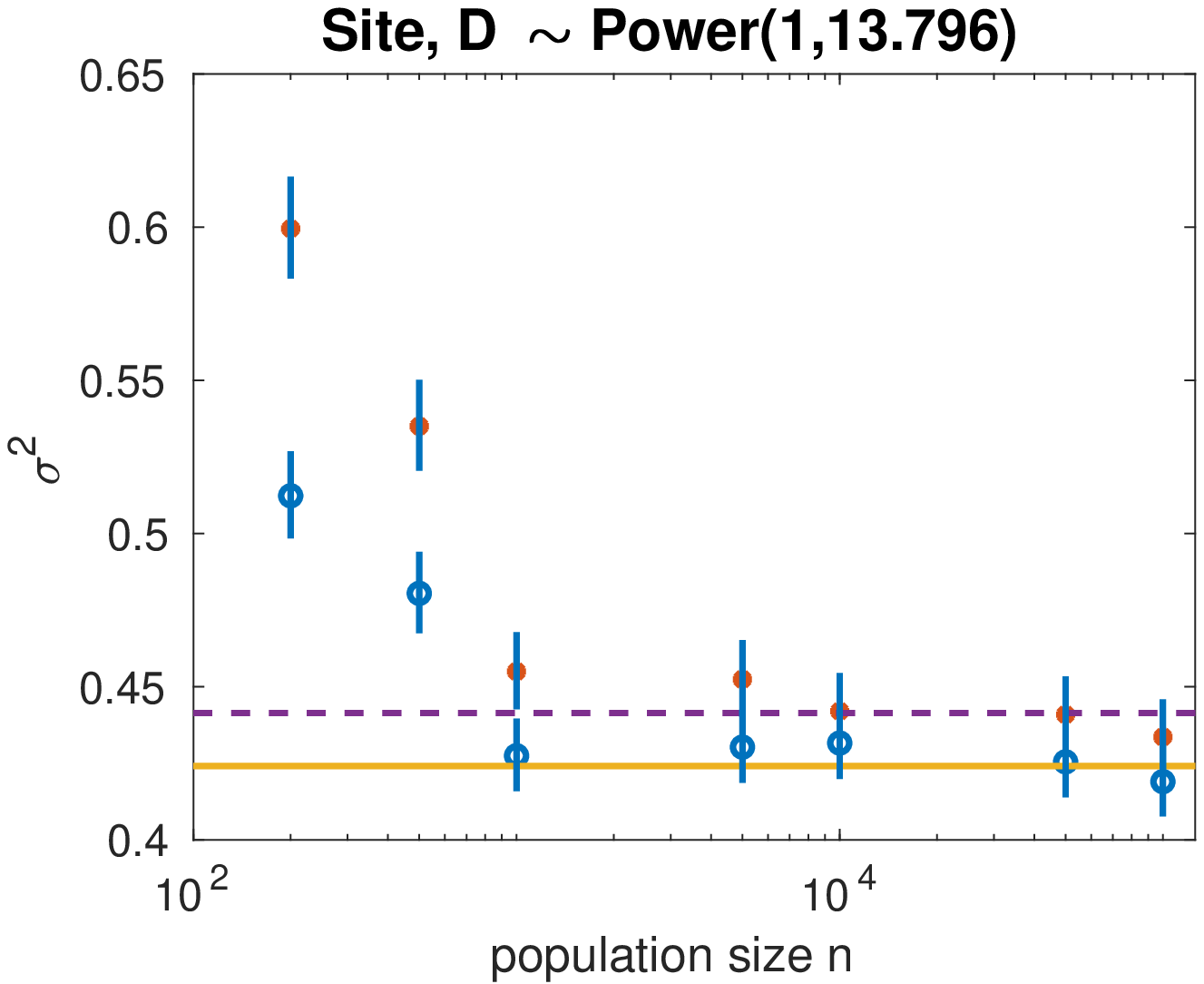}}
\end{center}
\caption{Plots of scaled variances $n^{-1}\var(\Cn)$ of the size of the largest component in bond and site percolation on
configuration model random graphs; see text for details.}
\label{fig:perc1}
\end{figure}

\begin{figure}
\begin{center}
\resizebox{\hfigwidth}{!}{\includegraphics{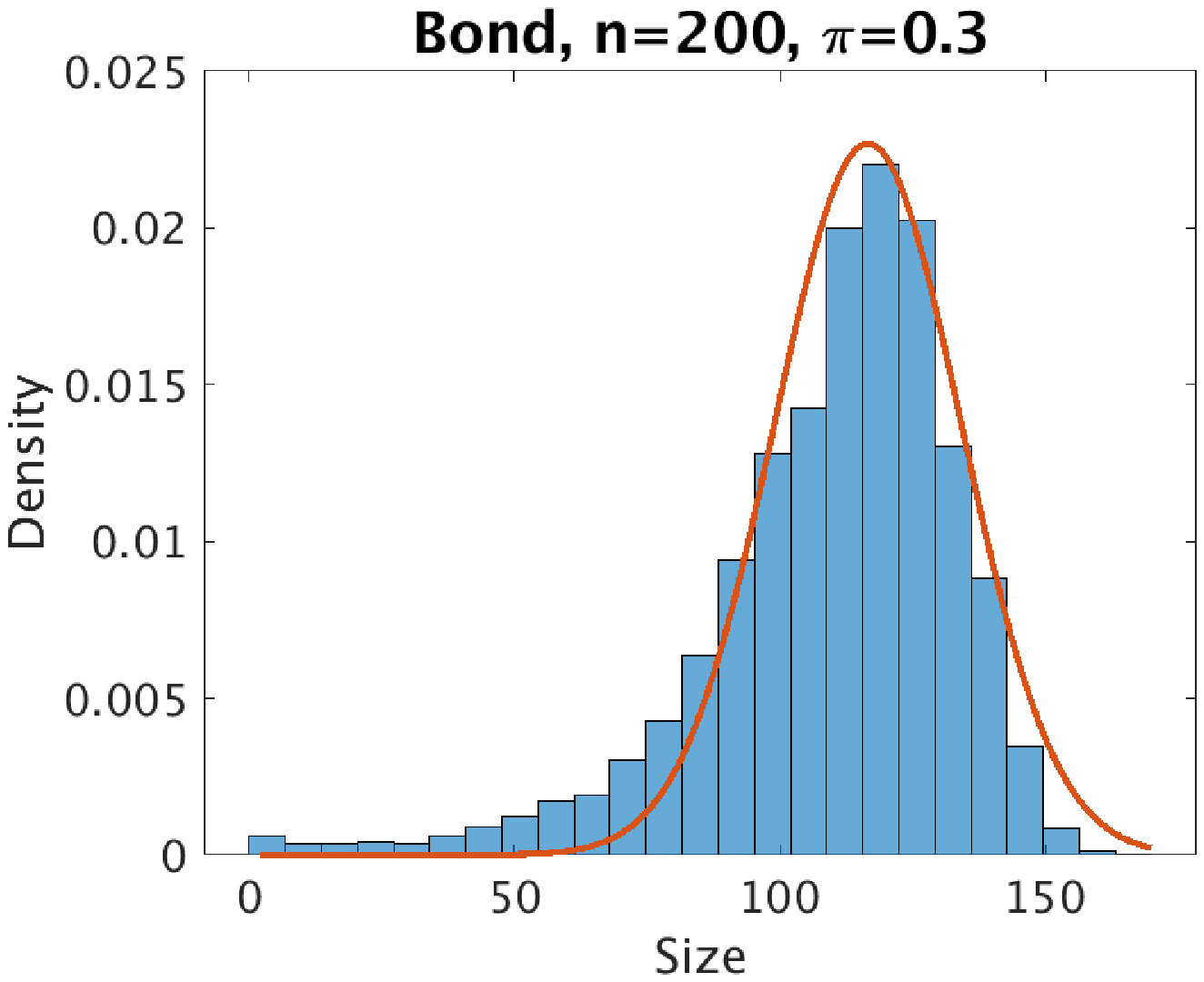}}
\resizebox{\hfigwidth}{!}{\includegraphics{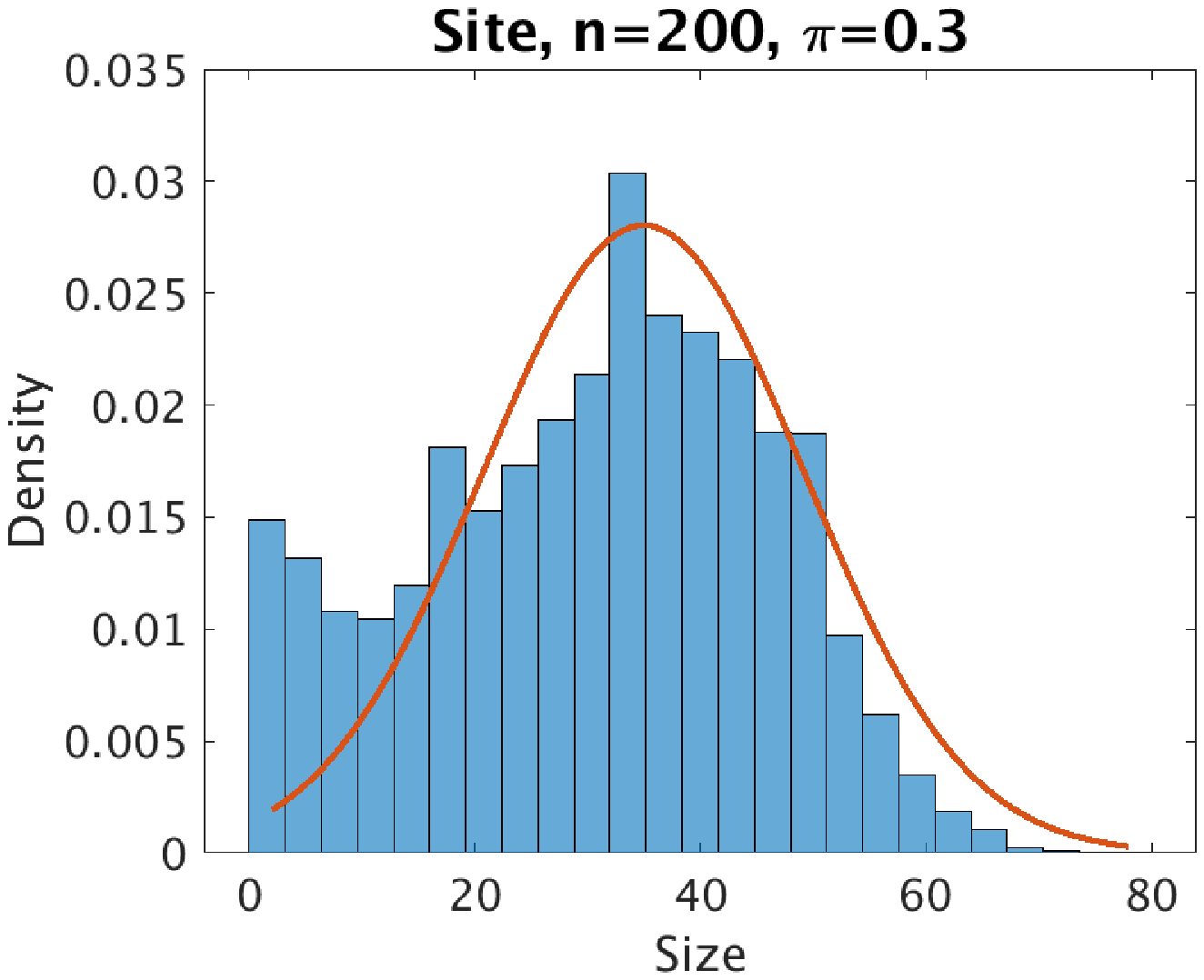}}\\
\resizebox{\hfigwidth}{!}{\includegraphics{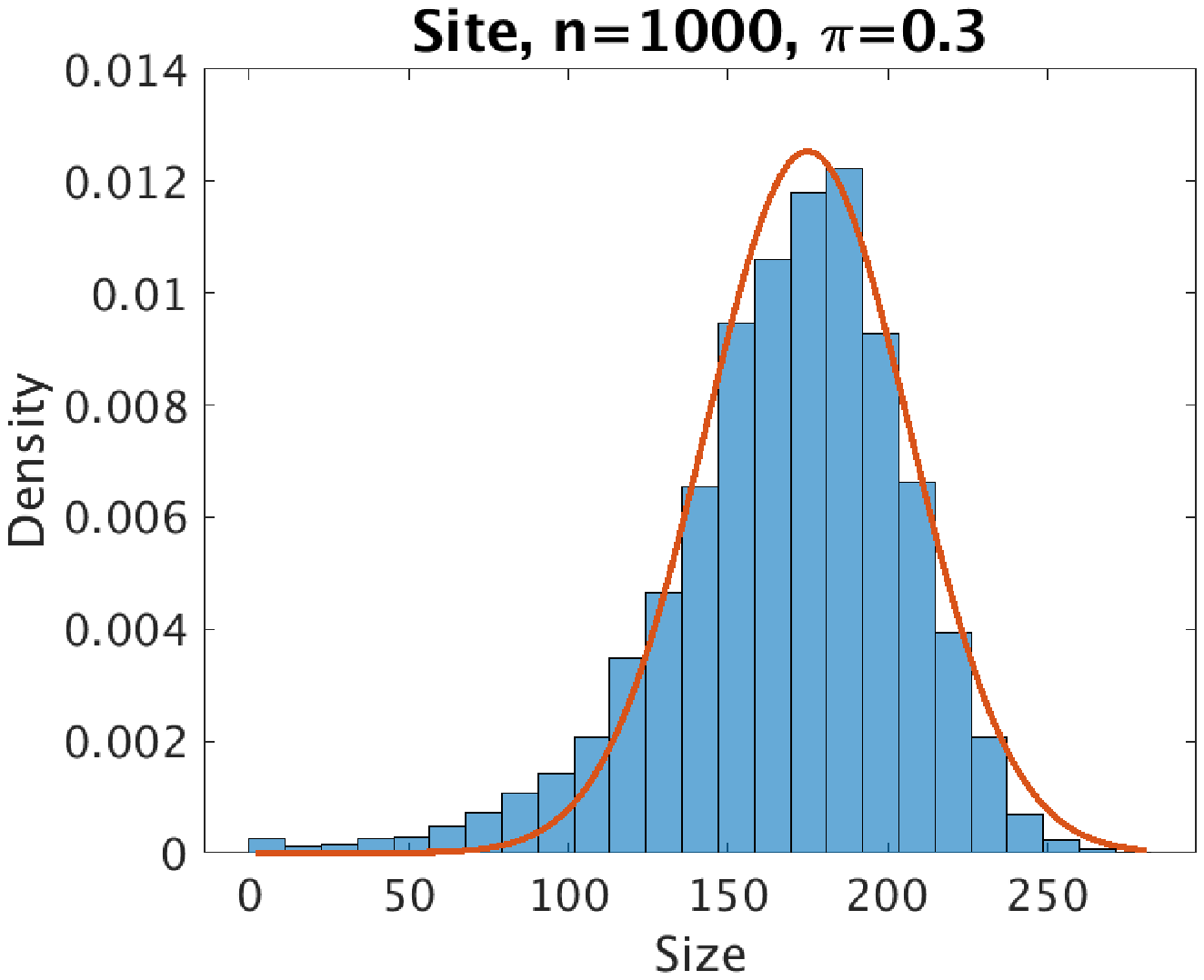}}
\resizebox{\hfigwidth}{!}{\includegraphics{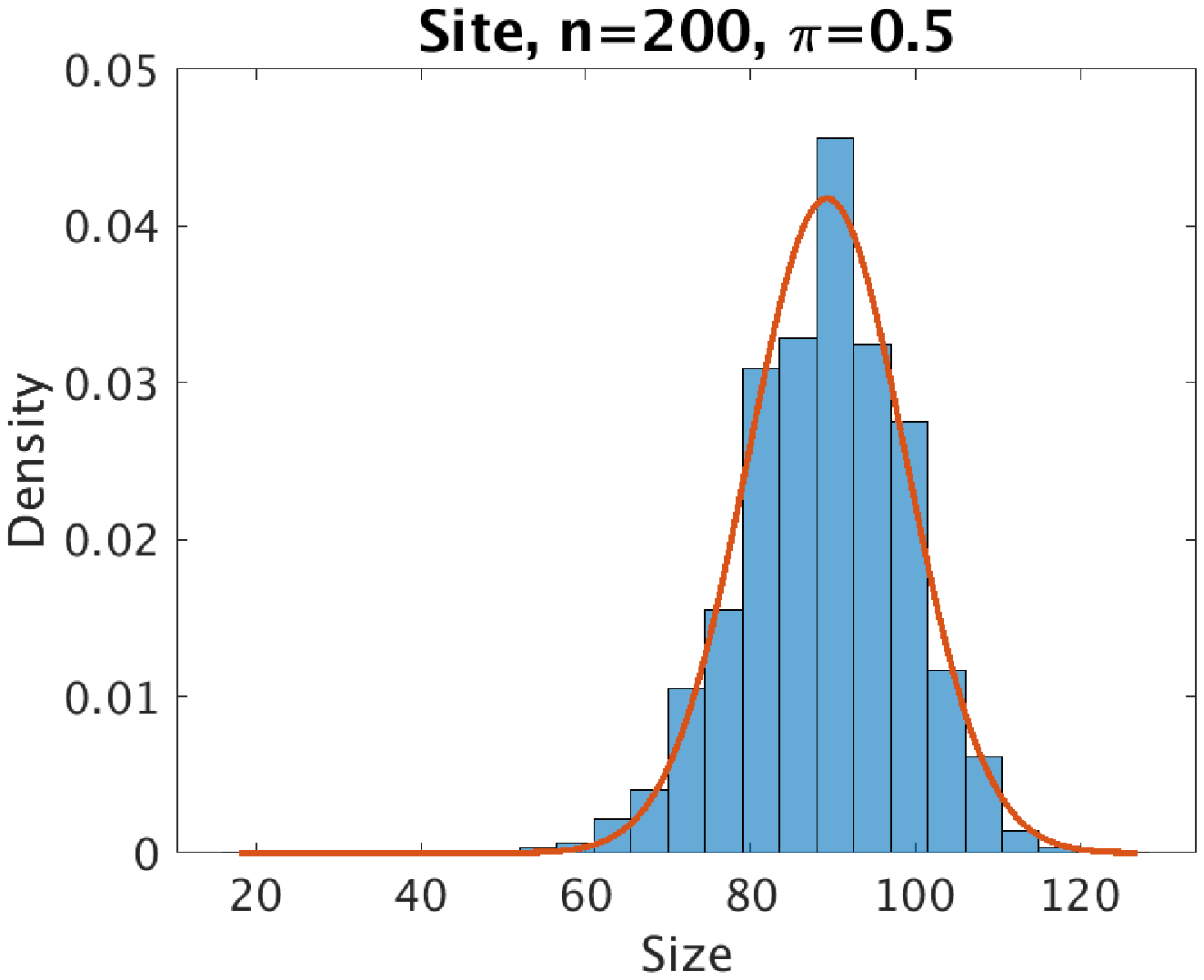}}
\end{center}
\caption{Histograms 100,000 simulations of size of largest component in percolation on
NSW random graphs having $D \sim{ \rm Po}(5)$, with asymptotic normal approximation superimposed.}
\label{fig:perc2}
\end{figure}

Finally, Figure~\ref{fig:perc3} shows histograms of the size of the largest component in 100,000 simulated bond and site
percolations on an MR random graph with $n=200, \pi=0.3$ and $D \sim {\rm Power}(1,13.796)$.  Two asymptotic normal approximations
are superimposed.  The solid lines are the densities of ${\rm N}(n\rho,n\sigmaMRbond)$ (bond percolation) and  ${\rm N}(n\pi\rho,n\sigmaMRsite)$ (site percolation), with $\rho, \sigmaMRbond$ and $\sigmaMRsite$ obtained by setting $D \sim 
{\rm Power}(1,13.796)$ in Theorem~\ref{thm:percclt}.  An improved approximation (dashed lines) is obtained by instead setting
$D$ to be the empirical ditribution of $\Dn_1, \Dn_2,\dots, \Dn_n$.  The difference between the approximations is more
noticeable for bond percolation. (The support of the histogram has been truncated slightly to make the difference clearer.)
The difference decreases with $n$ and is appreciably greater with heavy-tailed degree distributions.  Observe from Figures~\ref{fig:perc2} and~\ref{fig:perc3} that the asymptotic normal approximation underestimates the left tail and
overestimated the right tail of the distribution of $\Cn$.  This phenomenon is present also in the asymptotic normal approximation
of the epidemic final size $\Tn$.

\begin{figure}
\begin{center}
\resizebox{0.49\textwidth}{!}{\includegraphics{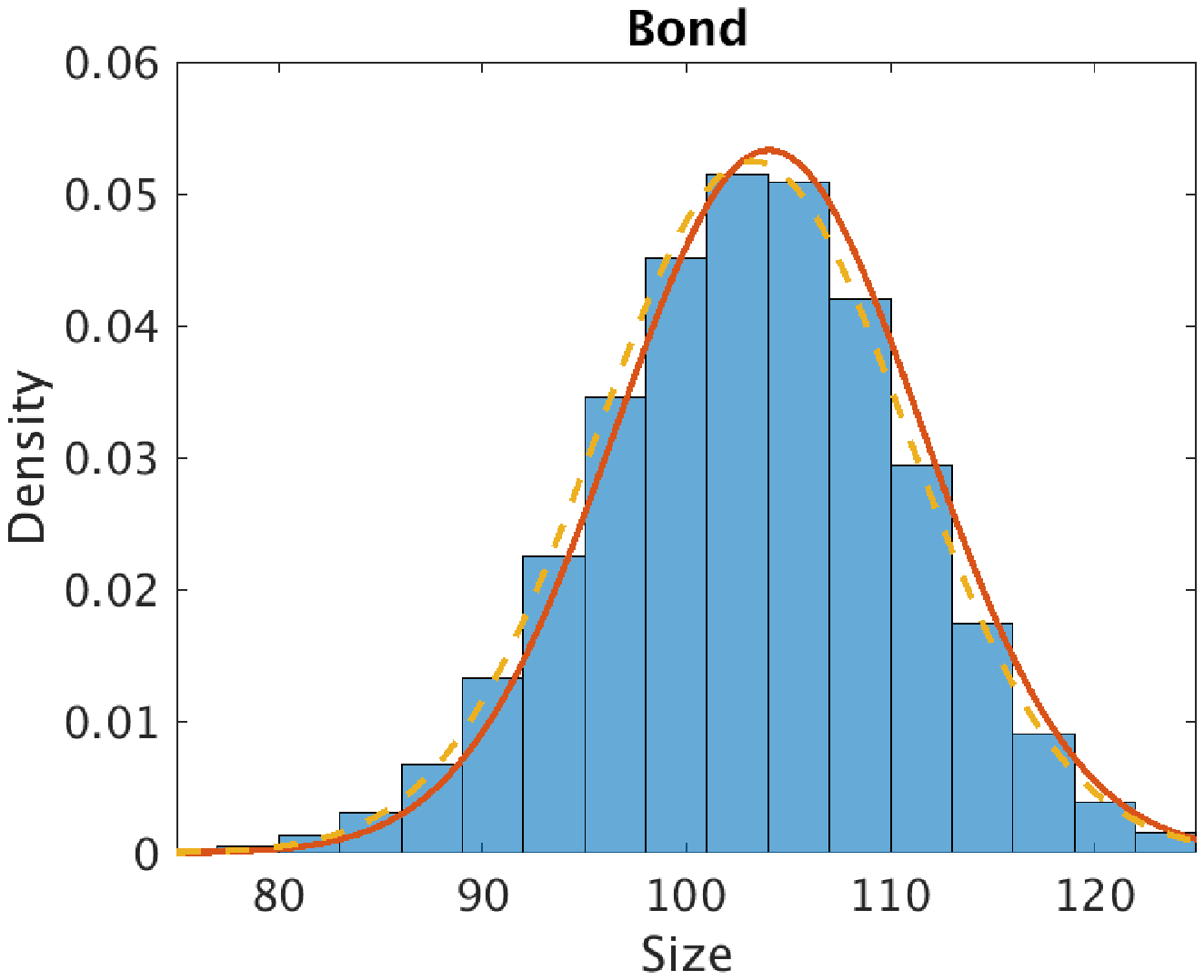}}
\resizebox{0.49\textwidth}{!}{\includegraphics{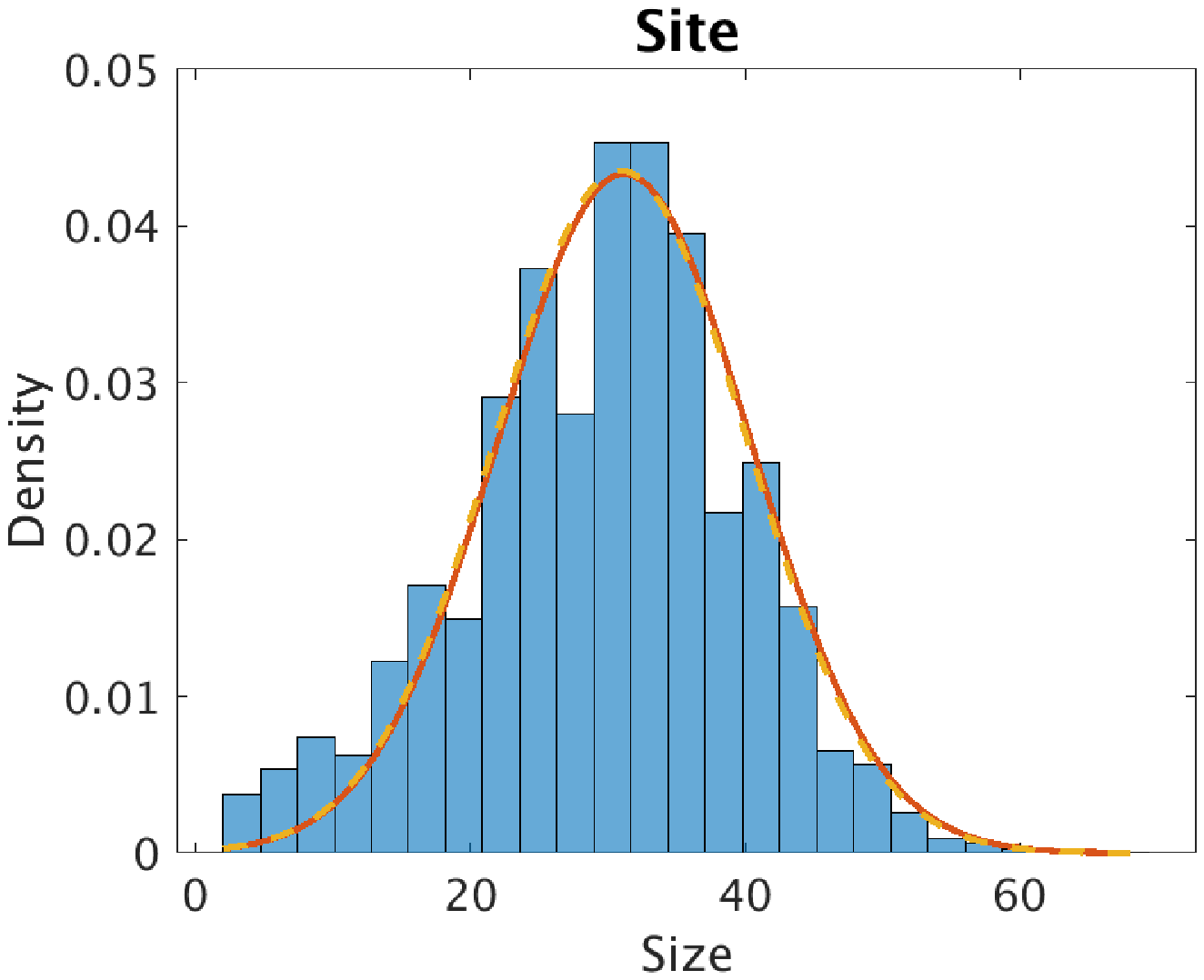}}
\end{center}
\caption{Histograms 100,000 simulations of size of largest component in percolation on
MR random graphs having $n=200, \pi=0.3$ and $D \sim {\rm Power}(1,13.796)$, with two asymptotic normal approximation superimposed; see text for details.}
\label{fig:perc3}
\end{figure}


\section{Density dependent population processes}
\label{sec:DPPP}
This section collects together some results for density dependent population processes
that are required in the paper.  It is based on~\cite{EK86}, Chapter 11, and~\cite{Pollett90}, though the statement of the functional central limit theorem is slightly more general than that in those references.  The notation is local to this section. 

For $n=1,2,\dots$, let $\{\bXn(t):t \ge 0\}=\{(\Xn_1(t),\Xn_2(t),\dots,\Xn_p(t)): t \ge 0\}$ be a continuous-time Markov
chain with state space $\Sn \subseteq \mathbb{Z}^p$ and transition intensities of the form 
\begin{equation}
\label{equ:ADDPP}
q^{(n)}(\bi,\bi+\bl)=n\beta^{(n)}_{\bl}(n^{-1}\bi)\qquad(\bi \in \Sn, \bl \in \Delta),
\end{equation}
where $\Delta$ is the set of possible jumps from a typical state $\bi=(i_1,i_2,\dots,i_p)$ and the $\beta^{(n)}_{\bl}:E \to \mathbb{R}$ are
nonnegative functions defined on an open set $E \subseteq \mathbb{R}^p$.  We assume that 
$\Delta$ is finite.  The theory in~\cite{EK86}, Chapter 11, and~\cite{Pollett90} allows 
$\Delta$ to be infinite but only the finite case is required in our application and the results are easier to state
in that setting.  Suppose that $\beta_{\bl}(\bx)=\lim_{n \to \infty}\beta^{(n)}_{\bl}(\bx)$ exists for all 
$\bl \in \Delta$ and all $\bx \in E$; the corresponding  family of processes is then called 
asymptotically density dependent by \cite{Pollett90}. 
In~\cite{EK86}, Chapter 11, a family of processes which satisfies~\eqref{equ:ADDPP}
with $\beta^{(n)}_{\bl}$ replaced by $\beta_{\bl}$ is called a density dependent family, and it is noted that
the results usually carry over with little additional effort to the more general form where
\[
q^{(n)}(\bi,\bi+\bl)=n\left[\beta_{\bl}(n^{-1}\bi)+O(n^{-1})\right] \qquad(\bi \in \Sn, \bl \in \Delta).
\] 

Let
\begin{equation}
\label{equ:driftF}
F(\bx)=\sum_{\bl \in \Delta}\bl \beta_{\bl}(\bx) \quad\mbox{and}\quad \Fn(\bx)=\sum_{\bl \in \Delta}\bl \beta_{\bl}^{(n)}(\bx)\qquad(\bx \in E).
\end{equation}
The following weak law of large numbers follows from~\cite{EK86}, Theorem 11.2.1, allowing for random initial conditions
and asymptotic density dependence; see also~\cite{Kurtz:1970}, Theorem 3.1 and~\cite{Pollett90}, Theorem 3.1.
\begin{theorem}
\label{thm:wlln}
Suppose that for each compact $K \in H$,
\begin{equation*}
\sum_{\bl \in \Delta} \sup_{\bx \in K} |\bl| \beta_{\bl}(\bx) <\infty,\qquad \lim_{n \to \infty}\sup_{\bx \in K}|\Fn(\bx)-F(\bx)|=0
\end{equation*}
and there exists $M_K>0$ such that 
\[
|F(\bx)-F(\by)|< M_K|\bx-\by| \qquad\mbox{for all } \bx,\by \in E.
\]
Suppose also that $n^{-1}\bXn(0)\convp \bx_0$ as $n \to \infty$.  Let $\bx(t)$ $(t \ge 0)$ be given by
\[
\bx(t)=\bx_0+\int_0^t F(\bx(u))\,{\rm d}u.
\]
Then, for all $t \ge 0$,
\[
\sup_{0 \le u \le t}|n^{-1}\bXn(u)-\bx(u)| \convp 0 \qquad \mbox{as }n \to \infty.
\]
\end{theorem}

Write $F(\bx)=\left(F_1(\bx),F_2(\bx),\dots, F_p(\bx)\right)$ and $\bx=(x_1,x_2,\dots,x_p)$.  
Let $\partial F(\bx)$ and $G(\bx)$ be the $p \times p$ matrix functions defined by
\[
\partial F(\bx)=\left[\dfrac{\partial F_i}{\partial x_j}(\bx)\right]
\qquad \mbox{and} \qquad
G(\bx)=\sum_{\bl \in \Delta} \bl^{\top}\bl \beta_{\bl}(\bx).
\]
Further, let $\Phi(t,u)=[\phi_{ij}(t,u)]$ $(0 \le u \le t<\infty)$ be the solution of the matrix differential
equation
\begin{equation*}
\dfrac{\partial}{\partial t}\Phi(t,u)=\partial F(\bx(t))\Phi(t,u),
\quad \Phi(u,u)=I.
\end{equation*}
Note that
\begin{equation*}
\phi_{ij}(t,u)=\dfrac{\partial x_i(t-u)}{\partial x_j(0)} \qquad(i,j=1,2,\dots,p).
\end{equation*}

The following functional central limit theorem follows from~\cite{EK86}, Theorem 11.2.3,
(see also~\cite{Kurtz:1970}, Theorem~ 3.5, and~\cite{Pollett90}, Theorem 3.2) allowing for asymptotically random initial conditions.  
Let $\Rightarrow$ denote weak convergence in the space of right-continuous functions
$f:[0,\infty) \to \mathbb{R}^p$ having limits from the left (i.e.~c\`{a}dl\`{a}g functions), 
endowed with the Skorohod metric.
\begin{theorem}
\label{thm:fclt}
Suppose that for each compact $K \in H$,
\begin{equation*}
\sum_{\bl \in \Delta}|\bl|^2 \sup_{\bx \in K} \beta_{\bl}(\bx) <\infty 
\qquad \mbox{and} \qquad
 \lim_{n \to \infty}\sqrt{n}\sup_{\bx \in K}|\Fn(\bx)-F(\bx)|=0,
\end{equation*}
and that  $\beta_{\bl}$ $(\bl \in \Delta)$ and $\partial F$ are continuous.  Suppose also that
\[
\sqrt{n}\left(n^{-1}\bXn(0)-\bx(0)\right)\convD \bV(0)\qquad \mbox{as } n \to \infty,
\]
where
$\bV(0) \sim N(\bzero,\Sigma_0)$.  Then, as $n \to \infty$, 
\begin{equation}
\label{equ:FCLT}
\left\{\sqrt{n}\left(n^{-1}\bXn(t)-\bx(t)\right): t \ge 0 \right\} \Rightarrow \{\bV(t):t \ge 0\},
\end{equation}
where $\{\bV(t):t \ge 0\}$ is a zero-mean
Gaussian process with covariance function given, for $t_1,t_2 \ge 0$, by
\begin{align}
\label{equ:covinit}
{\rm cov}\left(\bV(t_1), \bV(t_2)\right)=&\Phi(t_1,0)\Sigma_0 \Phi(t_2,0)^{\top}\\ 
&\qquad+ \int_0^{\min(t_1,t_2)}\Phi(t_1,u)
G(\bx(u))\Phi(t_2,u) ^{\top} \,{\rm d}u.\nonumber
\end{align}
\end{theorem}

We also require the following theorem concerned with the first exit of $\{\bXn(t):t \ge 0\}$ from a region of its state space. The theorem is a subset of~\cite{EK86}, Theorem 11.4.1,
expressed now for asymptotically density dependent processes having random initial conditions.  

\begin{theorem}
\label{thm:hitting}
Suppose that the conditions of Theorem~\ref{thm:fclt} are satisfied. Let
$\varphi:\mathbb{R}^p \to \mathbb{R}$ be continuously differentiable.  Suppose that
$\varphi(\bx(0))>0$.  Let $\taun=\inf\{t \ge 0: \varphi(n^{-1}\bXn(t)) \le 0\}$ and
$\tau=\inf\{t \ge 0: \varphi(\bx(t)) \le 0\}$.  Suppose that $\tau<\infty$ and $\nabla \varphi(\bx(\tau))\cdot F(\bx(\tau))<0$,
where $\cdot$ denotes inner vector product.  Then, as $n \to \infty$,
\begin{equation*}
\sqrt{n}\left(n^{-1}\bXn(\taun)-\bx(\tau)\right) \convD \bV(\tau)
-\frac{\nabla \varphi(\bx(\tau)\cdot \bV(\tau)}{\nabla \varphi(\bx(\tau)\cdot F(\bx(\tau))}F(\bx(\tau)).
\end{equation*}
\end{theorem}


\begin{cor}
\label{cor:hitting}
Suppose that the conditions of Theorem~\ref{thm:hitting} are satisfied.  Let
\begin{equation*}
B=I-\frac{F(\bx(\tau))\bigotimes \nabla \varphi(\bx(\tau))}{\nabla \varphi(\bx(\tau))\cdot F(\bx(\tau))},
\end{equation*}
where $\bigotimes$ denotes outer vector product.  Then,
\begin{equation*}
\sqrt{n}\left(n^{-1}\bXn(\taun)-\bx(\tau)\right) \convD {\rm N}(\bzero,B \Sigma(\tau) B^{\top}) \qquad \mbox{as }n \to \infty,
\end{equation*}
where
\begin{equation*}
\Sigma(\tau)=\Phi(\tau,0)\Sigma_0 \Phi(\tau,0)^{\top} + \int_0^{\tau}\Phi(\tau,u)
G(\bx(u))\Phi(\tau,u) ^{\top} \,{\rm d}u.
\end{equation*}
\end{cor}

\begin{proof}
Corollary~\ref{cor:hitting} follows immediately from Theorems~\ref{thm:fclt} and~\ref{thm:hitting} on
noting that
\[
\bV(\tau)
-\frac{\nabla \varphi(\bx(\tau)\cdot \bV(\tau)}{\nabla \varphi(\bx(\tau)\cdot F(\bx(\tau))}F(\bx(\tau))=\bV(\tau)B^{\top}.
\]
\end{proof}

\section{Proofs}
\label{sec:proofs}
\subsection{Alternative construction of final size $\Tn$}
\label{sec:altconstr}
We describe first the well-known construction of the final outcome of the epidemic
$\EEn$ using a random directed graph.  We then use that construction to show that
$\Tn$ can be realised using the location of the first exit of an asymptotically density dependent population
process from a given region.  

Given a realisation of $\GGn$, construct a directed random graph $\GGtn$, having vertex set
$\NNn=\{1,2,\dots,n\}$, as follows.
For each $i=1,2,\dots,n$, by sampling from its infectious period distribution $I$ and 
then the relevant Poisson processes, draw up a list of individuals $i$ would make 
contact with if $i$ were to become infected.  Then, for each ordered pair of individuals,
$(i,j)$ say, a directed edge from $i$ to $j$ is present in $\GGtn$ if and only if $j$ is in
$i$'s list.  Let $\IIn$ denote the set of initial infectives in $\EEn$. For distinct
$i, j \in \NNn$, write $i \leadsto j$ if and only if there is a chain of directed edges
from $i$ to $j$ in $\GGtn$.  Let $\TTn$ be the set of initial susceptibles that are infected in
$\EEn$.  Then $\TTn=\{j \in \NNn \setminus \IIn: i \leadsto j \mbox{ for some i }\in \IIn\}$
and the final size $\Tn$ of $\EEn$ is given by the cardinality of $\TTn$.

Note that $\TTn$ is determined purely by the presence/absence of directed edges in $\GGtn$
and does not depend on the times of the corresponding infections in $\EEn$.  (This implies that
the distribution of $\TTn$, and hence also $\Tn$ is invariant to the introduction of a
latent/exposed period into the model, i.e.~the time elapsing after infection of an individual before
it is able to infect other individuals.)
It follows that
the final outcome $\TTn$ has the same distribution as that of a related epidemic $\EEtn$,
with set of initial infectives $\IIn$, in which for any infective, $i$ say, it is determined
upon infection which, if any, of its neighbours $i$ will contact and those contacts take place
at the first points of independent Poisson processes, each having rate $1$.  More precisely,
suppose $i$ is infected at time $t_0$ and $i$ has $d$ neighbours, $i_1,i_2,\dots,i_d$ say.
Let $I_i$ be a realisation of $I$ and, given $I_i$, let $\chi_{i1}, \chi_{i2},\dots, \chi_{id}$ be 
i.i.d.~Bernoulli random variables with success probability $1-\exp(-\lambda I_i)$.  Let
$W_{i1},W_{i2},\dots,W_{id}$ be an independent set of i.i.d.~unit-mean exponential random random variables.  
Then $i$ contacts $i_j$ if and only if $\chi_{ij}=1$, and in that case the contact occurs at time 
$t_0+W_{ij}$.  Of course, the $I$ and $W$ random variables for any set of distinct infectives are mutually independent.

Given the degrees $\Dn_1,\Dn_2,\dots,\Dn_n$ and the set of initial infectives $\IIn$, 
the random graph $\GGn$ and the epidemic on it $\EEtn$ can be constructed simultaneously as follows
(cf.~\cite{BN:2008}).  The process starts with no half-edge paired. The individuals in $\IIn$ become infected at time $t=0$ and all other individuals are
susceptible.  When an individual is infected, it determines immediately which, if any, of its half-edges it
will transmit infection along and when it will make those contacts, according to the probabilistic law described
above; the half-edges that the indivdiual will infect along and not infect along are then called infective and recovered half-edges,
respectively.  When infection is transmitted along a half-edge that half-edge is paired with a half-edge chosen uniformly
at random from all unpaired half-edges, forming an edge in the network.  If the chosen half-edge is attached to a susceptible individual then that 
individual is infected and determines immediately along which, if any, of its remaining half-edges it will transmit infection.  If the 
chosen half-edge is infective or recovered then nothing happens, apart from the two half-edges being replaced by an edge.
The process continues until there is no infective half-edge remaining.  (In the epidemic on the NSW random graph, if 
$\Dn_1+\Dn_2+\dots+\Dn_n$ is odd then it is possible for the process to end with one unpaired half-edge, which is
infective, but we can ignore that possibility because under the conditions of the theorems its probability
tends to $0$ as $n \to \infty$.)

Note that as we are interested only in the final outcome of the epidemic, it is not necessary to keep
track of the degrees of individuals to which infective and recovered half-edges are attached; we
need to know just the total numbers of such half-edges.  For $t \ge 0$, let $\Xn_i(t)$ be the number 
of degree-$i$ susceptible individuals at time $t$ $(i=0,1,\dots,\dmax)$, and let $\Yne(t)$ and $\Zne(t)$ be the
total numbers of infective and recovered half-edges, respectively, at time $t$. Let $\{\bWn(t)\}=\{\bWn(t):t \ge 0\}$,
where $\bWn(t)=(\bXn(t), \Yne(t), \Zne(t))$ and $\bXn(t)=(\Xn_0(t), \Xn_1(t),\dots,\Xn_{\dmax}(t))$.

The process $\{\bWn(t)\}$ is a continuous-time Markov chain, whose initial state  $\bWn(0)$ is
random, even for the epidemic on the MR random graph as the numbers of infective and recovered
half-edges created by the initial infectives are random.  In the epidemic on the NSW random graph,
$\Xn_i(t)$ $(i=0,1,\dots,\dmax)$ are also random.  Before giving the possible transition intensities of $\{\bWn(t)\}$
some more notation is required.

Let $\Hn$ denote the state space of  $\{\bWn(t)\}$ and $\bn=(n^X_0,n^X_1,\dots,n^X_{\dmax},n^Y_E,n^Z_E)$ 
denote a typical element of $\Hn$.  Let $n^X_E=\sum_{i=1}^{\dmax} i n^X_i$.  Define the unit vectors $\bes_i$ $(i=0,1,\dots,\dmax)$, $\bei$ and
$\ber$ on $\Hn$, where, for example, $\bes_i$ has a one in the $i$th suscpetible component and zeros
elsewhere.  For $i=1,2,\dots,\dmax$ and $k=0,1,\dots, i-1$, let $p_{i,k}$ be the probability that if a degree-$i$
susceptible is infected it subsequently transmits infection along $k$ of its remaining $i-1$ half-edges.  
(Note that a degree-$0$ susceptible cannot be infected.) 
Thus
$p_{i,k}={\rm P}(X=k)$, where $X\sim {\rm Bin}(i-1,1-\exp(-\lambda I))$.  The transition intensities of $\{\bWn(t)\}$
are as follows.

\begin{enumerate}
\item[(i)]
For $i=1,2,\dots,\dmax$ and $k=0,1,\dots, i-1$, an infective half-edge is paired with a degree-$i$ 
susceptible yielding $k$ infective half-edges and $i-1-k$ recovered half-edges
\[
\qn(\bn, \bn-\bes_i+(k-1)\bei+(i-k-1)\ber)=\frac{n^Y_E i n^X_i p_{i,k}}{n^X_E+n^Y_E+n^Z_E-1};
\]
\item[(ii)]
an infective half-edge is paired with an infective half-edge
\[
\qn(\bn, \bn-2\bei)=\frac{n^Y_E(n^Y_E-1)}{n^X_E+n^Y_E+n^Z_E-1};
\]
\item[(iii)]
an infective half-edge is paired with a recovered half-edge
\[
\qn(\bn, \bn-\bei-\ber)=\frac{n^Y_E n^Z_E}{n^X_E+n^Y_E+n^Z_E-1}.
\]
\end{enumerate}
Note that these transition intensities are independent of the population size $n$.  We index then by $n$ to
connect with theory of density dependent population  processes.

Let $\taun=\inf\{t \ge 0: \Yne(t)=0\}$.  Then $\Xn_i(\taun)$ $(i=0,1,\dots,\dmax)$ give the 
numbers of susceptibles of the different degrees at the end of the epidemic and
$\Tn=\sum_{i=1}^{\dmax}(\Xn_i(0)-\Xn_i(\taun)$. 
\subsection{Random time-scale transformation}
\label{sec:timetrans}

 We wish to apply Corollary~\ref{cor:hitting} to obtain a central limit theorem for $\Tn$ but that corollary cannot be applied directly to
$\{\bWn(t)\}$ as $\taun \convp \infty$ as $n \to \infty$.  Thus we consider the following
random time-scale transformation of $\{\bWn(t)\}$; cf.~\cite{EK86}, page 467, \cite {Watson80} and~\cite{JLW:2014}.

For $t \in [0, \taun]$, let
\[
\An(t)=\int_0^t \frac{\Yne(u)}{\Xne(u)+\Yne(u)+\Zne(u)-1} \,{\rm d}u,
\]
where $\Xne(u)=\sum_{i=1}^{\dmax} i \Xn_i(u)$. Let $\tautn=\An(\taun)$. For $t \in [0,\tautn]$, let $\Un(t)=\inf\{u \ge 0:\An(u)=t\}$ and
$\bWnt(t)=\bWn\left(\Un(t)\right)$.  
Then $\{\bWnt(t)\}=\{\bWnt(t): 0 \le t \le \tautn\}$ is a continuous-time Markov chain with transition intensities:
\begin{enumerate}
\item[(i)]
for $i=1,2,\dots,\dmax$ and $k=0,1,\dots, i-1$, 
\[
\qtn(\bn, \bn-\bes_i+(k-1)\bei+(i-k-1)\ber)=i n^X_i p_{i,k};
\]
\item[(ii)]
\[
\qtn(\bn, \bn-2\bei)=n^Y_E-1;
\]
\item[(iii)]
\[
\qtn(\bn, \bn-\bei-\ber)=n^Z_E.
\]
\end{enumerate}
For $t \ge 0$, let $\bXnt(t)=(\Xnt_0(t), \Xnt_1(t),\dots,\Xnt_{\dmax}(t))$, so
$\bWnt(t)=(\bXnt(t), \Ynet(t), \Znet(t))$.

For reasons that will become clear later, we need to extend $\{\bWnt(t)\}$ so that it is defined beyond time
$\tautn$ and allow $\Ynet(t)$ and $\Znet(t)$ to take negative values.  Thus we add two new transition intensities:
\begin{enumerate}
\item[(ii$'$)]
for $n^Y_E<0$,
\[
\qtn(\bn, \bn+2\bei)=-(n^Y_E+1);
\]
\item[(iii$'$)]
for $n^Z_E<0$,
\[
\qtn(\bn, \bn+\bei+\ber)=-n^Z_E.
\]
\end{enumerate}
Note that the transition intensities (ii) and (iii) are defined only for $n^Y_E>0$ and $n^Z_E>0$,
respectively.

The state space of $\{\bWnt(t)\}$ is a subset of 
\[
\Hnt=\left([0, n]^{\dmax+1}\times (-\infty,n\dmax]^2\right) \cap \mathbb{Z}^{\dmax+3}.
\]  
Let $\bl_{ik}^{(1)}=-\bes_i+(k-1)\bei+(i-k-1)\ber$ $(i=1,2,\dots,\dmax;k=0,1,\dots,i-1)$, $\bl^{(2)}_+=-2\bei$,
$\bl^{(2)}_-=2\bei$, $\bl^{(3)}_+=-\bei-\ber$ and $\bl^{(3)}_+=\bei+\ber$.  The set of possible jumps of
$\{\bWnt(t)\}$ from a typical state $\bn$ is $\Delta=\Delta_1 \cup \Delta_2 \cup \Delta_3$, where
$\Delta_1=\{\bl^{(1)}_{ik}:i=1,2,\dots,\dmax;k=0,1,\dots,i-1\}$, $\Delta_2=\{\bl^{(2)}_+,\bl^{(2)}_-\}$
and $\Delta_3=\{\bl^{(3)}_+,\bl^{(3)}_-\}$.  
Let $\bw=(\bx,y_E,z_E)$, where $\bx=(x_0,x_1,\dots,x_{\dmax})$, and $H=[0,1]^{\dmax+1}\times (-\infty,\dmax]^2$.
The intensities of the jumps of $\{\bWnt(t)\}$ admit the
form
\begin{equation}
\label{equ:aDDPP}
\qtn(\bn,\bn+\bl)=n\betat_{\bl}^{(n)}(n^{-1}\bn)\qquad(\bn \in \Hnt, \bl \in \Delta),
\end{equation}
where the functions $\betat_{\bl}^{(n)}:H \to \mathbb{R}_+$ $(\bl \in \Delta)$ are given by
\begin{equation}
\label{equ:betanltilde}
\betat_{\bl}^{(n)}(\bw)=\begin{cases}
              \betat_{\bl}(\bw) & \text{ if } \bl \in \Delta \setminus \{\bl^{(2)}_+,\bl^{(2)}_-\},\\
              \betat_{\bl}(\bw)-n^{-1} &\text{ if } \bl=\bl^{(2)}_+,\\
              \betat_{\bl}(\bw)+n^{-1} &\text{ if } \bl=\bl^{(2)}_-,
              \end{cases}
\end{equation}
with
\begin{equation}
\label{equ:intensityfun}
\betat_{\bl}(\bw) = \begin{cases}
	      \betat_{ik}^{(1)}(\bx,y_E,z_E)=ix_i p_{i,k}& \text{ for } \bl=\bl_{ij}^{(1)} \in \Delta_1, \\
	      \betat^{(2)}_+(\bx,y_E,z_E)=y_E 1_{\{y_E>0\}}& \text{ for } \bl=\bl^{(2)}_+,\\
	      \betat^{(2)}_-(\bx,y_E,z_E)=-y_E 1_{\{y_E<0\}}& \text{ for } \bl=\bl^{(2)}_-,\\
	      \betat^{(3)}_+(\bx,y_E,z_E)=z_E 1_{\{z_E>0\}}& \text{ for } \bl=\bl^{(3)}_+,\\
	      \betat^{(3)}_-(\bx,y_E,z_E)=-z_E 1_{\{z_E<0\}}& \text{ for } \bl=\bl^{(3)}_-.	      
\end{cases}
\end{equation}
The family of processes $\{\bWnt(t)\}$ is asymptotically density dependent (see Section~\ref{sec:DPPP}).

\subsection{Proof of Theorem~\ref{thm:posclt}}
\label{sec:proofposclt}
Note that $\bWn(\taun)=\bWnt(\tautn)$, so the final size of the epidemic is given by
\begin{equation}
\label{equ:Tn}
\Tn=\sum_{i=0}^{\dmax}\left(\Xnt_i(0)-\Xnt_i(\tautn)\right).
\end{equation}
Note also that $\tautn=\inf\{t \ge 0:\Ynet(t)=0\}$.
We use Corollary~\ref{cor:hitting}
to obtain a central limit theorem for $\bWnt(\tautn)$, and hence for
$\Tn$.  The asymptotic variance matrix in the central limit theorem for $\bWnt(\tautn)$ is not in
closed form.  However, we derive a closed-form expression for the asymptotic variance of $\Tn$. The main concepts of the proof are presented here, with some detailed but elementary calculations
deferred to Appendix~\ref{app:details}.
\subsubsection{Deterministic model}
\label{section:detmodel}
As at~\eqref{equ:driftF}, define the drift function $\tilde{F}(\bw)=\sum_{\bl \in \Delta}\bl \betat_{\bl}(\bw)$.
Using~\eqref{equ:intensityfun},
\begin{align}
\label{equ:driftFw}
\tilde{F}(\bw)&=-\sum_{i=1}^{\dmax} i x_i \bes_i+\left\{\sum_{i=1}^{\dmax} i x_i[(i-1)p_I-1] -2y_E-z_E\right\}\bei\\
&\qquad\qquad+\left[\sum_{i=1}^{\dmax} i(i-1)x_i q_I-z_E\right]\ber.\nonumber
\end{align}
For $t \ge 0$, let $\bwt(t)=(\xt_0(t),\xt_1(t),\dots,\xt_{\dmax}(t),\yet(t),\zet(t))$ be defined by
\begin{equation}
\label{equ:bwt}
\bwt(t)=\bwt(0)+\int_0^t \tilde{F}(\bw(u)) \,{\rm d}u.
\end{equation}
Thus $\bwt(t)$ satisfies the differential equations
\begin{align}
\dfrac{d\xt_i}{dt}&=-i \xt_i \quad (i=0,1,\dots,\dmax), \label{equ:dxtidt}\\
\dfrac{d\yet}{dt}&=\sum_{i=2}^{\dmax} i(i-1)p_I \xt_i-\xet-2\yet-\zet,\label{equ:dyetdt}\\
\dfrac{d\zet}{dt}&= \sum_{i=2}^{\dmax} i(i-1)q_I\xt_i-\zet,\label{equ:dzetdt}
\end{align}
where $\xet=\sum_{i=1}^{\dmax}i\xt_i$, having solution (see Appendix~\ref{app:det})
\begin{align}
\xt_i(t)&=\xt_i(0)\re^{-it}\quad (i=0,1,\dots,\dmax),\label{equ:xtilde}\\
\yet(t)&=(\xet(0)+\yet(0)+\zet(0))\re^{-2t}-[\zet(0)+q_I\xet(0)]\re^{-t}\label{equ:yetilde}\\
&\qquad-p_I\sum_{i=1}^{\dmax}i\xt_i(0)\re^{-it},\nonumber\\
\zet(t)&=[\zet(0)+q_I\xet(0)]\re^{-t}-q_I \sum_{i=1}^{\dmax}i\xt_i(0)\re^{-it}.\label{equ:zetilde}
\end{align}
Let $\etat(t)=\xet(t)+\yet(t)+\zet(t)$.  
Note that, for $t \ge 0$,
\begin{equation}
\label{equ:eta}
\etat(t)=\etat(0)\re^{-2t}.
\end{equation}

\subsubsection{Initial conditions}
\label{section:initcond}

Consider first the epidemic on the MR random graph. Recall that in the epidemic $\EEn$, for $i=0,1,\dots,\dmax$, there are
$\vn_i$ individuals of degree $i$, of whom $\an_i$ are initially infected.  Also, $\left(\Ynet(0), \Znet(0)\right)$ is given by
the total numbers of infective and recovered half-edges created by the initial infectives.
Thus $\Xnt_i(0)=\vn_i-\an_i$ ($i=0,1,\dots,\dmax$) and
\[
\left(\Ynet(0), \Znet(0)\right)=\sum_{i=1}^{\dmax}\sum_{j=1}^{\an_i}
\left(Y_{ij},i-Y_{ij}\right),
\]
where $Y_{ij}$ $(i=1,2,\dots,\dmax;j=1,2,\dots)$ are independent and
$Y_{ij} \sim {\rm Bin}(i,1-\exp(-\lambda I))$. Let $\epsilon_E=\sum_{i=1}^{\dmax} i \epsilon_i$.  Then, recalling that
$\epsilonn_i=n^{-1} \an_i$,
\begin{align*}
\sqrt{n}&\left[n^{-1}\left(\Ynet(0), \Znet(0)\right)-\epsilon_E(p_I, q_I)\right]\\
&=\frac{1}{\sqrt{n}}\sum_{i=1}^{\dmax}  \left[\sum_{j=1}^{\an_i}\left(Y_{ij},i-Y_{ij}\right)-\an_i i (p_I,q_I)\right]
+\sqrt{n}\sum_{i=1}^{\dmax} i (\epsilonn_i-\epsilon_i).
\end{align*}

For $i=1,2,\dots,\dmax$, let $\sigma_{Y,i}^2={\rm var}(Y_{i1})$.  Now $\lim_{n \to \infty}\sqrt{n}(\epsilonn_i-\epsilon_i) = 0$ $(i=0,1,\dots,\dmax)$, by assumption,
so the central limit theorem and Slutsky's theorem imply that
\[
\sqrt{n}\left[n^{-1}\left(\Ynet(0), \Znet(0)\right)-\epsilon_E(p_I, q_I)\right]
\convD {\rm N}\left(\bzero, \begin{bmatrix}
\sigma_Y^2 & -\sigma_Y^2 \\
-\sigma_Y^2 &  \sigma_Y^2
\end{bmatrix}\right)
\]
as $n \to \infty$, where
\begin{equation}
\label{equ:sigma2Y}
\sigma_Y^2=\sum_{i=1}^{\dmax} \epsilon_i \sigma_{Y,i}^2.
\end{equation}
(A closed-form expression for $\sigma_Y^2$ is given by~\eqref{equ:sigmaY2} in Appendix~\ref{app:sigma0MR2}.)
Since $\bXnt(0)$ is non-random, it follows using~\eqref{equ:vni} that
\begin{equation}
\label{equ:MRinitCLT}
\sqrt{n}\left(n^{-1}\bWnt(0)-\bwt(0)\right) \convD {\rm N}\left(\bzero,\SigmaMR_0\right) \qquad \mbox{as }
n \to \infty,
\end{equation}
where
\begin{equation}
\label{equ:MRdetinit}
\bwt(0)=(p_1-\epsilon_1, p_2-\epsilon_2,\cdots,p_{\dmax}-\epsilon_{\dmax},\epsilon_E p_I,\epsilon_E q_I)
\end{equation}
and
\begin{equation}
\label{equ:SigmaMR0}
\SigmaMR_0=
\begin{bmatrix}
0 & \bzero & \bzero\\
\bzero & \sigma_Y^2 & -\sigma_Y^2 \\
\bzero & -\sigma_Y^2 & \sigma_Y^2
\end{bmatrix}.
\end{equation}

Turning to the epidemic on the NSW random graph, recall that in $\EEn$ the number of initial infectives
$\an$ is prescribed and the $\an$ initial infectives are chosen by sampling uniformly at random without replacement from the 
population.  Thus the network can be constructed using two independent sets of i.i.d.~copies of $D$,
viz.~$D_1',D_2',\dots,D_{n-\an}'$ for the initial susceptibles and $D_1,D_2,\dots,D_{\an}$ for the initial 
infectives.  Let $(Y_E,Z_E)$ be the bivariate random variable obtained by first sampling $D$ and then
letting $(Y_E,Z_E)=(Y_E,D-Y_E)$, where $Y_E|D~\sim{\rm Bin}(D,1-\exp(-\lambda I))$.  Let
$\sigma_{Y_E}^2={\rm var}(Y_E)$, $\sigma_{Y_E, Z_E}={\rm cov}(Y_E, Z_E)$ and $\sigma_{Y_E}^2={\rm var}(Y_E)$. 
(Closed-form expressions for $\sigma_{Y_E}^2$, $\sigma_{Z_E}^2$ and $\sigma_{Y_E, Z_E}$ are given 
in~\eqref{equ:sigmaYe2}-\eqref{equ:sigmaYeZe2} in Appendix~\ref{app:sigma0NSW2}.) 
Let
$\bp=(p_0,p_1,\dots,p_{\dmax})$ and $\Sigma_{XX}$ be the $(\dmax+1) \times (\dmax+1) $ matrix with elements
\begin{equation}
\label{equ:SigmaXX}
(\Sigma_{XX})_{ij}=\begin{cases}
	      -p_ip_j& \text{ if } i \ne j, \\
	      p_i(1-p_i) & \text{ if } i=j.
\end{cases}
\end{equation}

Recalling that $\lim_{n \to \infty}\sqrt{n}(\epsilonn-\epsilon)=0$, where $\epsilonn=n^{-1}a^{(n)}$,
a similar argument to the above shows that 
\begin{equation}
\label{equ:NSWinitCLT}
\sqrt{n}\left(n^{-1}\bWnt(0)-\bwt(0)\right) \convD {\rm N}\left(\bzero,\SigmaNSW_0\right) \qquad \mbox{as }
n \to \infty,
\end{equation}
where
\begin{equation}
\label{equ:NSWdetinit}
\bwt(0)=\left((1-\epsilon) \bp,\epsilon \mud p_I,\epsilon \mud q_I\right)
\end{equation}
and
\begin{equation}
\label{equ:SigmaNSW0}
\SigmaNSW_0=
\begin{bmatrix}
(1-\epsilon)\Sigma_{XX}& \bzero & \bzero\\
\bzero & \epsilon \sigma_{Y_E}^2 & \epsilon \sigma_{Y_E, Z_E} \\
\bzero & \epsilon\sigma_{Y_E, Z_E} & \epsilon \sigma_{Z_E}^2
\end{bmatrix}.
\end{equation}

\subsubsection{Central limit theorem}
\label{section:clt}
Noting from~\eqref{equ:betanltilde} that 
\[
\max_{\bl \in \Delta}\sup_{\bw \in H}|\betat_{\bl}^{(n)}(\bw)-\betat_{\bl}(\bw)|=n^{-1},
\]
it is easily checked that $\{\bWnt(t)\}$ satisfies the conditions of Theorem~\ref{equ:FCLT}.  Now
$\tautn=\inf\{t \ge 0: \varphi(n^{-1}\bWnt(t))\le 0\}$, where $\varphi(\bw)=\varphi(\bx,y_E,z_E)=y_E$.

Suppose that $\xt_i(0)= p_i- \epsilon_i$ $(i=0,1,\dots,\dmax)$, $\yet(0)=p_I \sum_{i=1}^{\dmax}i \epsilon_i$ and
$\zet(0)=q_I \sum_{i=1}^{\dmax}i \epsilon_i$, so $\etat(0)=\mud$.  Let $\taut=\inf\{t \ge 0:\varphi(\bwt(t)) \le 0\}=\inf\{t \ge 0:\yet(t)=0\}$.  Then it
follows from~\eqref{equ:yetilde} that $\taut$ satisfies the equation
\begin{equation}
\label{equ:taut}
\re^{-\taut}-q_I-\mud^{-1}p_I\fde^{(1)}(\re^{-\taut})=0.
\end{equation}
We show in Appendix~\ref{app:tautprop} that, under the conditions of Theorem~\ref{thm:posclt}, the 
equation~\eqref{equ:taut} has a unique solution in $(0,\infty)$.  
Note that $z=\re^{-\taut}$, where $z$ is defined at~\eqref{equ:z}.  Also, using~\eqref{equ:xtilde} the 
deterministic final fraction of the population that is susceptible is given by
\[
\sum_{i=0}^{\dmax} \xt_i(\taut)=\sum_{i=0}^{\dmax} (p_i-\epsilon_i) \re^{-i \taut}=\fde(\re^{-\taut}).
\]
The corresponding deterministic final size is $\rho=1-\epsilon-\fde(\re^{-\taut})$, agreeing with~\eqref{equ:rho}.

Let $a(\taut)=\nabla \varphi(\bw(\taut))\cdot \tilde{F}(\bw(\taut))$.  Then, using~\eqref{equ:driftFw}, and noting that
$\yet(\taut)=0$,
\begin{eqnarray*}
a(\taut)&=&\sum_{i=1}^{\dmax}i \xt_i(\taut)[(i-1)p_I-1]-2 \yet(\taut) -\zet(\taut)\\
&=&p_I\sum_{i=1}^{\dmax}i(i-1)\xt_i(\taut)-\etat(\taut).
\end{eqnarray*}
Thus, using~\eqref{equ:xtilde} and~\eqref{equ:eta},
\begin{equation}
\label{equ:ataut}
a(\taut)=\re^{-2\taut}\left(p_I\fde^{(2)}(\re^{-\taut})-\mud\right),
\end{equation}
since $\etat(0)=\mud$.
We show in Appendix~\ref{app:tautprop} that $a(\taut)< 0$, so we may apply Corollary~\ref{cor:hitting}.

Writing $\bwt=(\tilde{w}_0,\tilde{w}_1,\dots,\tilde{w}_p)$, where $p=\dmax+2$, let $\Phit(t,u)=[\tilde{\phi}_{ij}(t,u)]$
$(0 \le u \le t < \infty)$, where
\begin{equation}
\label{equ:phit}
\tilde{\phi}_{ij}(t,u)=\dfrac{\partial \tilde{w}_i(t-u)}{\partial \tilde{w}_j(0)} \qquad(i,j=0,1,\dots,p).
\end{equation}
Also, let
\begin{equation}
\label{equ:Sigmat}
\Sigmat(\taut)=\Phit(\taut,0)\Sigma_0 \Phit(\taut,0)^{\top}+\int_0^{\taut}\Phit(\taut,u)
\tilde{G}(\bwt(u))\Phit(\taut,u) ^{\top} \,{\rm d}u,
\end{equation}
where
\begin{equation}
\label{equ:Gtilde}
\tilde{G}(\bwt(u))=\sum_{\bl \in \Delta} \bl^{\top}\bl \betat_{\bl}(\bwt(u)),
\end{equation}
and $\Sigma_0=\SigmaMR_0$ or $\SigmaNSW_0$ dependening on whether the epidemic is on
an MR or an NSW random graph.  Then application of Corollary~\ref{cor:hitting} yields
\begin{equation}
\label{equ:Wntildeclt}
\sqrt{n}\left(n^{-1}\bWnt(\tautn)-\bwt(\taut)\right) \convD {\rm N}(\bzero,B \Sigmat(\taut) B^{\top})  \qquad \mbox{as }n \to \infty,
\end{equation}
where
\[
B=I-\frac{\tilde{F}(\bwt(\taut))\bigotimes \nabla \varphi(\bwt(\taut))}{\nabla \varphi(\bwt(\taut))\cdot \tilde{F}(\bwt(\taut))}.
\]
Thus, recalling~\eqref{equ:Tn},
\[
\sqrt{n}\left(n^{-1}\Tn-\rho\right) \convD {\rm N}(0, \sigma^2)\qquad \mbox{as }n \to \infty,
\]
where
\begin{equation}
\label{equ:sigma2}
\sigma^2=(\bone,0,0)B\Sigmat(\taut)B^{\top}(\bone,0,0)^{\top}.
\end{equation}

Now $\nabla \varphi(\bwt(\taut))=(\bzero,1,0)$ and, using~\eqref{equ:driftFw}, $(\bone,0,0)[\tilde{F}(\bwt(\taut))]^{\top}=-\xet(\taut)$, 
so $(\bone,0,0)B=(\bone, b(\taut),0)$, where 
\begin{equation}
\label{equ:btaut}
b(\taut)=a(\taut)^{-1}\xet(\taut).
\end{equation}
Further, it follows 
from~\eqref{equ:xtilde}-\eqref{equ:zetilde} and~\eqref{equ:phit} that
\begin{equation*}
\left[(\bone,0,0)\Phit(\taut,u)\right]_i=\re^{-i(\taut-u)}\qquad (i=0,1\dots,\dmax)
\end{equation*}
and
\begin{equation}
\label{equ:bzero10phit}
\left[(\bzero,1,0)\Phit(\taut,u)\right]_i=
\begin{cases} 
i\left[\re^{-2(\taut-u)}-q_I\re^{-(\taut-u)}-p_I\re^{-i(\taut-u)}\right]& \text{ if } i=0,1,\dots,\dmax, \\
	      \re^{-2(\taut-u)} & \text{ if } i=\dmax+1,\\
	      -\re^{-(\taut-u)}(1-\re^{-(\taut-u)}) & \text{ if } i=\dmax+2.
\end{cases}
\end{equation}
Further, let
\begin{eqnarray}
\label{equ:bc}
\bc(\taut,u)&=&(\bone,0,0)B\Phit(\taut,u)\nonumber\\
&=&\left(\bc_S(\taut,u),c_I(\taut,u),c_R(\taut,u)\right),
\end{eqnarray}
where 
\begin{align}
\bc_S(\taut,u)&=\left(c_0(\taut,u), c_1(\taut,u),\dots,c_{\dmax}(\taut,u)\right),\label{equ:bcs}\\
c_I(\taut,u)&=b(\taut)\re^{-2(\taut-u)},\label{equ:cI}\\
c_R(\taut,u)&=-b(\taut)\re^{-(\taut-u)}(1-\re^{-(\taut-u)})\label{equ:cR}
\end{align}
and, for $i=0,1,\dots,\dmax$, 
\begin{equation}
\label{equ:ci}
c_i(\taut,u)=\re^{-i(\taut-u)}+b(\taut)i\left[\re^{-2(\taut-u)}-p_I \re^{-i(\taut-u)}-q_I \re^{-(\taut-u)}\right].
\end{equation}
Noting that $\bc(\taut,u)\bl^{\top}$ is a scalar, it follows from~\eqref{equ:Sigmat} and~\eqref{equ:sigma2} that
\begin{equation}
\label{equ:sigma2a}
\sigma^2=\bc(\taut,0)\Sigma_0 \bc(\taut,0)^{\top}+\sum_{\bl \in \Delta}\int_0^{\taut} \left(\bc(\taut,u)\bl^{\top}\right)^2
\betat_{\bl}(\bwt(u))\,{\rm d}u.
\end{equation}

The asymptotic variances, $\sigmaMR$ and $\sigmaNSW$ in Theorem~\ref{thm:posclt} can be obtained by
substituting $\Sigma_0=\SigmaMR_0$ and $\Sigma_0=\SigmaNSW_0$, respectively, in~\eqref{equ:sigma2a}
and using~\eqref{equ:bc}-\eqref{equ:ci}, together with~\eqref{equ:intensityfun} and~\eqref{equ:xtilde}-\eqref{equ:zetilde}, to evaluate the second term on the
right-hand side of~\eqref{equ:sigma2a}.  The details are lengthy and may be found in Appendix~\ref{app:asymvar}.

\subsection{Proof of Theorem~\ref{thm:majclt}}
\label{sec:proofmajclt}
We prove Theorem~\ref{thm:majclt} for the epidemic on an NSW random graph.  The proof for the epidemic on an MR random graph 
is similar but simpler, as there is no randomness in the degrees of individuals, and is thus omitted.
The proof proceeds in two stages.  First, in Section~\ref{sec:coupling}, we couple the early stages of the epidemic
$\EEtn$, defined in Section~\ref{sec:altconstr}, to a two-type branching process $\BBtn$ which assumes that all infective half-edges in
$\EEtn$ are paired with susceptible half-edges. The branching processes $\BBtn$ $(n=1,2,\dots)$ are coupled to a limiting branching
process $\BBt$.  The couplings and standard properties of the limiting branching process $\BBt$ show that, with probability tending
to $1$ as $n \to \infty$, a major outbreak occurs if and only if the branching process $\BBt$ does not go extinct, and yield weak convergence results concerning the composition of the population in $\EEtn$ when the number of infective half-edges
first reaches $\log n$ in the event of a major outbreak (see Theorem~\ref{thm:BPapprox}).  Then, in Section~\ref{sec:epilogn},
we use the random time-scale transformation introduced in Section~\ref{sec:timetrans} to determine the asymptotic distribution
of the final size of a major outbreak.  The argument proceeds as in the proof of Theorem~\ref{thm:posclt} but the equation
defining $\taut$ now has a solution at $0$ and one at $\taut>0$ (see the discussion following~\eqref{equ:taut1}) and a lower
bounding branching process for the epidemic $\EEtn$ is used to show that $\taut>0$ is the relevant asymptotic hitting time. 

For ease of presentation we assume that, for $n=1,2,\dots$, there is one initial infective in $\EEtn$ (i.e.~that $a=1$), who
is chosen by sampling a half-edge uniformly at random from all $\Dn_1+\Dn_2+\dots+\Dn_n$ half-edges and infecting the individual who owns
that half-edge. The proofs are easily extended to $a>1$ and other ways of choosing the initial infective(s) but the details are more
complicated.

\subsubsection{Coupling of epidemic and branching processes}
\label{sec:coupling}
Let $(\Omega, \mathcal{F}, {\rm P})$  be a probability space on which are defined the following independent
sets of random varaibles:
\begin{enumerate}
\item[(i)] $D_1,D_2,\dots$ i.i.d. $\sim D$;
\item[(ii)] $U_0,U_1,\dots$ i.i.d. $\sim {\rm U}(0,1)$;
\item[(iii)] $L_1, L_2, \dots$ i.i.d. $\sim {\rm Exp}(1)$;
\item[(iv)] $Y_{01}, Y_{02}, \dots Y_{0\dmax}$, where $Y_{0i} \sim {\rm Bin}(i, 1-\re^{-\lambda I})$;
\item[(v)] for $i=1,2,\dots,\dmax$, $Y_{i1}, Y_{i2},\dots$ i.i.d. $\sim {\rm Bin}(i-1, 1-\re^{-\lambda I})$.
\end{enumerate}

For $n=1,2,\dots$, let $\pn_i=n^{-1}\sum_{k=1}^n 1_{\{D_k=n\}}$ $(i=0,1,\dots,\dmax)$ and $\pnt_i=i\pn_i/\mudn$
$(i=1,2,\dots,\dmax)$, where $\mudn=n^{-1}\sum_{k=1}^n D_k$.  Note that by the strong law of large numbers
$\mudn>0$ (and $\pnt_i$ is well defined) for all sufficiently large $n$ almost surely. Let
$\cnt_i=\sum_{j=1}^i \pnt_j$ $(i=1,2,\dots,\dmax)$.  For $x \in (0,1)$, let $\dnt(x)=\min\{i:x \le \cnt_i\}$.
Similarly, let $\pt_i=ip_i/\mud$ and $\ct_i=\sum_{j=1}^i \pt_j$ $(i=1,2,\dots,\dmax)$, and let 
$\dt(x)=\min\{i:x \le \ct_i\}$ $(0<x<1)$.

For $n=1,2,\dots$, construct on $(\Omega, \mathcal{F}, {\rm P})$  a realisation of a two-type continuous-time Markov branching process $\BBtn$, which approximates the process of infected and recovered half-edges in the 
epidemic $\EEtn$, as follows.   The types are denoted $I$ and $R$ depending on whether the individual corresponds to an infective
or recovered half-edge.  Only type-$I$ individuals have offspring and they do so at their moment of death.  Type-$R$ individuals live forever.
For $t \ge 0$, let $\Ynehat(t)$ and $\Znehat(t)$ denote respectively the number of type-$I$ and type-$R$ individuals alive in $\BBtn$ at time $t$.
The initial state $(\Ynehat(0),\Znehat(0))$ is determined as follows.  Let $\dn_0=\dnt(U_0)$, which corresponds to the initial infective in
$\EEtn$ having degree $\dn_0$.  If $\dn_0=0$ then $(\Ynehat(0),\Znehat(0))=(0,0)$ and $\BBtn$ goes extinct immediately.  Alternatively, if
$\dn_0>0$ then $(\Ynehat(0),\Znehat(0))=(Y_{0\dn_0}, d_0-Y_{0\dn_0})$.  In that case, for $k=1,2,\dots$, the $k$th type-$I$ individual born in $\BBtn$
(including the initial individuals) has degree $\dn_k=\dnt(U_k)$ and lives until age $L_k$, when it dies.  Denote this individual by
$i^*$.  Suppose that $\dn_k=i$ and $i^*$ is the $l$th degree-$i$ individual (excluding the initial individuals) born in $\BBtn$.  Then, when $i^*$ dies, 
it leaves $Y_{il}$ type-$I$ and $i-1-Y_{il}$ type-$R$ offspring.   Of course, reproduction stops in $\BBtn$ if $\Ynehat(t)=0$.
Construct also on $(\Omega, \mathcal{F}, {\rm P})$ a realisation of a two-type continuous-time branching process $\BBt$, defined
analogously to $\BBtn$ but  using the function $\dt$ instead of $\dnt$.  For $t \ge 0$, let $Y_E(t)$ and $Z_E(t)$ denote respectively the numbers of
type-$I$ and type-$R$ individuals alive in $\BBt$ at time $t$.

For $n=1,2,\dots$, construct on $(\Omega, \mathcal{F}, {\rm P})$ a realisation of the epidemic $\EEtn$, defined in Section~\ref{sec:altconstr},
as follows.  Give the $n$ individuals in $\EEtn$ the labels $1,2,\dots,n$ in increasing order of degree.  Now label the $n\mudn$ half-edges
$1,2,\dots,n\mudn$, starting with the half-edges (if any) attached to individual $1$,  then the half-edges (if any) attached to individual $2$
and so on.  Thus half-edges attached to the same individual have consecutive labels.  As in Section~\ref{sec:altconstr}, for $t \ge 0$, 
let $\Yne(t)$ and  $\Zne(t)$ denote respectively the number of infective and recovered half-edges in $\EEtn$ at time $t$.
The initial infective in $\EEtn$ is the individual who owns the half-edge labelled $\left \lfloor{n\mudn U_0}\right \rfloor+1$. Note that
this individual has degree $\dn_0=\dnt(U_0)$.  If $\dn_0=0$ then the epidemic stops immediately.  Alternatively, if $\dn_0>0$ then the initial 
infective infects along $\kn_0=Y_{0\dn_0}$ of its half-edges with its remaining $\dn_0-\kn_0$ half-edges becoming recovered half-edges.  The epidemic
stops if $\kn_0=0$, otherwise the initial infective transmits infection along its infective half-edges at times $L_1,L_2,\dots,L_{k_0}$.

For $k=1,2,\dots,$ let $l_k^{(n)}=\left \lfloor{n\mudn U_{k}}\right \rfloor+1$ and note that the half-edge having label $l_k^{(n)}$ is attached to an individual
having degree $\dn_k=\dnt(U_k)$.
When infection is transmitted along a half-edge that half-edge, $l_*$ say, is attempted to be paired with the half-edge having label 
$l_k^{(n)}$, where $k$ is the number of the $U_0,U_1,\dots$ that have been used already in the construction
of $\EEtn$.  (Thus, for example, the first half-edge emanating from the initial infective is attempted to be paired with the half-edge having
label $l_1^{(n)}$.)  If the half-edge $l_k^{(n)}$ has already been paired or $l_k^{(n)}=l_*$ then the attempt fails and
$l_*$ is attempted to be paired with the half-edge $l_{k+1}^{(n)}$, and so on until a valid pairing is obtained and an edge is formed.  Suppose that
a valid pairing is made with the half-edge having label $l_V$.  Let $i^*$ be the individual that owns the half-edge $l_V$ and $i$ be the degree
of $i^*$.  If $i^*$ is susceptible then it becomes an infective, otherwise nothing happens apart from the formation of the edge.  Suppose that $i^*$ is susceptible and is the $l$th degree-$i$ susceptible to be infected in $\EEtn$, excluding the initial infective.  Then $i^*$ infects along $Y_{il}$ of its half-edges
and its remaining $i-1-Y_{il}$ half-edges become recovered half-edges. (When $i^*$ was infected one of its $i$ half-edges was paired to its infector.)
Let $k_1=Y_{il}$. The times of these $k_1$, infections, relative to the infection time of $i^*$, are given by $L_{k^*+1}, L_{k^*+2}, \dots, L_{k^*+k_1}$,
where $k^*$ is the number of infective half-edges created in $\EEtn$ prior to the infection of $i^*$.  The epidemic terminates when $\Yne(t)=0$, or when 
$\Yne(t)=1$ and all other half-edges have been paired.

As in Section~\ref{sec:altconstr}, for $t \ge 0$, let $\bXn(t)=(\Xn_0(t), \Xn_1(t),\dots,\Xn_{\dmax}(t))$, where $\Xn_i(t)$ is the number of degree-$i$ susceptible individuals at time $t$ in
$\EEtn$.  For $n=1,2,\dots$, let $\tauhatn=\inf\{t \ge 0:\Yne(t) \ge \log n\}$, where $\tauhatn=\infty$ if $\Yne(t) < \log n$ for all $t \ge 0$.
Let $A_{\rm ext}=\{\omega \in \Omega: \lim_{t \to \infty} Y_E(t)=0\}$ denote the set on which type-$I$ individuals become extinct in the branching process $\BBt$ and
let $\alpha$ denote the Malthusian parameter of $\BBt$.  Then $\alpha$ is given by the unique real solution of the equation
\begin{equation*}
\int_0^{\infty} \re^{-\alpha t} \mudt p_I \re^{-t}\,{\rm d}t=1,
\end{equation*}
so $\alpha=R_0-1$, where $R_0$ is defined at~\eqref{equ:Rzero}.

\begin{theorem}
\label{thm:BPapprox}
\begin{enumerate}
\item[(a)] $\lim_{n \to \infty} {\rm P}\left(\tauhatn=\infty | A_{\rm ext} \right)=1$.

\item[(b)] Suppose that $R_0>1$, so ${\rm P}(A_{\rm ext}^C)>0$.  Then, as $n \to \infty$, 
\begin{enumerate}
\item[(i)] ${\rm P}\left(\tauhatn<\infty | A_{\rm ext}^C \right) \to 1$;
\item[(ii)] $\frac{1}{\log n}\Yne(\tauhatn) \left|\right. A_{\rm ext}^C\convp 1$;
\item[(iii)] $\frac{1}{\log n}\Zne(\tauhatn) \left|\right. A_{\rm ext}^C\convp \alpha^{-1}q_I\mudt$;
\item[(iv)] $\sqrt{n}\left(n^{-1}\bXn(\tauhatn)-\bp\right) | A_{\rm ext}^C \convD {\rm N}(\bzero, \Sigma_{XX})$, where $\Sigma_{XX}$ is defined at~\eqref{equ:SigmaXX}.
\end{enumerate}
\end{enumerate}
\end{theorem}

\begin{proof}
The key observations underlying the proof are that (i) the processes $\{(\Ynehat(t),\Znehat(t)):t \ge 0\}$ and 
$\{(\Yne(t),\Zne(t)):t \ge 0\}$ coincide up until at least the first time that an attempt is made to pair a half-edge 
with a half-edge belonging to an individual previously used in the construction of $\EEtn$ and 
(ii) the branching processes $\BBtn$ and $\BBt$ coincide up until the first time that
$\dnt(U_k) \ne \dt(U_k)$.  Thus we show that the probability  that the processes $\{(\Ynehat(t),\Znehat(t)):0 \le t \le \tauhatn\}$ and 
$\{(Y_E(t),Z_E(t)):0 \le t \le \tauhatn\}$ coincide converges to $1$ as $n \to \infty$.  The theorem then follows using standard results concerning
the asymptotic behaviour of the branching process $\BBt$.

For $n=1,2,\dots$ and $k=1,2,\dots , n\mudn$, let $\mathcal{C}^{(n)}(k)$ be the set of half-edges attached
to the individual owning half-edge $k$ in $\EEtn$.  Thus $ k \in \mathcal{C}^{(n)}(k)$ and $|\mathcal{C}^{(n)}(k)| \le \dmax$.
For $n=1,2,\dots$, let
\[
\Mn=\min\left\{k \ge 1: l_k^{(n)} \in \bigcup_{i=0}^{k-1}\mathcal{C}^{(n)}(l_i^{(n)})\right\}
\]
be the number of pairings required in $\EEtn$ until an attempt is made to pair a half-edge 
with a half-edge belonging to a previoulsy used individual.  Let $\bD=(D_1,D_2,\dots)$.  Then, for $m=1,2,\dots$,
\begin{eqnarray*}
{\rm P}(\Mn \le m|\bD) &=&{\rm P}\left(\bigcup_{k=1}^m \left\{l_k^{(n)} \in \bigcup_{i=0}^{k-1}\mathcal{C}^{(n)}(l_i^{(n)})\right\}|\bD\right)\\
&\le& \frac{m(m+1)}{2} \frac{\dmax}{n\mudn}.
\end{eqnarray*}
Now $\mudn \convas \mud$ as $n \to \infty$ by the strong law of large numbers.  Hence, for any $\gamma \in (0,1/2)$,
${\rm P}(\Mn \le n^{\gamma}|\bD) \convp 0$ as $n \to \infty$.  Thus, 
\begin{equation}
\label{equ:Epicoupling}
\lim_{n \to \infty}{\rm P}(\Mn \le n^{\gamma})=\lim_{n \to \infty}{\rm E}[{\rm P}(\Mn \le n^{\gamma}|\bD)]=0,
\end{equation}
as ${\rm P}(\Mn \le n^{\gamma}|\bD)$ $(n=1,2,\dots)$ is uniformly bounded.

Turning to the coupling of $\BBtn$ and $\BBt$, for $n=\,2,\dots$,
\begin{eqnarray}
\label{equ:tvbound}
\max_{i=1,2,\dots,\dmax-1}|\cnt_i-\ct_i| &\le& \frac{1}{2}\sum_{i=1}^{\dmax}|\pnt_i-\pt_i|\\
&\le& \frac{1}{\mud}\sum_{i=1}^{\dmax}i|\pn_i-p_i|.\nonumber
\end{eqnarray}
The first inequality in~\eqref{equ:tvbound} follows from the total variation distance between the distributions
$\pnt_i$ $(i=1,2,\dots,\dmax)$ and $\pt_i$ $(i=1,2,\dots,\dmax)$, and the second inequality follows using the triangle
inequality and elementary manipulation (\cite{BN:2017}, equation (3.2)).  For $n=1,2,\dots$, let 
$\hat{M}^{(n)}=\min\{k \ge 0: \dnt(U_k) \ne \dt(U_k)\}$.  Then, for $k=0,1,\dots$,
\[
{\rm P}\left(\hat{M}_n \le k\right) \le \frac{(k+1)(\dmax-1)}{\mud} {\rm E}\left[\sum_{i=1}^{\dmax} i |\pn_i-p_i|\right].
\]
Now, for $i=1,2,\dmax$, as $n\pn_i \sim {\rm Bin}(n,p_i)$,
\[
{\rm E} \left[|\pn_i-p_i|\right] \le \sqrt{{\rm var}(\pn_i)}=\frac{1}{\sqrt{n}}\sqrt{p_k(1-p_k)}.
\]
Thus, for any $\gamma \in (0,1/2)$,
\begin{equation}
\label{equ:BPcoupling}
\lim_{n \to \infty} {\rm P}\left(\hat{M}^{(n)} \le n^\gamma\right)=0.
\end{equation}

For $t \ge 0$, let $T_E(t)$ be the total number of individuals (of either type) alive in 
the branching process $\BBt$ during $[0,t]$ and let $T_E(\infty)=\lim_{t \to \infty} T_E(t)$.
Thus $T_E(\infty,\omega)<\infty$ if and only if $\omega \in A_{\rm ext}$.
For $n=1,2,\dots$ and $t \ge 0$, let 
\[
\An_{\rm couple}(t)=\left\{\omega \in \Omega: \left(\Yne(u),\Zne(u)\right)=\left(Y_E(u),Z_E(u)\right)\mbox{ for all }u \in [0,t]\right\}.
\]

Let $\tau=\min\{t \ge 0:Y_E(t)=0\}$, where $\tau(\omega)=\infty$ if $\omega \in A_{\rm ext}^C$.  Then~\eqref{equ:Epicoupling}
and~\eqref{equ:BPcoupling} imply that, for $k=0,1,\dots$,
\[
\lim_{n \to \infty}{\rm P}\left(\An_{\rm couple}(\tau)\cap \left\{T_E(\infty)=k\right\} \right)=1.
\]
Part (a) of the theorem follows since $A_{\rm ext}=\bigcup_{k=0}^{\infty} \{T_E(\infty)=k\}$.

Note that although $\BBt$ is a two-type branching process, since only type-$I$ individuals produce offspring,
it can be treated as a single-type branching process consisting only of type-$I$ individuals, in which attached to each 
(type-$I$) individual is a random characteristic (see~\cite{Nerman81}) given by its number of type-$R$ offspring.
Then it follows from~\cite{Nerman81}, Theorem 5.4, that there exists a random variable $W \ge 0$, where 
$W(\omega)=0$ if and only if $\omega \in A_{\rm ext}$ such that, as $t \to \infty$,
\begin{equation}
\label{equ:yelim}
\re^{-\alpha t} Y_E(t) \convas  W,
\end{equation}
\begin{equation}
\label{equ:zelim}
\re^{-\alpha t} Z_E(t) \convas  \alpha^{-1}q_I\mudt W
\end{equation}
and
\begin{equation}
\label{equ:telim}
\re^{-\alpha t} T_E(t) \convas \alpha^{-1}\mudt W.
\end{equation}
(It is easily checked that~\eqref{equ:yelim}-\eqref{equ:telim} hold by calculating the the appropriate
$m_{\infty}^{\phi}$ in ~\cite{Nerman81}, Theorem 5.4, and using Corollary 3.2 of that paper. For a heuristic 
argument note, for example, that if $Y_E(u)\approx W \re^{\alpha u}$ $(u \ge 0)$ then $T_E(t) \approx 
\int_0^t \mudt \re^{\alpha u} W\,{\rm d}u \approx \alpha^{-1}\mudt W \re^{\alpha t}$, since type-$I$ individuals die at rate $1$
and have on average $\mudt$ offspring.)

For $n=1,2,\dots$, let $\taubarn=\inf\{t \ge 0:Y_E(t) \ge \log n\}$, where $\taubarn(\omega)=\infty$ if $\omega \in A_{\rm ext}$.
Then it follows from~\eqref{equ:yelim} that $\taubarn(\omega)<\infty$ and $\lim_{n \to \infty}
\taubarn(\omega)=\infty$, for ${\rm P}$-almost all $\omega \in A_{\rm ext}^C$.  Thus, for ${\rm P}$-almost all $\omega \in A_{\rm ext}^C$,
\begin{equation}
\label{equ:yetaubralim}
\lim_{n \to \infty} \frac{1}{\log n} Y_E(\taubarn)=1,
\end{equation}
since an individual has at most $\dmax$ offspring, and using~\eqref{equ:yelim}-\eqref{equ:telim},
\begin{equation}
\label{equ:zetaubarlim}
\lim_{n \to \infty} \frac{1}{\log n} Z_E(\taubarn)=\alpha^{-1}q_I\mudt
\end{equation}
and
\begin{equation}
\label{equ:tetaubarlim}
\lim_{n \to \infty} \frac{1}{\log n} T_E(\taubarn)=\alpha^{-1}\mudt.
\end{equation}
Now~\eqref{equ:Epicoupling}, \eqref{equ:BPcoupling} and~\eqref{equ:tetaubarlim} imply that
\begin{equation}
\label{equ:panlim}
\lim_{n \to \infty}{\rm P}\left(\An_{\rm couple}(\taubarn)|A_{\rm ext}^C\right)=1,
\end{equation}
and (i)-(iii) of part (b) of the theorem follow using~\eqref{equ:zelim} and~\eqref{equ:yetaubralim}.

To prove part (b)(iv), note by the multivariate central limit theorem  that
\[
\sqrt{n}\left(n^{-1}\bXn(0)-\bp\right) \convD {\rm N}(\bzero, \Sigma_{XX}) \qquad \mbox{as } n \to \infty.
\]
Also,~\eqref{equ:tetaubarlim} and~\eqref{equ:panlim} imply that, for $i=0,1,\dots,\dmax$,
\[
\lim_{n \to \infty}{\rm P}\left(\Xn_i(\tauhatn)-\Xn_i(0) \le 2 \alpha^{-1} \mudt \log n|A_{\rm ext}^C\right)=1,
\]
so $\frac{1}{\sqrt{n}}\left(\Xn_i(\tauhatn)-\Xn_i(0)\right)|A_{\rm ext}^C \convp 0$ as $n \to \infty$.  Part (b)(iv)
now follows using Slutsky's theorem.
\end{proof}

\subsubsection{Epidemic starting with $\left \lceil{\log n}\right \rceil$ infectives}
\label{sec:epilogn}

Turning to the proof of Theorem~\ref{thm:majclt}, note that by Theorem~\ref{thm:BPapprox},
$\lim_{n \to \infty}{\rm P}\left(\Gn \triangle A_{\rm ext}\right)=0$, where $\triangle$ denotes
symmetric difference.  Hence, using the strong Markov property,
we can determine the asymptotic distribution of $\TnNSW|\Gn$ by considering the random time-scale transformed process 
$\{\bWnt(t)\}$, defined in Section~\ref{sec:timetrans}, with initial state $\bWnt(0)=(\bXn(\tauhatn),\Yne(\tauhatn),\Zne(\tauhatn))$.
Thus, by Theorem~\ref{thm:BPapprox}(b),
\begin{equation}
\label{equ:wntildeinit}
\sqrt{n}\left(n^{-1}\bWnt(0)-(\bp,0,0)\right) \convD {\rm N}(\bzero, \SigmaNSW_0)\qquad \mbox{as } n \to \infty,
\end{equation}
where $\SigmaNSW_0$ is obtained by setting $\epsilon=0$ in~\eqref{equ:SigmaNSW0}.

By Theorem~\ref{thm:wlln}, for any $t>0$,
\begin{equation}
\label{equ:ynetweak}
\sup_{0 \le u \le t}|n^{-1}\Ynet(u)-\yet(u)| \convp 0 \qquad \mbox{as } n \to \infty,
\end{equation}
where, using~\eqref{equ:yetilde}
\begin{equation*}
\label{equ:yetilde1}
\yet(t)=\mud\left(\re^{-2t}-q_I\re^{-t}\right)-p_If_D^{(1)}(\re^{-t}).
\end{equation*}
Note that $\yet(t)=0$ if and only if 
\begin{equation}
\label{equ:taut1}
\re^{-t}-q_I-\mud^{-1}p_If_D^{(1)}(\re^{-t})=0
\end{equation}
and that $t=0$ is a solution of~\eqref{equ:taut1}. If $R_0 \le 1$ then $t=0$ is the only solution in
$[0, \infty)$, but if $R_0>1$  then we show in Appendix~\ref{app:tautprop} that, under the conditions of Theorem~\ref{thm:majclt},
there is a unique solution in $(0, \infty)$ which we denote by $\taut$.

As in Section~\ref{section:clt}, let $\tautn=\inf\{t \ge 0: \varphi(n^{-1}\bWnt(t))\le 0\}$.  Recall that $R_0>1$ is assumed in Theorem~\ref{thm:majclt}. It follows from~\eqref{equ:ynetweak}
that, $\min\left(\tautn, |\tautn-\taut|\right) \convp 0$ as $n \to \infty$.   
We show below that there exists $\epsilon_0>0$ such that 
\begin{equation}
\label{equ:tautneps}
\lim_{n \to \infty}{\rm P}(\tautn<\epsilon_0)=0.
\end{equation}
Theorem~\ref{thm:hitting} and Corollary~\ref{cor:hitting}, as stated, cannot be applied in the present setting
as $\varphi(\bwt(0))=0$.  However, in the terminology of Theorem~\ref{thm:hitting}, the proof 
of~\cite{EK86}, Theorem 11.4.1, extends easily to the situation when $\taun=\inf\{t \ge \epsilon_0: \varphi(n^{-1}\bXn(t)) \le 0\}$ and
$\tau=\inf\{t \ge \epsilon_0: \varphi(\bx(t)) \le 0\}$, for fixed $\epsilon_0>0$.  Hence, so do Theorem~\ref{thm:hitting} and Corollary~\ref{cor:hitting}.

Let 
\begin{equation}
\label{equ:Tnt}
\Tnt=\sum_{i=0}^{\dmax}\left(\Xnt_i(0)-\Xnt_i(\tautn)\right).
\end{equation}
Then using~\eqref{equ:tautneps} and the above extension of Corollary~\ref{cor:hitting} (we show in Appendix~\ref{app:tautprop} that $a(\taut)=\nabla \varphi(\bw(\taut))\cdot \tilde{F}(\bw(\taut))<0$), setting $\epsilon=0$ in the proof of Theorem~\ref{thm:posclt}
yields
\[
\sqrt{n}\left(n^{-1}\Tnt-\rho\right) \convD {\rm N}(0,\sigmatNSW) \qquad \mbox{as } n \to \infty.
\]
Theorem~\ref{thm:majclt} follows using Slutsky's theorem since, by~\eqref{equ:tetaubarlim},
\[
\frac{1}{\sqrt{n}}\left(\TnNSW-\Tnt\right) \convp 0 \qquad\mbox{as } n \to \infty.
\]

To complete the proof of Theorem~\ref{thm:majclt} we use a device, first introduced by~\cite{Whittle55} in the setting of a homegeneously mixing epidemic, to show that there exists $\epsilon_0>0$ such that~\eqref{equ:tautneps} holds.  For $n=1,2,\dots$, construct on $(\Omega, \mathcal{F},{\rm P})$ an epidemic $\EEcheckn$  
analogously to $\EEtn$ except having initial state given in an obvious notation by
\[
\left(\bXncheck(0),\Ynecheck(0),\Znecheck(0)\right) =\left(\bXn(\tauhatn),\Yne(\tauhatn),\Zne(\tauhatn)\right).
\]
Let $\BBcheckn$ be the corresponding branching process which assumes all attempted pairing in $\EEcheckn$ are valid.  Let
\[
\Tncheck=\sum_{i=0}^{\dmax}\left(\Xncheck_i(0)-\Xncheck_i(\infty)\right)
\]
be the final size of $\EEcheckn$.  Fix $\delta \in (0,1)$.  Then, provided $\Tncheck \le n\delta$, for a given pairing, the probability that 
an invalid pairing is attempted is at most $\pn(\delta)=\min(\delta n\dmax/(n\mudn),1)=\min(\delta \dmax /\mudn,1)$.  It follows that
\begin{equation}
\label{equ:lbBP}
{\rm P}\left(\Tncheck \ge n\delta\right) \ge {\rm P}\left(\Tncheck_B(\delta) \ge n\delta\right),
\end{equation}
where $\Tncheck_B(\delta)$ is the total number of type-$I$ individuals (excluding the initial individuals) ever alive in the
branching process $\BBcheckn(\delta)$ that is obtained from $\BBcheckn$ by aborting each type-$I$ individual independently with 
probability $\pn(\delta)$.

Let $\Rn_0(\delta)$ be the mean number of type-$I$ individuals spawned by a typical type-$I$ individual in $\BBcheckn(\delta)$
and let $\pin(\delta)$ be the extinction probability for type-$I$ individuals in $\BBcheckn(\delta)$ if the process were to start with 
one type-$I$ individual. As $n \to \infty$,  $\pn_i \convas p_i$  ($i=0,1,\dots,\dmax$), by the strong law of large numbers, so
$\pnt_i \convas \pt_i$  ($i=1,2,\dots,\dmax$) and $\mundt \convas \mudt$, where $\mundt=\sum_{i=1}^{\dmax}(i-1)\pnt_i$.  Thus,
\[
\Rn_0(\delta)=(1-\pn(\delta))\mundt p_I \convas (1-p(\delta))R_0 \qquad \mbox{as } n \to \infty,
\]
where $p(\delta)=\min(\delta \dmax/\mud,1)$.  Further, as $n \to \infty$, the offspring distribution of $\BBcheckn(\delta)$ converges
almost surely to that of $\BBt(\delta)$, where $\BBt(\delta)$ is the branching process obtained from $\BBt$ by aborting each
type-$I$ individual independently with probability $p(\delta)$.  Thus, by a simple extension of~\cite{BJML:2007}, Lemma 4.1,
$\pin(\delta) \convas \pi(\delta)$ as $n \to \infty$, where $\pi(\delta)$ is the extinction probability for type-$I$ individuals in
$\BBt(\delta)$ starting from one type-$I$ individual.  

Recall that $R_0>1$.  Thus $\delta \in (0,1)$ can be chosen sufficently small such that $(1-p(\delta))R_0>1$, whence
$\pi(\delta)<1$.

Now, $\Ynecheck(0) \ge \log n$ by construction, so using~\eqref{equ:lbBP}, for such $\delta$,
\[
\lim_{n \to \infty} {\rm P}\left(\Tncheck \ge n\delta|\bD\right) \ge 1- \lim_{n \to \infty}\left(\pin\right)^{\log n}=1,
\]
for ${\rm P}$-almost all $\omega \in  A_{\rm ext}^C$.  Hence 
\begin{equation}
\label{equ:ptncheck}
\lim_{n \to \infty}{\rm P}\left(\Tncheck \ge n\delta | A_{\rm ext}^C\right)=1.
\end{equation}

Finally, for $t \ge 0$, let $\xt(t)=\sum_{i=0}^{\dmax}\xt_i(t)=f_D(\re^{-t})$, using~\eqref{equ:xtilde} with $\xt_i(0)=p_i$ 
$(i=0,1,\dots,\dmax)$.  By construction, $\bWn(\taun)=\bWnt(\tautn)$, so~\eqref{equ:ptncheck} and Theorem~\ref{thm:wlln} applied to the process $\{\bWnt(t)\}$,
with initial state satisfying~\eqref{equ:wntildeinit}, imply that~\eqref{equ:tautneps} holds with $\epsilon_0$ being given by the
unique solution in $(0,\infty)$ of $1-f_D(\re^{-\epsilon_0})=\delta$.

\subsection{Proof of Theorem~\ref{thm:zeroclt}}
\label{sec:proofzeroclt}
The proof of Theorem~\ref{thm:zeroclt} for the epidemic on the NSW random graph parallels in an obvious fashion the argument in Section~\ref{sec:epilogn} above without the need to condition on $A_{\rm ext}^C$, so the details are omitted.  Again, the proof for the epidemic on the MR random graph 
is similar but simpler.

\subsection{Proof of Theorem~\ref{thm:percclt}}
\label{sec:proofpercclt}
Consider first bond percolation and note that when $I\equiv 1$ and $\lambda=-\log(1-\pi)$ then in the
directed random graph $\GGtn$ defined at the start of Section~\ref{sec:altconstr}, all the possible directed edges in $\GGtn$
(i.e. between pairs of neighbours in $\GGn$) are present independently, each with probability $\pi$. Further, when constructing
the final outcome $\def\TTn{\mathcal{T}^{(n)}}$ of the corresponding epidemic $\EEn$ using $\GGtn$, for any pair $(i,j)$ of distinct individuals
use is made of at most one of the directed edges $i \to j$ and $j \to i$ (if $i$ infects $j$, whether or not $j$ tries to infect $i$
is immaterial).  Thus, in this situation $\GGtn$ can be replaced by $\GGnbond$, obtained using bond percolation on $\GGn$, and in $\EEn$, an initial susceptible is
ultimately infected if and only if in $\GGnbond$ there is a chain of edges connecting it to an initial infective.

Suppose that there is one initial infective.  Then the final size of the epidemic (including the initial infective) is given by
the size of the connected component of $\GGnbond$ that contains the initial infective.  With probability tending to $1$ as $n \to \infty$, 
a major outbreak occurs in $\EEn$ if and only if the initial infective belongs to the largest connected component of $\GGnbond$, and
the final size of a major epidemic is given by the size $\Cn$ of that connected component.  Further $\Cn$ has the same asymptotic distribution as
$\Tn|\Gn$.  Thus~\eqref{equ:MRbondclt} and~\eqref{equ:NSWbondclt} follow immediately on setting
$I\equiv 1$ and $\lambda=-\log(1-\pi)$ in Theorem~\ref{thm:majclt}.  (Note that $p_I=\pi$ and, since $I$ is constant, $q_I^{(2)}=q_I^2$.)

Turning to site percolation, consider the epidemic $\EEn$ with ${\rm P}(I=\infty)=\pi=1-{\rm P}(I=0)$, so each infective infects all of its
neighbours with probability $\pi$ and none of them otherwise.  Suppose that there is one initial infective, $i^*$ say, and 
let $\TTn=\{j \in \NNn \setminus \{i^*\}: i^* \leadsto j\}$ using the directed random graph $\GGtn$.  Then $\TTn\cup \{i^*\}$ differs
from the connected component containing $i^*$ in $\GGnsite$ (site percolation on $\GGn$) in that $\TTn$ also includes infected indiviudals
having $I=0$, which are deleted in $\GGnsite$.  Thus, to obtain a central limit theorem for $\Cn$, we need one for 
$\Vn=|\{j \in \NNn \setminus\{i^*\}:I_j=\infty \mbox{ and } i^* \leadsto j\}|$, the final size of $\EEn$ counting only
individuals with $I=\infty$.  This can be obtained by augmenting the process $\{\bWn(t)\}$ as we now outline.  

Let $\{\bWnhat(t)\}=\{\bWnhat(t):t \ge 0\}$, where 
\[
\bWnhat(t)= (\bWn(t), \Vn(t))=(\bXn(t), \Yne(t), \Zne(t), \Vn(t))
\]
and $\Vn(t)$ is the total number of initial susceptibles that are infected during $(0,t]$ and have $I=\infty$. A typical
element of $\Hnhat$, the state space of $\{\bWnhat(t)\}$, is now $\bnhat=(n^X_0,n^X_1,\dots,n^X_{\dmax},n^Y_E,n^Z_E, n^V)$.  The transition
intensities of $\{\bWnhat(t)\}$ are essentially the same of those of $\{\bWn(t)\}$, given at the end of Section~\ref{sec:altconstr}, except
now transitions of type (i) are partitioned according to whether or not the infected susceptible has $I=\infty$.  For
$i=1,2,\dots,\dmax$, we have
\[
\qhatn(\bnhat, \bnhat-\bes_i-\bei+(i-1)\ber)=\frac{n^Y_E i n^X_i p_{i,k}(1-\pi)}{n^X_E+n^Y_E+n^Z_E-1},
\]
corresponding to the degree-$i$ infected susceptible having $I=0$, and
\[
\qhatn(\bnhat, \bnhat-\bes_i+(i-1)\bei+\bev)=\frac{n^Y_E i n^X_i p_{i,k}\pi}{n^X_E+n^Y_E+n^Z_E-1},
\]
corresponding to the degree-$i$ infected susceptible having $I=\infty$.

We apply the same random time-scale transformation to $\{\bWnhat(t)\}$ as done to $\{\bWn(t)\}$ in Section~\ref{sec:timetrans}.
Denote the time-transformed process by $\{\bWnbar(t)\}$.
Let $\blhat^{(1)}_{i,0}=\bl_{i,0}^{(1)}$
and $\blhat^{(1)}_{i,i-1}=\bl_{i,i-1}^{(1)}+\bev$ $(i=1,2,\dots,\dmax)$, and $\blhat^{(j)}_+=\bl^{(j)}_+$ and
$\blhat^{(j)}_-=\bl^{(j)}_-$ $(j=2,3)$.
The set of possible jumps of $\{\bWnbar(t)\}$ from a typical state $\bnhat$ is now $\hat{\Delta}=\hat{\Delta}_1 \cup \hat{\Delta}_2 \cup \hat{\Delta}_3$,
where $\hat{\Delta}_1=\bigcup_{i=1}^{\dmax} \{\blhat^{(1)}_{i,0}\}\cup\{\blhat^{(1)}_{i,i-1}\}$, 
$\hat{\Delta}_2=\{\blhat^{(2)}_+,\blhat^{(2)}_-\}$ and $\hat{\Delta}_3=\{\blhat^{(3)}_+,\blhat^{(3)}_-\}$.

Let $\bwhat=(\bx,y_E,z_E,v)$.  The family of processes $\{\bWnbar(t)\}$ is again asymptotically density dependent with
corresponding functions $\betabar_{\bl}(\bwhat)$ $(\bl \in \hat{\Delta})$ given by (cf.~\eqref{equ:intensityfun})
\begin{equation}
\label{equ:intensityfunperc}
\betabar_{\bl}(\bwhat) = \begin{cases}
	      \betabar_{i,0}^{(1)}(\bx,y_E,z_E,v)=ix_i (1-\pi)& \text{ for } \bl=\blhat_{i,0}^{(1)} \in \hat{\Delta}_1, \\
	      \betabar_{i,i-1}^{(1)}(\bx,y_E,z_E,v)=ix_i \pi& \text{ for } \bl=\blhat_{i,i-1}^{(1)} \in \hat{\Delta}_1, \\	      
	      \betabar^{(2)}_+(\bx,y_E,z_E,v)=y_E 1_{\{y_E>0\}}& \text{ for } \bl=\blhat^{(2)}_+,\\
	      \betabar^{(2)}_-(\bx,y_E,z_E,v)=-y_E 1_{\{y_E<0\}}& \text{ for } \bl=\blhat^{(2)}_-,\\
	      \betabar^{(3)}_+(\bx,y_E,z_E,v)=z_E 1_{\{z_E>0\}}& \text{ for } \bl=\blhat^{(3)}_+,\\
	      \betabar^{(3)}_-(\bx,y_E,z_E,v)=-z_E 1_{\{z_E<0\}}& \text{ for } \bl=\blhat^{(3)}_-.	      
\end{cases}
\end{equation}
The associated drift function is (cf.~\eqref{equ:driftFw})
\begin{align*}
\label{equ:driftFw}
\bar{F}(\bwhat)&=-\sum_{i=1}^{\dmax} i x_i \bes_i+\left\{\sum_{i=1}^{\dmax} i x_i[(i-1)\pi-1] -2y_E-z_E\right\}\bei\\
&\qquad\qquad+\left[\sum_{i=1}^{\dmax} i(i-1)(1-\pi)x_i-z_E\right]\ber+\left[\sum_{i=1}^{\dmax} ix_i \pi\right]\bev.\nonumber
\end{align*}

For $t \ge 0$, let $\bwbar(t)=(\xbar_0(t),\xbar_1(t),\dots,\xbar_{\dmax}(t),\yebar(t),\zebar(t), \vbar(t))$ be defined analogously to $\bwt(t)$ at~\eqref{equ:bwt}.  
Noting that $p_I=\pi$ and $q_I=1-\pi$, the corresponding deterministic model for $\bwbar(t)$ is given by~\eqref{equ:dxtidt}-\eqref{equ:dzetdt},
with $\xt_i$ replaced by $\xbar_i$ etc., augmented with
\begin{equation}
\label{equ:dvtdt}
\dfrac{d\vbar}{dt}=\pi \sum_{i=1}^{\dmax} i\xbar_i.
\end{equation}
Thus,~\eqref{equ:xtilde}-\eqref{equ:zetilde} still hold, with the above change of notation, and, using~\eqref{equ:xtilde},
\begin{equation}
\label{equ:vbar}
\vbar(t)=\vbar(0)+\pi\sum_{i=1}^{\dmax}\xbar_i(0)(1-\re^{-it}).
\end{equation}

The stopping time $\tautn$ is unchanged, except now $\varphi(\bwhat)=\varphi(\bx,y_E,z_E,v)=y_E$.
Recall that site percolation corresponds to one initial infective, so under both the MR and NSW models,
$\epsilon_i=0$ and $\xt_i(0)=p_i$ $(i=0,1,\dots,\dmax)$.  Thus $\taut=\inf\{t\ge 0:\yebar(t)=0\}$ ($=\inf\{t\ge 0:\yet(t)=0\}$) is given by the unique solution of~\eqref{equ:taut1} in $(0,\infty)$.  Hence $z=\re^{-\taut}$ is given by the unique solution in $(0,1)$ of~\eqref{equ:zperc}.  
Now $\vt(0)=0$, since $\Vn(0)=0$ for all $n=1,2,\dots$, so using~\eqref{equ:vbar} and recalling that $\rho=1-f_D(z)$,
\[
\vbar(\taut)=\pi\left[1-f_D(\re^{-\taut})\right]=\pi \rho.
\]

Let $\Vnbar=\Vnbar(\tautn)$, so in the site percolation setting, $\Cn\eqD 1+\Vnbar|\Gn$. Now  application of Corollary~\ref{cor:hitting}, 
as at~\eqref{equ:Wntildeclt}, yields
\begin{equation*}
\sqrt{n}\left(n^{-1}\bWnbar(\tautn)-\bwbar(\taut)\right) \convD {\rm N}(\bzero,\Bbar \Sigmabar(\taut) \Bbar^{\top})  \qquad \mbox{as }n \to \infty,
\end{equation*}
where
\[
\Bbar=I-\frac{\bar{F}(\bwbar(\taut))\bigotimes \nabla \varphi(\bwbar(\taut))}{\nabla \varphi(\bwbar(\taut))\cdot \bar{F}(\bwbar(\taut))}
\]
and $\Sigmabar(\taut)$ is obtained by making obvious changes to~\eqref{equ:phit}-\eqref{equ:Gtilde} to account for the extra dimension.
Further, $\Vnbar=\bWnbar(\tautn)(\bzero, 0, 0, 1)^{\top}$,
so
\[
\sqrt{n}\left(n^{-1}\Vnbar-\pi\rho\right) \convD {\rm N}(0, \bar{\sigma}^2)\qquad \mbox{as }n \to \infty,
\]
where
\begin{equation*}
\hat{\sigma}^2=(\bzero,0,0,1)\Bbar\Sigmabar(\taut)\Bbar^{\top}(\bzero,0,0,1)^{\top}.
\end{equation*}
A simple calculation yields $(\bzero,0,0,1)\Bbar=(\bzero,-\pi b(\taut),0,1)$, where $b(\taut)$ is given by~\eqref{equ:btaut}
with $a(\taut)$ obtained by replacing $\fde$ by $f_D$ in~\eqref{equ:ataut}.  Define $\Phibar(t,u)$ analogously to $\Phit(t,u)$ at~\eqref{equ:phit}.
Then it follows using~\eqref{equ:xtilde}-\eqref{equ:zetilde} and~\eqref{equ:vbar} that
\begin{equation*}
\left[(\bzero,0,0,1)\Phibar(\taut,u)\right]_i=
\begin{cases} 
\pi\left(1-re^{-i(\taut-u)}\right)& \text{ if } i=0,1,\dots,\dmax, \\
	      0 & \text{ if } i=\dmax+1,\dmax+2,\\
	      1 & \text{ if } i=\dmax+3,
\end{cases}
\end{equation*}
and $\left[(\bzero,1,0,0)\Phibar(\taut,u)\right]_i$ is $0$, if $i=\dmax+3$, and given by the right-hand side of~\eqref{equ:bzero10phit} (with $p_I=\pi$
and $q_I=1-\pi$, if $i=0,1,\dots,\dmax+2$.  

It follows that $\bcbar(\taut,u)=(\bzero,0,0,1)\Bbar\Phibar(\taut,u)$ is given by
\begin{equation}
\label{equ:bchat}
\bcbar(\taut,u)=-\pi(\bc(\taut,u),0)+\pi(\bone,0,0,0)+(\bzero,0,0,1),
\end{equation}
where $\bc(\taut,u)$ is defined at~\eqref{equ:bc}, and (cf.~\eqref{equ:sigma2a})
\begin{equation*}
\bar{\sigma}^2=\bcbar(\taut,0)\Sigmabar_0 \bcbar(\taut,0)^{\top}+\sum_{\blhat \in \hat{\Delta}}\int_0^{\taut} \left(\bcbar(\taut,u)\blhat^{\top}\right)^2
\betabar_{\bl}(\bwbar(u))\,{\rm d}u,
\end{equation*}
where $\Sigmabar_0=\SigmabarMR_0$ or $\SigmabarNSW_0$, depending on whether the random graph is MR or NSW.  Here, 
$\SigmabarMR_0$ and $\SigmabarNSW_0$ are the asymptotic variance matrices of $n^{-\frac{1}{2}}\bWnbar(0)$ for the
MR and NSW random graphs, respectively; cf.~\eqref{equ:MRinitCLT} and~\eqref{equ:NSWinitCLT}.  

We show in Appendix~\ref{app:asymvarperc} that
\begin{align}
\label{equ:sumsigmaperc}
&\sum_{\blhat \in \hat{\Delta}}\int_0^{\taut} \left(\bcbar(\taut,u)\blhat^{\top}\right)^2 \betabar_{\bl}(\bwbar(u))\,{\rm d}u
=\pi(1-\pi)[\rho-2h(z)(1-z)\mud]\\
&\qquad\qquad\qquad+\pi^2 \sum_{\bl \in \Delta}\int_0^{\taut} \left(\bc(\taut,u)\bl^{\top}\right)^2 \betat_{\bl}(\bwt(u))\,{\rm d}u\nonumber
\end{align}
and, for both the MR and NSW random graphs, that
\begin{equation}
\label{equ:bchatSigma0}
\bcbar(\taut,0)\bar{\Sigma}_0 \bcbar(\taut,0)^{\top}=\bc(\taut,0)\Sigma_0 \bc(\taut,0)^{\top}.
\end{equation}
Equations~\eqref{equ:MRsiteclt} and~\eqref{equ:NSWsiteclt} then follow using Theorem~\ref{thm:majclt}, noting that $q_I^{(2)}=\pi(1-\pi)$.

\section{Concluding comments}
\label{sec:conc}
A shortcoming of our results, from a mathematical though not a practical viewpoint, is the requirement of a maximal degree 
$\dmax$.  It seems likely that Theorems~\ref{thm:posclt}-\ref{thm:zeroclt} continue to hold when that requirement is relaxed,
subject to appropriate conditions on the degree sequence (MR model) or degree distribution $D$ (NSW model).  This
conjecture is supported by the recent work of~\cite{BR:2017}, who use Stein's method to obtain a central limit theorem for local graph
statistics in the MR model, allowing for unbounded degree, and as an application obtain a central limit theorem for the size of the giant
component with asymptotic variance given by $\sigmatMR$ with $p_I=1$.  To extend the present proof to models with unbounded degree
would require a functional central limit theorem for density dependent population processes with countable state spaces. \cite{BL:2012}
give such a theorem but it is not applicable in our setting as it would require a finite upper bound on the number of neighbours an
individual can infect.

The central limit theorems can be extended, at least in principle, to allow for the infection rate $\lambda$ to depend on the degree
of an infective, and also to more general infection processes in which the set of its neighbours that are contacted by a given infective
is a symmetric sampling procedure (\cite{ML:1986}).  In both cases it is straightforward to determine the limiting deterministic model
in Section~\ref{section:detmodel} and the equation corresponding to~\eqref{equ:taut}, which governs $\taut$, but calculation of the asymptotic variances in the central limit
theorems is likely to be prohibitive.

The configuration model does not display clustering in the limit as $n \to \infty$ and several authors have considered
modifications of the configuration model that introduce clustering.  In~\cite{Trapman:2007} and~\cite{CL:2014}, in the
configuration model construction, for $d=1,2,\dots$, some individuals having $d$ half-edges are replaced by fully connected
cliques, each of size $d$, with each member of a clique having exactly one half-edge.  The half-edges are then paired up in
the usual fashion.  In~\cite{Gleeson:2009} and~\cite{BST:2010}, the network is formed as in the configuration model
and the population is also partitioned into fully connected cliques.  In both models, the set of edges in the network is the
union of those in cliques and the paired half-edges.  The methodology in this paper can be extended to this general class of models as
follows.

As in Section~\ref{sec:altconstr}, the network and epidemic are constructed simultaneously.  The objects counted
are now fully susceptible cliques, typed by their size and degree composition, and infective and recovered half-edges.
When infection is transmitted down a half-edge that half-edge is paired with a uniformly chosen half-edge as before.
If it is paired with a susceptible half-edge, then an epidemic is triggered within the corresponding clique and
associated half-edges, leading to the creation of further infective and recovered half-edges and, unless the 
clique epidemic infects the entire clique, a new susceptible clique having reduced size.  Central limit theorems for the 
final size of epidemics on MR and NSW versions of such random graphs should follow using similar arguments to before but again calculation of the asymptotic 
variances may be difficult.  If the infectious period is constant, so the epidemic model is equivalent to bond percolation on 
the network, the analysis may perhaps be simplified by first splitting the cliques into components determined by bond percolation
and then using similar methods to the present paper treating the components as super-individuals.

\appendix

\section{Detailed derivations}
\label{app:details}

\subsection{Deterministic solution~\eqref{equ:xtilde}-\eqref{equ:zetilde}}
\label{app:det}
First note that~\eqref{equ:xtilde} follows immediately from~\eqref{equ:dxtidt}.  Substituting~\eqref{equ:xtilde}
into~\eqref{equ:dzetdt} yields
\[
\dfrac{d\zet}{dt}= q_I\left[\sum_{i=2}^{\dmax} i(i-1)\xt_i(0)\re^{-it}\right]-\zet
\]
and~\eqref{equ:zetilde} follows using the integrating factor $\re^{t}$.

Multiplying~\eqref{equ:dxtidt} by $i$ and adding over $i=1,2,\dots,\dmax$ yields
\begin{equation}
\label{equ:dxetdt}
\dfrac{d\xet}{dt}=-\sum_{i=2}^{\dmax}i(i-1)\xt_i-\xet.
\end{equation}
Recall that $\etat(t)=\xet(t)+\yet(t)+\zet(t)$.  Summing~\eqref{equ:dxetdt},~\eqref{equ:dyetdt} and~\eqref{equ:dzetdt} gives
\[
\dfrac{d\etat}{dt}=-2\etat,
\]
and~\eqref{equ:eta} follows.
Equation~\eqref{equ:yetilde} then follows since $\yet(t)=\etat(t)-\xet(t)-\zet(t)$ and, using~\eqref{equ:xtilde}, $\xet(t)=\sum_{i=1}^{\dmax}i\xt_i(0)\re^{-it}$.

\subsection{Properties of $\taut$ and $a(\taut)$}
\label{app:tautprop}
In this appendix we show that $\taut\in (0,\infty)$ and $a(\taut)<0$, first in the setting of Theorem~\ref{thm:posclt}
and then in the setting of Theorems~\ref{thm:majclt} and~\ref{thm:zeroclt}.  Let $G(s)=p_I\fde^{(1)}(s)-\mud(s-q_I)$
$(0 \le s \le 1)$. Then, from~\eqref{equ:taut}, $z=\re^{-\taut}$ satisfies $G(z)=0$.  Now $G(0)=p_I\fde^{(1)}(0)+\mud q_I>0$,
unless $q_I=0$ and $p_1-\epsilon_1=0$.  Also $\fde^{(1)}(1)<\mud$, since $p_i\epsilon_i>0$ for at least one $i>0$, so 
$G(1)<0$.  Thus, under the conditions of Theorem~\ref{thm:posclt}, $G(s)$ has at least one zero in $(0,1)$.  Moreover
it has precisely one zero, $z$ say, as $G(s)$ is convex on $[0,1]$, since $G^{(2)}(s)=p_I\fde^{(3)}(s) \ge 0$ for all $s \in [0,1]$.
Hence $\taut=-\log z \in (0,\infty)$, as required, and it follows from~\eqref{equ:ataut} that $a(\taut)<0$ if and only if
$G^{(1)}(z)<0$.  Suppose, for contradiction, that $G^{(1)}(z)\ge 0$.  Then, since $G^{(2)}(s)\ge 0$ for all $s \in [z,1]$,
\[
G(1) \ge G(z) +\int_z^1  G^{(1)}(s)\,{\rm d}s \ge 0.
\]
However, $G(1)<0$, so $G^{(1)}(z)<0$, as required.

Turning to the setting of Theorems~\ref{thm:majclt} and~\ref{thm:zeroclt}, now let $G(s)=p_If_D^{(1)}(s)-\mud(s-q_I)$
$(0 \le s \le 1)$.  Then, from~\eqref{equ:taut1},  $z=\re^{-\taut}$ satisfies $G(z)=0$.  Now $G(1)=0$ and, under the conditions of
Theorems~\ref{thm:majclt} and~\ref{thm:zeroclt}, $G(0)>0$.  Also, using~\eqref{equ:Rzero}, $G^{(1)}(1)=p_I f_D^{(2)}(1)-\mud>0$,
since $R_0>1$.  Thus, since $G(s)$ is convex on $[0,1]$, there exists a unique $z \in (0,1)$ such that $G(z)=0$, so $\taut \in (0, \infty)$.
Moreover, $G^{(2)}(s)>0$ for $s \in [z,1]$ and a similar contradiction argument to before shows that $G^{(1)}(z)<0$, whence
$a(\taut)<0$.

\subsection{Asymptotic variances $\sigmaMR$ and $\sigmaNSW$}
\label{app:asymvar}
In this appendix we derive the expressions for $\sigmaMR$ and $\sigmaNSW$ given in
Theorem~\ref{thm:posclt}.  Recalling the partition $\Delta=\Delta_1 \cup \Delta_2 \cup \Delta_3$ defined just before~\eqref{equ:aDDPP}, it follows from~\eqref{equ:sigma2a} that
\begin{equation}
\label{equ:sigmaMRNSW}
\sigmaMR=\sigma_{0,MR}^2+\sum_{i=1}^3 \sigma_i^2 \qquad \mbox{and}\qquad \sigmaNSW=\sigma_{0,NSW}^2+\sum_{i=1}^3 \sigma_i^2,
\end{equation}
where
\begin{equation}
\label{equ:sigma0MRNSW}
\sigma_{0,MR}^2=\bc(\taut,0)\SigmaMR_0 \bc(\taut,0)^{\top},\qquad  
\sigma_{0,NSW}^2=\bc(\taut,0)\SigmaNSW_0 \bc(\taut,0)^{\top}
\end{equation}
and, for $i=1,2,3$,
\begin{equation*}
\sigma_i^2=\int_0^{\taut} \sum_{\bl \in \Delta_i}\left(\bc(\taut,u)\bl^{\top}\right)^2
\betat_{\bl}(\bwt(u))\,{\rm d}u.
\end{equation*}
We calculate $\sum_{i=1}^3 \sigma_i^2$ in Appendix~\ref{app:sumsigmai}, $\sigma_{0,MR}^2$ and then $\sigmaMR$ in
Appendix~\ref{app:sigma0MR2}, and $\sigma_{0,NSW}^2$ and then $\sigmaNSW$ in Appendix~\ref{app:sigma0NSW2}.

\subsubsection{Calculation of $\sum_{i=1}^3 \sigma_i^2$}
\label{app:sumsigmai}

Noting that $\yet(t) \ge 0$ and $\zet(t) \ge 0$ for $0 \le t \le \taut$,
it follows from~\eqref{equ:intensityfun},~\eqref{equ:cI} and~\eqref{equ:cR} that
\begin{equation}
\label{equ:sigma2b}
\sigma_2^2=\int_0^{\taut} 4 c_I(\taut,u)^2 \yet(u)\,{\rm d}u 
\end{equation}
and
\begin{eqnarray}
\label{equ:sigma3b}
\sigma_3^2&=&\int_0^{\taut} (c_I(\taut,u)+c_R(\taut,u))^2 \zet(u)\,{\rm d}u \\
&=&\int_0^{\taut} (2c_I(\taut,u)-c_J(\taut,u))^2 \zet(u)\,{\rm d}u,\nonumber
\end{eqnarray}
where
\begin{equation}
\label{equ:cJ}
c_J(\taut,u)=b(\taut)\re^{-(\taut-u)}.
\end{equation}

Now 
\begin{equation}
\label{equ:sigma1b}
\sigma_1^2=\int_0^{\taut} \gt(u)\,{\rm d}u,
\end{equation}
where, using~\eqref{equ:intensityfun}, 
\begin{eqnarray}
\label{equ:gtu}
\gt(u)&=&\sum_{\bl \in \Delta_1}\left(\bc(\taut,u)\bl^{\top}\right)^2
\betat_{\bl}(\bwt(u))\nonumber\\
&=&\sum_{i=1}^{\dmax}\sum_{k=0}^{i-1}\left[\bc(\taut,u)(\bl_{ik}^{(1)})^{\top}\right]^2 p_{i,k} i \xt_i(u) .
\end{eqnarray}

Recalling~\eqref{equ:ci}, it follows using~\eqref{equ:intensityfun} that, for
$i=1,2,\dots,\dmax$ and $k=0,1,\dots,i-1$,
\begin{equation}
\label{equ:bcbl1}
\bc(\taut,u)(\bl_{ik}^{(1)})^{\top}=(p_Ib(\taut)-1)\re^{-i(\taut-u)}-2c_I(\taut,u)+[iq_I-(i-k-1)]c_I(\taut,u).
\end{equation}
Recall that 
$p_{i,k}={\rm P}(X=k)$, where $X\sim {\rm Bin}(i-1,1-\exp(-\lambda I))$; see Section~\ref{sec:altconstr}.  Elementary calculation
yields, for $i=1,2,\dots,\dmax$, that
\begin{equation}
\label{equ:sum1}
\sum_{k=0}^{i-1} (i-k-1)p_{i,k}=(i-1)q_I
\end{equation}
and
\begin{equation}
\label{equ:sum2}
\sum_{k=0}^{i-1} (i-k-1)^2p_{i,k}=(i-1)(i-2)q_I^{(2)}+(i-1)q_I.
\end{equation}
Using~\eqref{equ:sum1} and~\eqref{equ:sum2}, it follows from~\eqref{equ:bcbl1}
and some algebra that, for $i=1,2,\dots,\dmax$
\begin{align}
\label{equ:sumcl2}
\sum_{k=0}^{i-1}&\left[\bc(\taut,u)(\bl_{ik}^{(1)})^{\top}\right]^2p_{i,k}
=[2c_I(\taut,u)-q_Ic_J(\taut,u)]^2\\
&\qquad+q_Ip_Ic_J(\taut,u)^2(i-1)+(q_I^{(2)}-q_I^2)c_J(\taut,u)^2(i-1)(i-2)\nonumber\\
&\qquad+2[2c_I(\taut,u)-q_Ic_J(\taut,u)]\re^{-i(\taut-u)}[1-ib(\taut)p_I]\nonumber\\
&\qquad+\re^{-2i(\taut-u)}\left[(1-b(\taut)p_I)^2-(i-1)b(\taut)p_I(2-3b(\taut)p_I)\right.\nonumber\\
&\qquad\qquad\qquad\qquad\left.+(i-1)(i-2)b(\taut)^2p_I^2\right].\nonumber
\end{align}

Recall from~\eqref{equ:MRdetinit} and~\eqref{equ:NSWdetinit} that, under both the
MR and NSW models, $\xt_i(0)=p_i-\epsilon_i$ $(i=0,1,\dots,\dmax)$, so~\eqref{equ:xtilde} 
yields 
\begin{equation}
\label{equ:xit}
\xt_i(u)=(p_i-\epsilon_i) \re^{-i u} \qquad (i=0,1,\dots,\dmax).
\end{equation}
For $i,k=0,1,\dots$,
let $i_{[k]}=i(i-1)\dots(i-k+1)$ denote a falling factorial, with the convention that
$i_{[0]}=1$.  Then it follows from~\eqref{equ:xit} that, for $\theta \in \mathbb{R}$ and
$k=1,2,\dots$,
\[
\sum_{i=1}^{\dmax} i_{[k]} \theta^{i-k}\xt_i(u)=\re^{-ku}\fde^{(k)}(\theta \re^{-u}).
\]
Thus, for $k=1,2,\dots$,
\begin{eqnarray}
\sum_{i=1}^{\dmax} i_{[k]} \xt_i(u)&=&\re^{-ku}\fde^{(k)}(\re^{-u}),\label{equ:facsum1}\\
\sum_{i=1}^{\dmax} i_{[k]} \re^{-i(\taut-u)} \xt_i(u)&=&\re^{-k \taut}\fde^{(k)}(\re^{-\taut}),\label{equ:facsum2}\\
\sum_{i=1}^{\dmax} i_{[k]} \re^{-2i(\taut-u)} \xt_i(u)&=&\re^{-k(2\taut-u)}\fde^{(k)}(\re^{-(2\taut-u)}).\label{equ:facsum3}
\end{eqnarray}
Substituting~\eqref{equ:sumcl2} into~\eqref{equ:gtu} and using~\eqref{equ:facsum1}-\eqref{equ:facsum3} yields
\begin{align*}
\gt(u)&=\left[2c_I(\taut,u)-q_Ic_J(\taut,u)\right]^2\xet(u)+c_J(\taut,u)^2p_Iq_I\re^{-2u}\fde^{(2)}(\re^{-u})\\
&\qquad+c_J(\taut,u)^2(q_I^{(2)}-q_I^2)\re^{-3u}\fde^{(3)}(\re^{-u})\\
&\qquad+2\left[2c_I(\taut,u)-q_Ic_J(\taut,u)\right][1-b(\taut)p_I]\re^{-\taut}\fde^{(1)}(\re^{-\taut})\\
&\qquad-2b(\taut)p_I\left[2c_I(\taut,u)-q_Ic_J(\taut,u)\right]\re^{-2\taut}\fde^{(2)}(\re^{-\taut})\\
&\qquad+\left[1-b(\taut)p_I\right]^2\re^{-(2\taut-u)}\fde^{(1)}(\re^{-(2\taut-u)})\\
&\qquad-b(\taut)p_I\left[2-3b(\taut)p_I\right]\re^{-2(2\taut-u)}\fde^{(2)}(\re^{-(2\taut-u)})\\
&\qquad+b(\taut)^2p_I^2 \re^{-3(2\taut-u)}\fde^{(3)}(\re^{-(2\taut-u)}).
\end{align*}
It then follows using~\eqref{equ:sigma2b},~\eqref{equ:sigma3b} and~\eqref{equ:sigma1b} that
\begin{equation}
\label{equ:sigma123}
\sigma_1^2+\sigma_2^2+\sigma_3^2=\sum_{i=1}^{10} I_i,
\end{equation}
where
\[
I_i=\int_0^{\taut} f_i(u)\,{\rm d}u,
\]
with
\begin{align*}
f_1(u)&=4c_I(\taut,u)^2\left(\xet(u)+\yet(u)+\zet(u)\right),\\
f_2(u)&=\left[c_J(\taut,u)^2-4c_I(\taut,u)c_J(\taut,u)\right]]\zet(u),\\
f_3(u)&=q_I\left[q_Ic_J(\taut,u)^2-4c_I(\taut,u)c_J(\taut,u)\right]]\xet(u),\\
f_4(u)&=c_J(\taut,u)^2p_Iq_I\re^{-2u}\fde^{(2)}(\re^{-u}),\\
f_5(u)&=c_J(\taut,u)^2(q_I^{(2)}-q_I^2)\re^{-3u}\fde^{(3)}(\re^{-u}),\\
f_6(u)&=2\left[2c_I(\taut,u)-q_Ic_J(\taut,u)\right][1-b(\taut)p_I]\re^{-\taut}\fde^{(1)}(\re^{-\taut}),\\
f_7(u)&=-2b(\taut)p_I\left[2c_I(\taut,u)-q_Ic_J(\taut,u)\right]\re^{-2\taut}\fde^{(2)}(\re^{-\taut}),\\
f_8(u)&=\left[1-b(\taut)p_I\right]^2\re^{-(2\taut-u)}\fde^{(1)}(\re^{-(2\taut-u)}),\\
f_9(u)&=-b(\taut)p_I\left[2-3b(\taut)p_I\right]\re^{-2(2\taut-u)}\fde^{(2)}(\re^{-(2\taut-u)}),\\
f_{10}(u)&=b(\taut)^2p_I^2 \re^{-3(2\taut-u)}\fde^{(3)}(\re^{-(2\taut-u)}).
\end{align*}

Noting that $\etat(0)=\mud$, it follows using~\eqref{equ:cI} and~\eqref{equ:eta} that
\begin{equation}
\label{equ:I1}
I_1=2b(\taut)^2\mud \re^{-2\taut}\left(1-\re^{-2\taut}\right).
\end{equation}

Recall from~\eqref{equ:MRdetinit} and~\eqref{equ:NSWdetinit} that under both the MR and NSW models,
$\xet(0)=\sum_{i=1}^{\dmax} i (p_i-\epsilon_i)$ and $\zet(0)=q_I\sum_{i=1}^{\dmax} i \epsilon_i$.
It then follows from~\eqref{equ:zetilde} and~\eqref{equ:xit} that 
\[
\zet(u)=q_I\left(\mud \re^{-u}-\xet(u)\right),
\]
so
\[
f_2(u)+f_3(u)=c_J(\taut,u)\mud q_I\left[c_J(\taut,u)-4c_I(\taut,u)\right]\re^{-u}-c_J(\taut,u)^2q_Ip_I\xet(u).
\]
Also, setting $k=1$ in~\eqref{equ:facsum1} yields $\xet(u)=\re^{-u}\fde^{(1)}(\re^{-u})$.  Thus,
\begin{align}
I_2+I_3+I_4&=\int_0^{\taut}c_J(\taut,u)\mud q_I\left[c_J(\taut,u)-4c_I(\taut,u)\right]\re^{-u}\,{\rm d}u\label{equ:I234a}\\
&\qquad+q_Ip_I\int_0^{\taut}c_J(\taut,u)^2\left[\re^{-2u}\fde^{(2)}(\re^{-u})-\re^{-u}\fde^{(1)}(\re^{-u})\right]\,{\rm d}u\nonumber.
\end{align}
The first integral in~\eqref{equ:I234a} involves only exponential functions and is easily evaluated.  Using~\eqref{equ:cJ}, the integrand
in the second integral in~\eqref{equ:I234a} can be expressed as 
\[
-b(\taut)^2\re^{-2\taut}\dfrac{d}{dt}\left[\re^u\fde^{(1)}(\re^{-u})\right],
\]
so that integral is also easily evaluated.  Hence, omitting the details,
\begin{align}
\label{equ:I234}
I_2+I_3+I_4&=2b(\taut)^2\mud\left(\re^{-2\taut}-\re^{-4\taut}\right)-b(\taut)^2\mud q_I\re^{-\taut}\left(1+\re^{-\taut}-2\re^{-2\taut}\right)\\
&\qquad+q_Ip_Ib(\taut)^2\left[\re^{-2\taut}\fde^{(1)}(1)-\re^{-\taut}\fde^{(1)}(\re^{-\taut})\right].\nonumber
\end{align}

Turning to $I_5$, note that
\[
c_J(\taut,u)^2\re^{-3u}\fde^{(3)}(\re^{-u})
=-b(\taut)^2\re^{-2\taut}\dfrac{d}{dt}\left[\fde^{(2)}(\re^{-u})\right],
\]
so
\begin{equation}
\label{equ:I5}
I_5=b(\taut)^2 \left(q_I^{(2)}-q_I^2\right)\re^{-2\taut}\left[\fde^{(2)}(1)-\fde^{(2)}(\re^{-\taut})\right].
\end{equation}
The integrals $I_6$ and $I_7$ are easily evaluated yielding
\begin{align}
\label{equ:I67}
I_6+I_7&=2b(\taut)\left\{\left[1-p_Ib(\taut)\right]\re^{-\taut}\fde^{(1)}(\re^{-\taut})-p_Ib(\taut)\re^{-2\taut}\fde^{(2)}(\re^{-\taut})\right\}\\
&\qquad\times\left(p_I-\re^{-2\taut}+q_I\re^{-\taut}\right).\nonumber
\end{align}

To evaluate $I_{8}$ to $I_{10}$, let $J_k=\int_0^{\taut} \re^{-k(2\taut-u)}\fde^{(k)}(\re^{-(2\taut-u)})\,{\rm d}u$ $(k=1,2,\dots)$.  Then
a simple reduction formula yields
\begin{equation*}
J_k=\re^{-(k-1)\taut}\fde^{(k-1)}(\re^{-\taut})-\re^{-2(k-1)\taut}\fde^{(k-1)}(\re^{-2\taut})-(k-1)J_{k-1},
\end{equation*}
for $k=2,3,\dots$, and
\[
J_1=\fde(\re^{-\taut})-\fde(\re^{-2\taut}).
\]
Applying these formulae to $I_{8}$ to $I_{10}$ yields after some algebra that
\begin{align}
\label{equ:I8910}
I_8+I_9+I_{10}&=\fde(\re^{-\taut})-\fde(\re^{-2\taut})\\
&\quad+b(\taut)p_I(b(\taut)p_I-2)
\left[\re^{-\taut}\fde^{(1)}(\re^{-\taut})-\re^{-2\taut}\fde^{(1)}(\re^{-2\taut})\right]\nonumber\\
&\quad+b(\taut)^2p_I^2\left[\re^{-2\taut}\fde^{(2)}(\re^{-\taut})-\re^{-4\taut}\fde^{(2)}(\re^{-2\taut})\nonumber\right].
\end{align}

Letting $z=\re^{-\taut}$, it follows from~\eqref{equ:sigma123} and the above equations that
\begin{align}
\label{equ:sigma123a}
\sum_{i=1}^3 \sigma_i^2 = &2\mud b(\taut)^2(z^2-z^4)-\mud q_I b(\taut)^2 (z+z^2-2z^3)\\
&+ p_I q_I b(\taut)^2\left[z^2\fde^{(1)}(1)-z\fde^{(1)}(z)\right]+p_I^2b(\taut)^2\left[z^2 \fde^{(2)}(z)-z^4\fde^{(2)}(z^2)\right]\nonumber\\
&+2b(\taut)\left[(1-p_Ib(\taut))z\fde^{(1)}(z)-p_Ib(\taut)z^2\fde^{(2)}(z)\right]\left(p_I-z^2+q_Iz\right)\nonumber\\
&+\fde(z)-\fde(z^2)+p_Ib(\taut)\left(p_Ib(\taut)-2\right)\left[z\fde^{(1)}(z)-z^2\fde^{(1)}(z^2)\right]\nonumber\\
&+(q_I^{(2)}-q_I^2)b(\taut)^2 z^2 \left[\fde^{(2)}(1)-\fde^{(2)}(z)\right].\nonumber
\end{align}

\subsubsection{Calculation of $\sigma_{0,MR}^2$ and $\sigmaMR$}
\label{app:sigma0MR2}
To determine $\sigma_{0,MR}^2$, note from~\eqref{equ:sigma0MRNSW}, \eqref{equ:SigmaMR0} and~\eqref{equ:bc}
that
\begin{eqnarray}
\label{equ:sigma0MR}
\sigma_{0,MR}^2&=&\left[c_I(\taut,0)-c_R(\taut,0)\right]^2 \sigma_Y^2\\
&=&b(\taut)^2z^2 \sigma_Y^2,\nonumber
\end{eqnarray}
using~\eqref{equ:cI}, \eqref{equ:cR} and $z=\re^{-\taut}$.  Now $\sigma_Y^2$ is given by~\eqref{equ:sigma2Y},
where, for $i=1,2,\dots,\dmax$, $\sigma_{Y,i}^2={\rm var}(Y_{i1})$ with $Y_{i1} \sim {\rm Bin}(i,1-\exp(-\lambda I))$.
A simple calculation, conditioning on $I$, yields
\begin{equation}
\label{equ:sigmaIi}
\sigma_{Y,i}^2=i(i-1)q_I^{(2)}+iq_I-i^2 q_I^2.
\end{equation}
Now $\sum_{i=1}^{\dmax}i \epsilon_i=\mud-\fde^{(1)}(1)$ and $\sum_{i=1}^{\dmax}i(i-1) \epsilon_i=f_D^{(2)}(1)-\fde^{(2)}(1)$, so
\begin{equation}
\label{equ:sigmaY2}
\sigma_Y^2=\left[(q_I^{(2)}-q_I^2)\left(f_D^{(2)}(1)-\fde^{(2)}(1)\right)+p_Iq_I\left(\mud-\fde^{(1)}(1)\right)\right].
\end{equation}

Using~\eqref{equ:sigmaMRNSW}, adding~\eqref{equ:sigma123a} and~\eqref{equ:sigma0MR}, after substituting from~\eqref{equ:sigmaY2}, yields
\begin{align}
\label{equ:sigmaMR}
\sigmaMR=&2\mud b(\taut)^2(z^2-z^4)-\mud q_I b(\taut)^2 (z+q_Iz^2-2z^3)-p_I q_I b(\taut)^2z\fde^{(1)}(z)\\
&+2b(\taut)\left[(1-p_Ib(\taut))z\fde^{(1)}(z)-p_Ib(\taut)z^2\fde^{(2)}(z)\right]\left(p_I-z^2+q_Iz\right)\nonumber\\
&+\fde(z)-\fde(z^2)+p_Ib(\taut)\left(p_Ib(\taut)-2\right)\left[z\fde^{(1)}(z)-z^2\fde^{(1)}(z^2)\right]\nonumber\\
&- p_I q_I b(\taut)^2 z\fde^{(1)}(z)+b(\taut)^2p_I^2z^2\left[\fde^{(2)}(z)-z^2\fde^{(2)}(z^2)\right]\nonumber\\
&+(q_I^{(2)}-q_I^2)b(\taut)^2 z^2\left[f_D^{(2)}(1)-\fde^{(2)}(z)\right].\nonumber
\end{align}
Using~\eqref{equ:taut} and recalling that $z=\re^{-\taut}$ yields
\begin{equation}
\label{equ:fde1zequation}
\mud(z-q_I)=p_I\fde^{(1)}(z).
\end{equation}
Also, using~\eqref{equ:ataut}, \eqref{equ:btaut} and~\eqref{equ:facsum1}, with $k=1$, gives
\begin{equation}
\label{equ:btaut1}
b(\taut)=\frac{\fde^{(1)}(z)}{z\left(p_I\fde^{(2)}(z)-\mud\right)},
\end{equation}
so
\begin{equation}
\label{equ:fde2zequation}
b(\taut)zp_I\fde^{(2)}(z)=b(\taut)z\mud+\fde^{(1)}(z).
\end{equation}
Substituting~\eqref{equ:fde1zequation} and~\eqref{equ:fde2zequation} into~\eqref{equ:sigmaMR}, recalling~\eqref{equ:rho}
and noting from \eqref{equ:fde1zequation} and~\eqref{equ:btaut1} that $h(z)=b(\taut)z$ yields~\eqref{equ:sigmamr} after some algebra.

\subsubsection{Calculation of $\sigma_{0,NSW}^2$ and $\sigmaNSW$}
\label{app:sigma0NSW2}

To determine $\sigma_{0,NSW}^2$, first note that~\eqref{equ:sigma0MRNSW},~\eqref{equ:SigmaNSW0} and~\eqref{equ:bc} imply that
\begin{equation}
\label{equ:sigma0nsw}
\sigma_{0,NSW}^2=(1-\epsilon)\sigma_A^2+\epsilon\sigma_B^2,
\end{equation}
where
\begin{equation}
\label{equ:sigmaA2}
\qquad\sigma_A^2=\bc(\taut,0) \Sigma_{XX}\bc(\taut,0)^{\top}
\end{equation}
and
\begin{equation}
\label{equ:sigmaB2}
\qquad\sigma_B^2=c_I(\taut,0)^2\sigma_{Y_E}^2+2c_I(\taut,0)c_R(\taut,0)\sigma_{Y_E, Z_E}+c_R(\taut,0)^2\sigma_{Z_E}^2.
\end{equation}
Recalling~\eqref{equ:SigmaXX} and~\eqref{equ:bcs},
\begin{equation}
\label{equ:csigmaXXct}
\bc(\taut,0) \Sigma_{XX}\bc(\taut,0)^{\top}=\sum_{i=0}^{\dmax} p_i c_i(\taut,0)^2-\left(\sum_{i=0}^{\dmax}p_i c_i(\taut,0)\right)^2,
\end{equation}
where from~\eqref{equ:ci} and recalling that $z=\re^{-\taut}$,
\begin{equation*}
c_i(\taut,0)=z^i+ib(\taut)z(z-q_I)-p_Ib(\taut)i z^i.
\end{equation*}
Thus,
\begin{equation*}
\sum_{i=0}^{\dmax}p_i c_i(\taut,0)=f_D(z)+b(\taut)z(z-q_I)\mud-p_Ib(\taut)zf_D^{(1)}(z)
\end{equation*}
and
\begin{align*}
\sum_{i=0}^{\dmax}p_i c_i(\taut,0)^2&=f_D(z^2)+b(\taut)^2z^2(z-q_I)^2(\sigma_D^2+\mud^2)\\
&\,+p_I^2b(\taut)^2z^2\left(z^2f_D^{(2)}(z^2)+f_D^{(1)}(z^2)\right)
+2b(\taut)z^2(z-q_I)f_D^{(1)}(z)\\
&\,-2p_Ib(\taut)z^2f_D^{(1)}(z^2)-2p_Ib(\taut)^2z^2(z-q_I)\left(zf_D^{(2)}(z)+f_D^{(1)}(z)\right).
\end{align*}

Note that in the NSW model
\begin{equation}
\label{equ:fdnsw}
\fde(s)=(1-\epsilon)f_D(s)\qquad(s \in \mathbb{R}).
\end{equation}
Hence, using~\eqref{equ:fde1zequation},
\begin{equation*}
\sum_{i=0}^{\dmax}p_i c_i(\taut,0)=f_D(z)-\frac{\epsilon}{1-\epsilon}b(\taut)z(z-q_I)\mud
\end{equation*}
and, using~\eqref{equ:fde1zequation} and~\eqref{equ:fde2zequation},
\begin{align*}
(1-\epsilon)&\sum_{i=0}^{\dmax}p_i c_i(\taut,0)^2=\fde(z^2)
+p_I^2b(\taut)^2z^2\left(z^2\fde^{(2)}(z^2)+\fde^{(1)}(z^2)\right)\\
&\qquad-2p_Ib(\taut)z^2\fde^{(1)}(z^2)+(1-\epsilon)b(\taut)^2z^2(z-q_I)^2\left(\sigma_D^2+\mud^2\right)\\
&\qquad-2b(\taut)^2z^3(z-q_I)\mud-2b(\taut)^2z^2(z-q_I)^2\mud.
\end{align*}
It then follows using~\eqref{equ:sigmaA2} and~\eqref{equ:csigmaXXct} that
\begin{align}
\label{equ:1mepssigmaa2}
(1-\epsilon)&\sigma_A^2=\fde(z^2)-(1-\epsilon)f_D(z)^2
+p_I^2b(\taut)^2z^2\left(z^2\fde^{(2)}(z^2)+\fde^{(1)}(z^2)\right)\\
&\qquad-2p_Ib(\taut)z^2\fde^{(1)}(z^2)+2\epsilon b(\taut)z(z-q_I)\mud f_D(z)\nonumber\\
&\qquad+b(\taut)^2z^2(z-q_I)^2\left[(1-\epsilon)\sigma_D^2+\frac{1-2\epsilon}{1-\epsilon}\mud^2-2\mud\right]\nonumber\\
&\qquad-2b(\taut)^2z^3(z-q_I)\mud.\nonumber
\end{align}

Recall the definition of $(Y_E,Z_E)$ just before~\eqref{equ:SigmaXX}.  Now ${\rm E}[Y_E|D]=p_ID$ and, 
using~\eqref{equ:sigmaIi}, ${\rm var}(Y_E|D)=D(D-1)q_I^{(2)}+Dq_I-D^2q_I^2$, so
\begin{eqnarray}
\label{equ:sigmaYe2}
\sigma_{Y_E}^2&=&{\rm E}\left[{\rm var}(Y_E|D)\right]+{\rm var}({\rm E}[Y_E|D])\\
&=&(q_I^{(2)}-q_I^2)f_D^{(2)}(1)+p_Iq_I\mud+p_I^2\sigma_D^2.\nonumber
\end{eqnarray}
Similar arguments show that
\begin{equation}
\label{equ:sigmaZe2}
\sigma_{Z_E}^2=(q_I^{(2)}-q_I^2)f_D^{(2)}(1)+p_Iq_I\mud+q_I^2\sigma_D^2
\end{equation}
and
\begin{equation}
\label{equ:sigmaYeZe2}
\sigma_{Y_E,Z_E}=-\left[(q_I^{(2)}-q_I^2)f_D^{(2)}(1)+p_Iq_I\mud\right]+p_Iq_I\sigma_D^2.
\end{equation}
Thus, using~\eqref{equ:sigmaB2},~\eqref{equ:cI},~\eqref{equ:cR} and $z=\re^{-\taut}$,
\begin{equation*}
\sigma_B^2=b(\taut)^2z^2\left[(q_I^{(2)}-q_I^2)f_D^{(2)}(1)+p_Iq_I\mud+\sigma_D^2(z-q_I)^2\right].
\end{equation*}
Noting from~\eqref{equ:fdnsw} that $\fde^{(1)}(1)=(1-\epsilon)\mud$ and $\fde^{(2)}(1)=(1-\epsilon)f_D^{(2)}(1)$,
it follows using~\eqref{equ:sigma0MR} that
\begin{equation}
\label{equ:sigmaB2m0MR}
\epsilon\sigma_B^2-\sigma_{0,MR}^2=\epsilon b(\taut)^2z^2(z-q_I)^2\sigma_D^2.
\end{equation}
Note from~\eqref{equ:sigmaMRNSW} that $\sigmaNSW-\sigmaMR=\sigma_{0,NSW}^2-\sigma_{0,MR}^2$.
Hence, using~\eqref{equ:sigma0nsw},~\eqref{equ:1mepssigmaa2} and~\eqref{equ:sigmaB2m0MR},
\begin{align}
\label{equ:sigmanswmmr}
\sigmaNSW-\sigmaMR&=\fde(z^2)-(1-\epsilon)f_D(z)^2
+p_I^2b(\taut)^2z^2\left(z^2\fde^{(2)}(z^2)+\fde^{(1)}(z^2)\right)\\
&\qquad-2p_Ib(\taut)z^2\fde^{(1)}(z^2)+2\epsilon b(\taut)z(z-q_I)\mud f_D(z)\nonumber\\
&\qquad+b(\taut)^2z^2(z-q_I)^2\left[\sigma_D^2+\frac{1-2\epsilon}{1-\epsilon}\mud^2-2\mud\right]\nonumber\\
&\qquad-2b(\taut)^2z^3(z-q_I)\mud.\nonumber
\end{align}

Recall that $h(z)=b(\taut)z$. Note from~\eqref{equ:rho} and~\eqref{equ:fdnsw} that \newline
$f_D(z)=(1-\epsilon-\rho)/(1-\epsilon)$ and also that $\fde^{(2)}(z)=(1-\epsilon)f_D^{(2)}(z)$.
Adding~\eqref{equ:sigmamr} and~\eqref{equ:sigmanswmmr} then yields~\eqref{equ:sigmansw}.

\subsubsection{Proof of Remark~\ref{rmk:varcomp}}
\label{app:proofrmkvarcomp}

To prove the assertions in Remark~\ref{rmk:varcomp}, suppose that the support of $D$ is not concentrated on a single point.
Then it is shown easily using~\eqref{equ:SigmaXX} that the matrix $\Sigma_{XX}$ is positive definite and
it follows immediately from~\eqref{equ:sigma0nsw},~\eqref{equ:sigmaA2} and~\eqref{equ:sigmaB2m0MR} that 
$\sigma_{0,NSW}^2>\sigma_{0,MR}^2$, whence, using~\eqref{equ:sigmaMRNSW}, $\sigmaNSW >\sigmaMR$. 
Further, setting $\epsilon=0$ yields $\sigmatNSW > \sigmatMR$.

\subsection{Asymptotic variances $\sigmaMRsite$ and $\sigmaNSWsite$}
\label{app:asymvarperc}
For $i=1,2,3$, let
\begin{equation*}
\bar{\sigma}_i^2=\int_0^{\taut} \sum_{\blhat \in \hat{\Delta}_i}\left(\bcbar(\taut,u)\blhat^{\top}\right)^2
\betabar_{\bl}(\bwbar(u))\,{\rm d}u.
\end{equation*}
Note from~\eqref{equ:bchat} that $\bcbar(\taut,u)(\blhat_+^{(2)})^{\top}=-\pi\bc(\taut,u)(\bl_+^{(2)})^{\top}$, 
so $\bar{\sigma}_2^2=\pi^2 \sigma_2^2$ since $\betabar_{\bl}(\bwbar(u))$ is independent of $\vbar$ for $\bl \in \hat{\Delta}_2$.
Similarly, $\bar{\sigma}_3^2=\pi^2 \sigma_3^2$

Turning to $\bar{\sigma}_1^2$, note that for $i=1,2,\dots,\dmax$, 
\[
\bcbar(\taut,u)(\blhat^{(1)}_{i,0})^{\top}=-\pi\left[\bc (\bl^{(1)}_{i,0})^{\top}\right]\quad\mbox{and}\quad
\bcbar(\taut,u)(\blhat^{(1)}_{i,i-1})^{\top}=1-\pi\left[\bc (\bl^{(1)}_{i,i-1}))^{\top}\right],
\]
whence, 
\begin{align*}
&(1-\pi)\left(\bcbar(\taut,u)(\blhat^{(1)}_{i,0})^{\top}\right)^2+\pi\left(\bcbar(\taut,u)(\blhat^{(1)}_{i,i-1})^{\top}\right)^2\\
&\qquad=\pi(1-\pi)
+\pi^2\left[(1-\pi)\left(\bc(\taut,u)\bl^{(1)}_{i,0})^{\top}\right)^2+\pi\left(\bc(\taut,u)\bl^{(1)}_{i,i-1})^{\top}\right)^2\right]\\
&\qquad\qquad+2\pi^2(1-\pi)\bc(\taut,u)\left(\bl^{(1)}_{i,0}-\bl^{(1)}_{i,i-1}\right)^{\top}.
\end{align*}
Now $\bl^{(1)}_{i,0}-\bl^{(1)}_{i,i-1}=(i-1)(\ber-\bei)$, so, using~\eqref{equ:bc},\eqref{equ:cI} and~\eqref{equ:cR},
\[
\bc(\taut,u)\left(\bl^{(1)}_{i,0}-\bl^{(1)}_{i,i-1}\right)^{\top}=-(i-1)c_J(\taut,u),
\]
where $c_J(\taut,u)=b(\taut)\re^{-(\taut-u)}$ (see~\eqref{equ:cJ}).  Thus, using~\eqref{equ:intensityfunperc}, $\bar{\sigma}_1^2=\int_0^{\taut}\bar{g}(u)\,{\rm d}u$, where
\begin{eqnarray*}
\bar{g}(u)&=&\sum_{i=1}^{\dmax}\left[(1-\pi)\left(\bcbar(\taut,u)(\blhat^{(1)}_{i,0})^{\top}\right)^2
+\pi\left(\bcbar(\taut,u)(\blhat^{(1)}_{i,i-1})^{\top}\right)^2\right]i\xt_i(u)\\
&=&\pi^2\tilde{g}(u)+\sum_{i=1}^{\dmax}\pi(1-\pi)\left[1-2(i-1)\pi c_J(\taut,u)\right]i\xt_i(u),
\end{eqnarray*}
and $\tilde{g}(u)$ is given by~\eqref{equ:gtu} with $p_{i,i-1}=\pi=1-p_{i,0}$.

Now $\xt_i(0)=p_i$ ($i=0,1,\dots,\dmax$), so~\eqref{equ:xtilde} yields $\sum_{i=1}^{\dmax}i_{[k]}\xt_i(u)=\re^{-ku}f_D^{(k)}(\re^{-u})$ $(k=1,2)$. Hence,
\begin{equation*}
\bar{\sigma}_1^2=\pi^2 \sigma_1^2+\pi(1-\pi)\int_0^{\taut}\re^{-u}f_D^{(1)}(u)\,{\rm d}u
-2\pi^2(1-\pi)b(\taut)\re^{-\taut}\int_0^{\taut}\re^{-u}f_D^{(2)}(u)\,{\rm d}u.
\end{equation*}
Thus, recalling that $z=\re^{-\taut}$ and $\rho=1-f_D(z)$,
\[
\int_0^{\taut}\re^{-u}f_D^{(1)}(u)\,{\rm d}u=f_D(1)-f_D(\re^{-\taut})=\rho
\]
and
\[
\int_0^{\taut}\re^{-u}f_D^{(2)}(u)\,{\rm d}u=f_D^{(1)}(1)-f_D^{(1)}(\re^{-\taut})=\mud-f_D^{(1)}(z).
\]
Further, \eqref{equ:zperc} implies $\pi[\mud-f_D^{(1)}(z)]=\mud(1-z)$, so
\[
\sum_{i=1}^3 \bar{\sigma}_i^2=\pi^2\sum_{i=1}^3 \sigma_i^2+\pi(1-\pi)[\rho-2z(1-z)b(\taut)\mud]
\]
and~\eqref{equ:sumsigmaperc} follows since $h(z)=b(\taut)z$.

For the MR random graph, setting $\epsilon_i=0$ ($i=1,2,\dots,\dmax$) in~\eqref{equ:sigma2Y}, using~\eqref{equ:SigmaMR0}
and recalling that $\Vn(0)=0$ for all $n$, shows that $\hat{\Sigma}_0=0$ and~\eqref{equ:bchatSigma0} follows.
For the NSW random graph, a similar argument setting $\epsilon=0$ in~\eqref{equ:SigmaNSW0} and using~\eqref{equ:bchat}
yields
\begin{equation}
\label{equ:bchatSima0nsw}
\bcbar(\taut,0)\Sigmabar_0 \bcbar(\taut,0)^{\top}=(\bc(\taut,0)+\bone)\Sigma_{XX}(\bc(\taut,0)+\bone)^{\top}.
\end{equation}
A simple calculation using~\eqref{equ:SigmaXX} (or noting that $\Sigma_{XX}$ is the variance matrix of
a single multinomial trial) gives $\bone\Sigma_{XX}\bone^{\top}=\bc(\taut,0)\Sigma_{XX}\bone^{\top}=0$,
so~\eqref{equ:bchatSigma0} follows from~\eqref{equ:bchatSima0nsw}.




\begin{thebibliography}{1}

\bibitem[Andersson(1998)]{Andersson:1998}
Andersson, H. (1998) Limit theorems for a random graph epidemic
model. {\sl Ann.~Appl.~Prob.} {\bf 8}, 1331--1349.











\bibitem[Ball {\it et al.}(2018)]{BBLS:2018}
Ball, F., Britton, T., Leung, K.Y.~and Sirl, D.~(2018) A stochastic SIR network
epidemic model with preventive dropping of edges.
Submitted to {\sl J.~Math.~Biol.}


\bibitem[Ball and Neal(2008)]{BN:2008}
Ball, F. and Neal, P.~(2008) Network epidemic models with two levels of mixing.
{\sl Math.~Biosci.} {\bf 212}, 69--87.

\bibitem[Ball and Neal(2017)]{BN:2017}
Ball, F. and Neal, P.~(2017) The asymptotic variance of the giant component of
configuration model random graphs.
{\sl Ann.~Appl.~Probab.} {\bf 27}, 1057--1092.

\bibitem[Ball and Sirl(2013)]{BS:2013}
Ball, F. and Sirl, D.~(2013) Acquaintance vaccination in an epidemic on a random graph with specified degree.
{\sl J.~Appl.~Prob.} {\bf 50}, 1147--1168.

\bibitem[Ball {\it et al.}(2010)]{BST:2010}
Ball, F., Sirl, D. and Trapman, P.~(2013) Analysis of a stochastic SIR epidemic on a random network
incorporating household structure.
{\sl Math.~Biosci.} {\bf 224}, 53--73.

\bibitem[Barbour and Luczak(2012)]{BL:2012}
Barbour, A. D. and Luczak, M. J.~(2012) Central limit approximations for {M}arkov population processes with countably many types.
{\sl Electron.~J.~Probab.} {\bf 17}(90),1--16.

\bibitem[Barbour and Reinert(2013)]{BR:2013}
Barbour, A. and Reinert, G.~(2013) Approximating the epidemic curve.
{\sl Electron.~J.~Probab.} {\bf 18}(54), 1--30. 

\bibitem[Barbour and R{\"o}llin(2017)]{BR:2017}
Barbour, A.D. and R{\"o}llin, A.~(2017) Central limit theorems in the configuration model.
{\sl arXiv}:1710.02644v1.




\bibitem[Bollob\'as(1980)]{Bollobas:1980}
Bollob\'as, B.~(1980) A probabilistic proof of an asymptotic formula for the number of labelled regular graphs.
{\sl European J.~Combin.} {\bf 1}, 311-316.


\bibitem[Bohman and Picollelli(2012)]{BP:2012}
Bohman, T. and Picollelli, M.~(2012) SIR epidemics on random graphs with a fixed degree sequence.
{\sl Random Structures Algorithms} {\bf 41}, 179--214.

\bibitem[Britton {\it et al.}(2007)]{BJML:2007}
Britton, T., Janson, S. and Martin-L{\"o}f, A.~(2007) Graphs with specified degree distributions,
simple epidemics and local vaccination strategies.
{\sl Adv.~Appl.~Prob.} {\bf 39}, 922--948.

\bibitem[Coupechoux and Lelarge(2014)]{CL:2014}
Coupechoux, E. and Lelarge, M.~(2014) How clustering affects epidemics in random networks.
{\sl Adv.~Appl.~Prob.} {\bf 46}, 985--1008.

\bibitem[Durrett(2007)]{Durrett:2007}
Durrett, R. (2007)
{\it Random Graph Dynamics}.
Cambridge University Press, Cambridge.

\bibitem[Decreusefond {\it et al.}(2012)]{DDMT:2012}
Decreusefond, L., Dhersin, J.-S., Moyal, P. and Tran, V. C.~(2012)
Large graph limit for an SIR process in random network with heterogeneous connectivity.
{\sl Ann.~Appl.~Probab.} {\bf 22}, 541--575.




\bibitem[Ethier and Kurtz(1986)]{EK86}
Ethier, S.N. and Kurtz, T.G.~(1986)
{\it Markov Processes: Characterization and Convergence}.
Wiley, New York.

\bibitem[Gleeson(2009)]{Gleeson:2009}
Gleeson, J.P. (2009) Bond percolation on a class of clustered random networks.
{\sl Phys.~Rev.~E} {\bf 80}, 036107.

\bibitem[van der Hofstad(2016)]{vdHofstad:2016}
van der Hofstad, R.~(2016)
{\it Random Graphs and Complex Networks Volume 1}.
Cambridge University Press, Cambridge.


\bibitem[Janson(2009a)]{Janson:2009a}
Janson, S.~(2009a)
On percolation in random graphs with given vertex degrees.
{\sl Electron.~J.~Probab.} {\bf 14}, 87--118.

\bibitem[Janson(2009b)]{Janson:2009b}
Janson, S.~(2009b)
The probability that a random multigraph is simple.
{\sl Combin.~Prob.~Comput.} {\bf 18}, 205--225.

\bibitem[Janson(2010)]{Janson:2010}
Janson, S.~(2010)
Asymptotic equivalence and contiguity of some random graphs.
{\sl  Random Structures Algorithms} {\bf 36}, 26--45.

\bibitem[Janson {\it et al.}(2012)]{JLW:2014}
Janson, S., Luczak, M. and Windridge, P.~(2014)
Law of large numbers for the {SIR} epidemic on a random graph with given degrees.
{\sl Random Structures Algorithms} {\bf 45}, 726--763.

\bibitem[Kenah and Robins(2007)]{Kenah07}
Kenah, E. and Robins, J.M.~(2007)
Second look at the spread of epidemics on networks.
{\sl Phys.~Rev.~E} {\bf 76}, 036113.

\bibitem[KhudaBukhsh {\it et al.}(2017)]{KWRK:2017}
KhudaBukhsh, W. R., Woroszylo, C., Rempala, G. A. and Koeppl, H.~(2017)
Functional central limit theorem for susceptible-infected process on configuration model graphs.
{\sl arXiv}:1703.06328v1.

\bibitem[Kiss {\it et al.}(2017)]{KMS:2017}
Kiss, I. Z., Miller, J. C. and Simon, P.~(2017)
{\it Mathematics of Epidemics on Networks: From Exact to Approximate Models}.
Springer.

\bibitem[Kurtz(1970)]{Kurtz:1970}
Kurtz, T. G.~(1970) Solutions of ordinary differential equations as limits of
pure jump Markov processes.
{\sl J.~Appl.~Prob.} {\bf 7}, 49--58.

\bibitem[Kurtz(1971)]{Kurtz:1971}
Kurtz, T. G.~(1971) Limit theorems for sequences of jump Markov processes approximating
ordinary differential equations.
{\sl J.~Appl.~Prob.} {\bf 8}, 344--356.

\bibitem[Martin-L{\"o}f(1986)]{ML:1986}
Martin-L{\"o}f, A.~(1986) Symmetric sampling procedures, general epidemic processes and their threshold limit theorems.
{\sl J.~Appl.~Prob.} {\bf 23}, 265--282.

\bibitem[Miller(2011)]{Miller:2011}
Miller, J.C.~(2011) A note on a paper by Erik Volz: SIR dynamics in random networks.
{\sl J.~Math.~Biol.} {\bf 62}, 349--358.

\bibitem[Miller {\it et al.}(2012)]{MSV:2012}
Miller, J.C., Slim, A.C. and Volz, E.M.~(2012)
Edge-based compartmental modelling for infectious disease spread.
{\sl J.~R.~Soc.~Interface} {\bf 9}, 890--906.

\bibitem[Molloy and Reed(1995)]{MR:1995}
Molloy, M. and Reed, B.~(1995) A critical point for random graphs with a given degree sequence.
{\sl Random Structures Algorithms} {\bf 6}, 161--179.

\bibitem[Nerman(1981)]{Nerman81}
Nerman, O.~(1981) On the convergence of supercritical general (C-M-J)
branching processes.
{\sl Z. Wahrscheinlichkeitsth} {\bf 57}, 365--395.

\bibitem[Newman(2002)]{Newman:2002}
Newman, M.~(2002) The spread of epidemic disease on networks. {\sl
Phys.~Rev.~E} {\bf 66}, 016128.

\bibitem[Newman {\it et al.}(2001)]{NSW:2001}
Newman, M.E.J., Strogratz, S.H. and Watts, D.J.~(2001)
Random graphs with arbitrary degree distributions and their applications.
{\sl Phys.~Rev.~E} {\bf 64}, 026118.


\bibitem[Pollett(1990)]{Pollett90}
Pollett P.K.~(1990) On a model for interference between searching
insect parasites.
{\sl J.~Austral.~Math.~Soc.~Ser.~B} {\bf 32}, 133--150.


\bibitem[Trapman(2007)]{Trapman:2007}
Trapman, P.~(2007) On analytical approaches to epidemics on networks.
{\sl Theor.~Popul.~Biol.} {\bf 71}, 160--173.

\bibitem[Volz(2008)]{Volz:2008}
Volz, E. M.~(2008) SIR dynamics in random networks with heterogeneous connectivity.
{\sl J.~Math.~Biol.} {\bf 56}, 293--310.

\bibitem[Watson(1980)]{Watson80}
Watson, R.~(1980) A useful random time-scale transformation for the standard epidemic model.
{\sl J.~Appl.~Prob.} {\bf 17}, 324--332.


\bibitem[Whittle(1955)]{Whittle55}
Whittle, P.~(1955) The outcome of a stochastic epidemic - a note on Bailey's paper. {\sl Biometrika} {\bf 42}, 116--122.


\end{thebibliography}
\end{document}